\documentclass[11pt]{article}
\usepackage{lastpage}
\usepackage{amssymb,amsmath,latexsym,amscd}
\usepackage{graphicx}
\usepackage[all]{xypic}
\usepackage{mathrsfs}
\usepackage{array,calc,url,tabularx, epic, tikz}
\usepackage{geometry}
\usepackage{makeidx}
\usepackage[all]{xy}
\usetikzlibrary{matrix}
\usetikzlibrary{arrows,positioning,decorations.pathmorphing,
  decorations.markings}

\newtheorem{theorem}{Theorem}[section]

\newtheorem{definition}[theorem]{Definition}
\newtheorem{remark}[theorem]{Remark}

\newtheorem{lemma}[theorem]{Lemma}
\newtheorem{corollary}[theorem]{Corollary}

\numberwithin{figure}{section}
\def\it{\textit}

\newenvironment{proof}{{\it{
Proof.}}}{\nopagebreak\mbox{}{\hfill$\square$}
\par\smallskip}

\newcommand{\rthm}[1]{Theorem~\ref{#1}}
\newcommand{\rlem}[1]{Lemma~\ref{#1}}

\newcommand{\rrem}[1]{Remark~\ref{#1}}

\def\e{\epsilon}
\def\a{\alpha}
\def\b{\beta}
\def\d{\delta}
\def\D{\Delta}
\def\gg{\gamma}
\def\ga{\gamma}

\def\G{\Gamma}
\def\s{\sigma}
\def\Si{\Sigma}
\def\t{\theta}
\def\l{\lambda}
\def\w{\omega}
\def\L{\Lambda}
\def\m{\mu}

\def\r{\rho}

\def\F{{\mathfrak{F}}}
\def\g{{\mathfrak{g}}}
\def\nn{{\mathfrak{n}}}
\def\h{{\mathfrak{h}}}
\def\bb{{\mathfrak{b}}}

\def\CC{{\mathcal{C}}}

\def\V{{\mathcal{V}}}
\def\Z{{\mathbb{Z}}}
\def\C{{\mathbb{C}}}
\def\O{{\mathcal{O}}}

\def\A{{\mathcal{A}}}

\def\F{{\mathcal{F}}}

\def\U{{\mathcal{U}}}
\def\S{\Sigma}
\def\F{{\mathcal{F}}}

\def\0{{\bar{0}}}
\def\1{{\bar{1}}}

\begin{document}

\title{The representation theory of the exceptional Lie superalgebras $F(4)$ and $G(3)$}
\date{}
\author{Lilit Martirosyan}

\maketitle
\begin{abstract}

This paper is a resolution of three related problems proposed by Yu. I. Manin and V. Kac for the so-called {\it{exceptional}} Lie superalgebras $F(4)$ and $G(3)$. The first problem posed by Kac (1978) is the problem of finding character and superdimension formulae for the simple modules. The second problem posed by Kac (1978) is the problem of classifying all indecomposable representations. The third problem posed by Manin (1981) is the problem of constructing the {\it{superanalogue}} of Borel-Weil-Bott theorem.
\end{abstract}

\section{Introduction}

After classifying all finite-dimensional simple Lie superalgebras over $\C$ in 1977, V. Kac proposed the problem of finding character and superdimension formulae for the simple modules (see \cite{K2}). The first main result in this thesis is solving this problem in full for the so-called {\it{exceptional}} Lie superalgebras $F(4)$ and $G(3)$.  

The next problem, also posed by V. Kac in 1977, is the problem of classifying all indecomposable representations of classical Lie superalgebras (see \cite{K2}). For the exceptional Lie superalgebras $F(4)$ and $G(3)$, we describe the {\it{blocks}} up to equivalence and find the corresponding {\it{quivers}}, which gives a full solution of this problem. We show that the blocks of atypicality 1 are {\it{tame}}, which together with Serganova's results for other Lie superalgebras proves a conjecture by J. Germoni. 

In the geometric representation theory of Lie algebras, the Borel-Weil-Bott (BWB) theorem (see \rthm{BWB}) plays a crucial role. This theorem describes how to construct families of representations from sheaf cohomology groups associated to certain vector bundles. It was shown by I. Penkov, that this theorem is not true for Lie superalgebras. In 1981, Yu. I. Manin proposed the problem of constructing a {\it{superanalogue}} of BWB theorem. The first steps towards the development of this theory were carried out by I. Penkov in \cite{P}. One of my results (see \rthm{BWBF} and \rthm{BWBG}) is an analogue of BWB theorem for the exceptional Lie superalgbras $F(4)$ and $G(3)$ for dominant weights.

\section{Main results}

\subsection{Classification of blocks}

Let $\CC$ denote the category of finite-dimensional $\g$-modules.  And let $\F$ be the full subcategory of $\CC$ consisting of modules such that the parity of any weight space coincides with the parity of the corresponding weight. 

The category $\F$ decomposes into direct sum of full subcategories called {\it{blocks}} $\F^\chi$, where $\F^\chi$ consists of all finite dimensional modules with (generalized) central character $\chi$. A block having more than one element is called an {\it{atypical}} block. By $F^\chi$, we denote the set of weights corresponding to central character $\chi$.

 A {\it{quiver diagram}} is a directed graph that has vertices the finite-dimensional irreducible representations of $\g$, and the number of arrows from vertex $\l$ to the vertex $\m$ is $dim{Ext^1_{\A}(L_\l,L_\m)}$.

\begin{theorem}\label{CBF} The following holds for $\g=F(4)$:

(1) The atypical blocks are parametrized by dominant weights $\m$ of $\mathfrak{sl}(3)$, such that $\m+\r_l=a\w_1+b\w_2$ with $a=3n+b$. Here, $b\in\mathbb{Z}_{>0}$ and $n\in\mathbb{Z}_{\geq 0}$; $\w_1$ and $\w_2$ are the fundamental weights of $\mathfrak{sl}(3)$; $\r_l$ is the Weyl vector for $\mathfrak{sl}(3)$.

(2) There are two, up to equivalence, atypical blocks, corresponding to dominant weights $\m$ of $\mathfrak{sl}(3)$, such that $\m+\r_l=a\w_1+b\w_2$ with $a=b$ or $a\neq b$. We call these blocks symmetric or non-symmetric and denote by $\F^{(a, a)}$ or $\F^{(a, b)}$ respectively.

(3) For the symmetric block $\F^{(a, a)}$, we have the following quiver diagram, which is of type $D_{\infty}$:

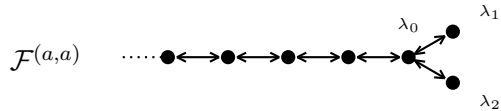
\begin{center}
  \begin{tikzpicture}[scale=.4]
    \draw (-2.5,0) node[anchor=east]  {$\F^{(a, a)}$};
    \foreach \x in {0,...,4}
    \draw[xshift=\x cm,thick,fill=black] (\x cm,0) circle (.2cm);
    \draw[xshift=8 cm,thick,fill=black] (30: 17 mm) circle (.2cm);
    \draw (10 cm,1.5 cm) node[anchor=west]  {\tiny $\l_1$};
    \draw[xshift=8 cm,thick,fill=black] (-30: 17 mm) circle (.2cm);
    \draw (10 cm,-1.5 cm) node[anchor=west]  {\tiny $\l_2$};
    \draw[dotted,thick] (-1.5 cm,0) -- +(1.4 cm,0);

    \draw (8 cm, 0.5 cm) node[anchor=south]  {\tiny $\l_0$};
    
    \foreach \y in {0.15,1.15,...,3.15}
    \draw[xshift=\y cm,thick] (\y cm,0) -- +(1.4 cm,0);
    \foreach \z in {0.15, 1.15, 2.15, 3.15}
    \draw[xshift=\z cm,thick] (\z cm,0) -- +(0.3 cm,0.15);
    \foreach \w in {0.15, 1.15, 2.15, 3.15}
    \draw[xshift=\w cm,thick] (\w cm,0) -- +(0.3 cm,-0.15);
    \foreach \u in {0.85, 1.85, 2.85, 3.85}
    \draw[xshift=\u cm,thick] (\u cm, 0) -- +(-0.3 cm,0.15);
    \foreach \v in {0.85, 1.85, 2.85, 3.85}
    \draw[xshift=\v cm,thick] (\v cm, 0) -- +(-0.3 cm,-0.15);
    \draw[xshift=8 cm,thick] (30: 3 mm) -- (30: 14 mm);
    \draw[xshift=8 cm,thick] (-30: 3 mm) -- (-30: 14 mm);
    \draw[thick] (8.2 cm,0.1) -- +(0.2 cm, 0.27);
    \draw[thick] (8.2 cm,0.1) -- +(0.35 cm, 0.05);
     \draw[thick] (8.2 cm,-0.1) -- +(0.2 cm, -0.27);
    \draw[thick] (8.2 cm,-0.1) -- +(0.35 cm, -0.05);
    \draw[thick] (9.2 cm, 7mm) -- +(-0.15 cm,-0.27);
    \draw[thick] (9.2 cm, 7mm) -- +(-0.35 cm,-0.03);
    \draw[thick] (9.2 cm, -7mm) -- +(-0.15 cm,0.27);
    \draw[thick] (9.2 cm, -7mm) -- +(-0.35 cm,0.03);
  \end{tikzpicture}
\end{center}


(4) For the non-symmetric block $\F^{(a, b)}$, we have the following quiver diagram, which is of type $A_{\infty}$: 


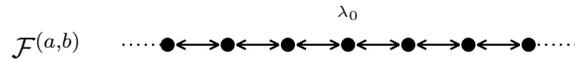
\begin{center}
  \begin{tikzpicture}[scale=.4]
    \draw (-2.5,0) node[anchor=east]  {$\F^{(a, b)}$};
    \foreach \x in {0,...,6}
    \draw[xshift=\x cm,thick,fill=black] (\x cm,0) circle (.2cm);
    
    \draw[dotted,thick] (-1.5 cm,0) -- +(1.4 cm,0);
    \draw[dotted,thick] (12.3,0) -- +(1.4 cm,0);
    \foreach \y in {0.15,1.15,...,3.15,4.15,5.15}
    \draw[xshift=\y cm,thick] (\y cm,0) -- +(1.4 cm,0);
    \foreach \z in {0.15, 1.15, 2.15, 3.15,4.15,5.15}
    \draw[xshift=\z cm,thick] (\z cm,0) -- +(0.3 cm,0.15);
    \foreach \w in {0.15, 1.15, 2.15, 3.15,4.15,5.15}
    \draw[xshift=\w cm,thick] (\w cm,0) -- +(0.3 cm,-0.15);
    \foreach \u in {0.85, 1.85, 2.85, 3.85,4.85,5.85}
    \draw[xshift=\u cm,thick] (\u cm, 0) -- +(-0.3 cm,0.15);
    \foreach \v in {0.85, 1.85, 2.85, 3.85,4.85,5.85}
    \draw[xshift=\v cm,thick] (\v cm, 0) -- +(-0.3 cm,-0.15);
    \draw (6 cm, 0.5 cm) node[anchor=south]  {\tiny $\l_0$};
    
  \end{tikzpicture}
\end{center}


\end{theorem}

\begin{theorem}\label{CBG} The following holds for $\g=G(3)$:

(1) The atypical blocks are parametrized by dominant weight $\m$ of $\mathfrak{sl}(2)$, such that $\m_l+\r=a\w_1$ with $a=2n+1$. Here, $n\in\mathbb{Z}_{\geq 0}$; $\w_1$ is the fundamental weight of $\mathfrak{sl}(2)$; $\r_l$ is the Weyl vector for $\mathfrak{sl}(2)$.

(2) There is one, up to equivalence, atypical block, corresponding to dominant weight $\m$ of $\mathfrak{sl}(2)$, such that $\m_l+\r=a\w_1$. Denote it by $\F^{a}$.

(3) For the block $\F^{a}$, we have the following quiver diagram, which is of type $D_{\infty}$:


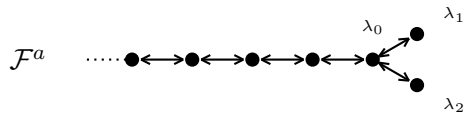
\begin{center}
  \begin{tikzpicture}[scale=.4]
    \draw (-2.5,0) node[anchor=east]  {$\F^{a}$};
    \foreach \x in {0,...,4}
    \draw[xshift=\x cm,thick,fill=black] (\x cm,0) circle (.2cm);
    \draw[xshift=8 cm,thick,fill=black] (30: 17 mm) circle (.2cm);
    \draw (10 cm,1.5 cm) node[anchor=west]  {\tiny $\l_1$};
    \draw[xshift=8 cm,thick,fill=black] (-30: 17 mm) circle (.2cm);
    \draw (10 cm,-1.5 cm) node[anchor=west]  {\tiny $\l_2$};
    \draw[dotted,thick] (-1.5 cm,0) -- +(1.4 cm,0);
        \draw (8 cm, 0.5 cm) node[anchor=south]  {\tiny $\l_0$};

    \foreach \y in {0.15,1.15,...,3.15}
    \draw[xshift=\y cm,thick] (\y cm,0) -- +(1.4 cm,0);
    \foreach \z in {0.15, 1.15, 2.15, 3.15}
    \draw[xshift=\z cm,thick] (\z cm,0) -- +(0.3 cm,0.15);
    \foreach \w in {0.15, 1.15, 2.15, 3.15}
    \draw[xshift=\w cm,thick] (\w cm,0) -- +(0.3 cm,-0.15);
    \foreach \u in {0.85, 1.85, 2.85, 3.85}
    \draw[xshift=\u cm,thick] (\u cm, 0) -- +(-0.3 cm,0.15);
    \foreach \v in {0.85, 1.85, 2.85, 3.85}
    \draw[xshift=\v cm,thick] (\v cm, 0) -- +(-0.3 cm,-0.15);
    \draw[xshift=8 cm,thick] (30: 3 mm) -- (30: 14 mm);
    \draw[xshift=8 cm,thick] (-30: 3 mm) -- (-30: 14 mm);
    \draw[thick] (8.2 cm,0.1) -- +(0.2 cm, 0.27);
    \draw[thick] (8.2 cm,0.1) -- +(0.35 cm, 0.05);
     \draw[thick] (8.2 cm,-0.1) -- +(0.2 cm, -0.27);
    \draw[thick] (8.2 cm,-0.1) -- +(0.35 cm, -0.05);
    \draw[thick] (9.2 cm, 7mm) -- +(-0.15 cm,-0.27);
    \draw[thick] (9.2 cm, 7mm) -- +(-0.35 cm,-0.03);
    \draw[thick] (9.2 cm, -7mm) -- +(-0.15 cm,0.27);
    \draw[thick] (9.2 cm, -7mm) -- +(-0.35 cm,0.03);
  \end{tikzpicture}
\end{center}


\end{theorem}

\subsection{Character and superdimension formulae}

The {\it{superdimension}} of a representation $V$ is the number $sdim{V}=dim{V_0}-dim{V_1}$ (see \cite{KW}).

Let $X=\{ x\in\g_{\1}| [x,x]=0\}$ be the {\it{self-commuting cone}} in $\g_{\1}$ studied in \cite{DuSe}. 
For $x\in X$, we denote by $\g_x$ the quotient $C_\g(x)/[x,\g]$ as in \cite{DuSe}, where $C_\g(x)=\{a\in\g | [a,x]=0\}$ is the centralizer of $x$ in $\g$.


We proved the following superdimension formula for the exceptional Lie superalgebra $\g=F(4)$.

\begin{theorem}\label{SFF} Let $\g=F(4)$. 
Let $\l\in F^{(a,b)}$ and $\m+\r_l=a\w_1+b\w_2$. If $\l \neq \l_1,\,\l_2$, the following superdimension formula holds:

$$sdim\,{L_\l}=\pm 2dim\,{L_\m(\g_x)}.$$

For the special weights, we have: $sdim\,{L_{\l_1}}=sdim\,{L_{\l_2}}=dim\,{L_\m(\g_x)}$. Here, $\g_x\cong\mathfrak{sl}(3)$. 
\end{theorem}

Similarly, we proved the following superdimension formula for the exceptional Lie superalgebra $\g=G(3)$.

\begin{theorem}\label{SFG} Let $\g=G(3)$. Let $\l\in F^{a}$ and $\m+\r_l=a\w_1$. If $\l\neq \l_1,\,\l_2$, the following superdimension formula holds:

$$sdim\,{L_\l}=\pm 2dim\,{L_\m(\g_x)}.$$

For the special weights, we have: $sdim\,{L_{\l_1}}=sdim\,{L_{\l_2}}=dim\,{L_\m(\g_x)}$. Here, $\g_x\cong\mathfrak{sl}(2)$. 

\end{theorem}

A root $\a$ is called {\it{isotropic}} if $(\a, \a)=0$. The {\it{degree of atypicality}} of the weight $\l$ the maximal number of mutually orthogonal linearly independent isotropic roots $\a$ such that $(\l+\r,\a)=0$. The {\it{defect}} of $\g$ is the maximal number of linearly independent mutually orthogonal isotropic roots. We use the above superdimension formulas and results in \cite{S3} to prove the following theorem, which is Kac-Wakimoto conjecture in \cite{KW} for $\g=F(4)$ and $G(3)$.  

\begin{theorem}\label{KWFG} Let $\g=F(4)$ or $G(3)$. The superdimension of a simple module of highest weight $\l$ is nonzero if and only if the degree of atypicality of the weight is equal to the defect of the Lie superalgebra. 
\end{theorem}

The following theorem gives a Weyl character type formula for the dominant weights. It was conjectured by Bernstein and Leites that formula \ref{BLM} works for all dominant weights. However, we obtain a different character formula \ref{BLSM} for the special weights $\l_1,\l_2$ for $F(4)$ and $G(3)$.

\begin{theorem}\label{BL} Let $\g=F(4)$ or $G(3)$. For a dominant weight $\l\neq \l_1,\l_2$, let $\a\in \D_\1$ be such that $(\l+\r,\a)=0$. Then

\begin{equation}\label{BLM}
ch L_{\l}=\frac{D_1\cdot e^{\r}}{D_0}\cdot\sum_{w\in W} sign(w)\cdot w(\frac{e^{\l+\r}}{(1+e^{-\a})}). 
\end{equation}

For the special weights $\l=\l_i$ with $i=1,2$, we have the following formula: 

\begin{equation}\label{BLSM}
ch( L_{\l})=\frac{D_1\cdot e^{\r}}{2D_0}\cdot\sum_{w\in W} sign{(w)}\cdot w(\frac{e^{\l+\r}(2+e^{-\a})}{(1+e^{-\a})}).
\end{equation}


\end{theorem}

\subsection{Analogue of Borel-Weil-Bott theorem for Lie superalgebras $F(4)$ and $G(3)$. }

Let $\g$ be a Lie (super)algebra with corresponding (super)group $G$. Let $\bb$ be the distinguished Borel subalgebra of $\g$ with corresponding (super)group $B$. Let $V$ be a $\bb$-module.

Denote by $\mathcal{V}$ the induced vector bundle $G\times_BV$ on the flag (super)variety $G/B$. The space of sections of $\mathcal{V}$ has a natural structure of a $\g$-module. The cohomology groups $H^i(G/B,\V^*)^*$ are $\g$-modules.

Let $C_\l$ denote the one dimensional representation of $B$ with weight $\l$. Denote by $\mathcal{O}_\l$ the line bundle $G\times_B C_\l$ on the flag (super)variety $G/B$. Let $L_\l$ denote the simple module with highest weight $\l$. See \cite{Ser}.

The classical result in geometric representation theory for finite-dimensional semisimple Lie algebra $\g$ states:

\begin{theorem}\label{BWB} (Borel-Weil-Bott) If $\l+\r$ is singular, where $\l$ is integral weight and $\r$ is the half trace of $\bb$ on its nilradical, then all cohomology groups vanish. If $\l+\r$ is regular, then there is a unique Weyl group element $w$ such that the weight $w(\l+\r)-\r$ is dominant and the cohomology groups $H^i(G/B, \mathcal{O}_{\l}^*)^*$ are non-zero in only degree $l=\text{length}(w)$ and in that degree they are equal to the simple module $L_{\m}$ with highest weight $\m=w(\l+\r)-\r$. 
\end{theorem}

We proved the following superanalogue for the exceptional Lie superalgebra $\g=F(4)$ for the dominant weights and for specific choice of $B$.

\begin{theorem}\label{BWBF} Let $\g=F(4)$.

(1) For $\m\in \F^{(a, a)}$ with $\m\neq \l_1$, $\l_2$ or $\l_0$, the group $H^0(G/B, \mathcal{O}_{\m}^*)^*$ has two simple subquotients $L_\m$ and $L_{\m'}$, where $\m'$ is the adjacent vertex to $\m$ in the quiver $D_\infty$ in the direction towards $\l_0$.

At the branching point $\l_0$ of the quiver, the group $H^0(G/B, \mathcal{O}_{\l_0}^*)^*$ has three simple subquotients $L_{\l_0}$, $L_{\l_1}$, and $L_{\l_2}$. For $i=1,2$, we have $H^0(G/B, \mathcal{O}_{\l_i}^*)^*=L_{\l_i}$.

The first cohomology is not zero only at the endpoints $\l_1$ and $\l_2$ of the quiver and $H^1(G/B, \mathcal{O}_{\l_1}^*)^*=L_{\l_2}$, $H^1(G/B, \mathcal{O}_{\l_2}^*)^*=L_{\l_1}$. All other cohomologies vanish.

(2) For $\m\in \F^{(a, b)}$, the group $H^0(G/B, \mathcal{O}_{\m}^*)^*$ has two simple subquotients $L_\m$ and $L_{\m'}$, where $\m'$ is the adjacent vertex to $\m$ in the quiver $A_\infty$ in the direction towards $\l_0$.

The first cohomology is not zero only in one particular point $\l_0$ of the quiver and 

$H^1(G/B, \mathcal{O}_{\l_0}^*)^*=L_{\l_0}$. Also, $H^0(G/B, \mathcal{O}_{\l_0}^*)^*=L_{\l_0}$. All other cohomologies vanish.
\end{theorem}

Similarly, we proved the following superanalogue of BWB theorem for the exceptional Lie superalgebra $\g=G(3)$ for the dominant weights.

\begin{theorem}\label{BWBG} Let $\g=G(3)$.

For $\m\in \F^{a}$ with $\m\neq \l_1$, $\l_2$ or $\l_0$, the group $H^0(G/B, \mathcal{O}_{\m}^*)^*$ has two simple subquotients $L_\m$ and $L_{\m'}$, where $\m'$ is the adjacent vertex to $\m$ in the quiver $D_\infty$ in the direction towards $\l_0$.

At the brunching point, the group $H^0(G/B, \mathcal{O}_{\l_0}^*)^*$ has three simple subquotients $L_{\l_0}$, $L_{\l_1}$, and $L_{\l_2}$. For $i=1,2$, we have $H^0(G/B, \mathcal{O}_{\l_i}^*)^*=L_{\l_i}$.

The first cohomology is not zero only at the endpoints of the quiver and $H^1(G/B, \mathcal{O}_{\l_1}^*)^*=L_{\l_2}$, $H^1(G/B, \mathcal{O}_{\l_2}^*)^*=L_{\l_1}$. All other cohomologies vanish.

\end{theorem}

\subsection{Germoni's conjecture and the indecomposable modules}

The following theorem together with results in \cite{GrSe} for other Lie superalgebras proves a conjecture by J. Germoni (\rthm{germoni1}). 

\begin{theorem}\label{tame} Let $\g=F(4)$ or $G(3)$. The blocks of atypicality 1 are {\it{tame}}.
\end{theorem}


\begin{theorem}\label{germoni1} Let $\g$ be a basic classical Lie superalgebra. Then all tame blocks are of atypicality less or equal 1. 
\end{theorem}

The following theorem together with \rthm{quivertheorem} gives a description of the indecomposable modules.

\begin{theorem}\label{relations} The quivers $A_{\infty}$ and $D_{\infty}$ are the ext-quiver for atypical blocks $\F^{(a,b)}$ and $\F^{(a,a)}$ of $F(4)$ and the quiver $D_{\infty}$ is the ext-quiver for atypical block $\F^{a}$ of $G(3)$ with the following relations:

For $\F^{(a,b)}$, we have:

$$d^+d^-+d^-d^+=(d^+)^2=(d^-)^2=0\text{ ,where } d^{\pm}=\sum_{l\in\Z}d^{\pm}_l$$

For $\F^{(a,a)}$ or $\F^{a}$ we have the following relations: 



$$d_l^-d_{l+1}^-=d_{l+1}^+d_l^+=0, \text{ for } l\geq 3$$

$$d_1^-d_2^+=d_2^-d_1^+=d_0^+d_2^+=d_2^-d_0^-=d_0^- d_3^-=d_3^+d_0^+=d_1^-d_0^-=d_0^+d_1^+=0$$

$$d_l^-d_l^+=d_{l+1}^+d_{l+1}^- \text{ for } l\geq 3$$

$$d_1^+d_1^-=d_2^+d_2^-=d_0^-d_0^+.$$

\end{theorem}

\section{Preliminaries} 

\subsection{Lie superalgebras}

The following preliminaries are taken from \cite{K}.

All vector spaces are over an algebraically closed field $k$ of characteristic $0$. A {\it{superspace}} over $k$ is a $\Z_2$-graded vector space $V=V_{\0}\oplus V_{\1}$.  Let $p(a)$ denote the degree of a homogeneous element $a$ and call $a$ {\it{even}} or {\it{odd}} if $p(a)$ is $0$ or $1$ respectively.

A {\it{Lie superalgebra}} is a superspace $\g=\g_{\0}\oplus\g_{\1}$, with a bilinear map $[\ ,\ ]:\g\otimes\g\rightarrow\g$, satisfying the following axioms for all homogenous $a,b,c\in\g$:

(a) $[a,b]=-(-1)^{p(a)p(b)}[b,a]$ ({\it{anticommutatvity}});

(b) $[a,[b,c]]=[[a,b],c]+(-1)^{p(a)p(b)}[b,[a,c]]$ ({\it{Jacobi identity}}).

All simple finite-dimensional Lie superalgebras have been classified by V.  Kac in \cite{K1}. The {\it{basic classical}} Lie superalgebras are all the simple Lie algebras, $A(m,n)$, $B(m,n)$, $C(n)$, $D(m,n)$, $D(2,1;\a)$, $F(4)$, and $G(3)$. The Lie superalgebras  $F(4)$, $G(3)$ are called {\it{exceptional}}.

Let $\h$ be Cartan subalgebra of $\g_{\0}$ Lie algebra, then $\g$ had a weight decomposition $\g=\bigoplus_{\a\in\h^*}\g_{\a}$, with $\g_{\a}=\{x\in\g\ |\ [h,x]=\a(h)x\ \forall h\in\h\}$. The set $\D=\{\a\in\h^*\ |\ \g_{\a}\neq 0\}$ is called the {\it{set of roots}} of $\g$ and $\g_{\a}$ is the {\it{root space corresponding}} to root $\a\in\D$. For a regular $h\in\h$, i.e. $Re\,\a(h)\neq 0$ $\forall\a\in\D$, we have a decomposition $\D=\D^+\cup\D^-$. Here, $\D^+=\{\a\in\D|Re\,\a(h)>0\}$ is called the {\it{set of positive roots}} and $\D^-=\{\a\in\D|Re\,\a(h)<0\}$ is called the {\it{set of negative roots}}. 

The Lie superalgebra $\g$ admits a triangular decomposition $\g=\nn^-\oplus\h\oplus\nn^+$ with $\nn^{\pm}_{\a}=\bigoplus_{\a\in\D^{\pm}}\g_{\a}$, with $\nn^{\pm}$ nilpotent subalgebras of $\g$. Then $\bb=\h\oplus\nn^+$ is a solvable Lie subsuperalgebra of $\g$, which is called the {\it{Borel subsuperalgebra}} of $\g$ with respect to the given triangular decomposition. Here, $\nn^+$ is an ideal of $\bb$.

We set $\D_0^{\pm}=\{\a\in \D^{\pm}|\g_{\a}\subset \g_{\0}\}$ and $\D_1^{\pm}=\{\a\in \D^{\pm}|\g_{\a}\subset \g_{\1}\}$. Then the set $\D^+_0\cup\D^-_0$ called the set of {\it{even roots}} and the set $\D^+_1\cup\D^-_1$ is called the set of {\it{odd roots}}. 

The universal enveloping algebra of $\g$ is defined to be the quotient $U(\g)=T(\g)/R$, where $T(\g)$ is the tensor superalgebra over space $\g$ with induced $\Z_2$-gradation and $R$ is the ideal of $T(\g)$ generated by the elements of the form $[a,b]-ab+(-1)^{p(a)p(b)}ba$.

The following is the Lie superalgebra analogue of Poincar�-Birkhoff-Witt (PBW) theorem: Let $\g=\g_{\0}\oplus\g_{\1}$ be a Lie superalgebra, $a_1,\dots, a_m$ be a basis of $\g_{\0}$ and $b_1,\dots, b_n$ be a basis of $\g_{\1}$, then the elements of the form $a_1^{k_1}\cdots a_m^{k_m} b_{i_1}\cdots b_{i_s}$ with $k_i\geq 0$ and $1\leq i_1 < \cdots < i_s\leq n$ form a basis of $U(\g)$.

By PBW theorem, $U(\g)=U(\nn^-)\otimes U(\h)\otimes U(\nn^+)$. Let $\t: U(\g)\rightarrow U(\h)$ be the projection with kernel $\nn^-U(\g)\otimes U(\g)\nn^+$. Let $Z(\g)$ to be the center of $U(\g)$. Then the restriction $\t_{|Z(\g)}: Z(\g)\rightarrow U(\h)\cong S(\h)$ is a homomorphism of rings called {\it{Harish-Chandra}} map. Since $\h$ is abelian, $S(\h)$ can be considered the algebra of polynomial functions of $\h^*$. The (generalized) {\it{central character}} is a map $\chi_{\l}: Z(\g)\rightarrow k$ such that $\chi_{\l}(z)=\t(z)(\l)$.

\subsection{Weyl group and odd reflections}

The {\it{Weyl group}} $W$ of Lie superalgebra $\g=\g_{\0}\oplus\g_{\1}$ is the Weyl group of the Lie algebra $\g_{\0}$. Weyl group is generated by {\it{even reflections}}, which are reflections with respect to even roots of $\g$. Define parity $\w$ on $W$, such that $\forall r\in W$, $\w(r)=1$ if $\w$ can be written as a product of even number of reflections and $\w(r)=-1$ otherwise. 

A linearly independent set of roots $\Si$ of a Lie superalgebra is called a {\it{base}} if for each $\b\in\Si$ there are $X_{\b}\in\g_{\b}$ and $Y_{\b}\in\g_{-\b}$ such that $X_{\b}$, $Y_{\b}$, $\b\in\Si$ and $\h$ generate $\g$ and for any distinct $\b,\gg \in \Si$ we have $[X_{\b},Y_{\gg}]=0$.

Let $\{X_{\b_i}\}$ be a base and let $h_{\b_i}=[X_{\b_i},Y_{\b_i}]$. We have the following relations $[h,X_{\b_i}]=\b_i(h)X_{\b_i}$, $[h,Y_{\b_i}]=-\b_i(h)Y_{\b_i}$, and $[X_{\b_i},Y_{\b_j}]=\d_{ij}h_{\b_i}$. We define the {\it{Cartan matrix}} of a base $\Si$ to be matrix $A_{\Si}=(\b_i(h_{\b_j}))=(a_{\b_i\b_j})$. A base where the number of odd roots is minimal is called a {\it{distinguished}} root base and the corresponding Cartan matrix is called {\it{distinguished}} Cartan matrix.

In the given base $\Si$,  let $\a\in\Si$ is such that $a_{\a\a}=0$ and $p(\a)=1$. An {\it{odd reflection}} $r_{\a}$ is defined in \cite{S1} by:

 $r_{\a}(\a)=-\a$, $r_{\a}(\b)=\b$ if $\a\neq\b$ and $a_{\a\b}=a_{\b\a}=0$, and 
 
 $r_{\a}(\b)=\b+\a$ if $a_{\a\b}\neq0$ and $a_{\b\a}\neq0$, for all $\b\in\Si$.

We call a root $\a\in\Si$ {\it{isotropic}}, if $a_{\a\a}=0$; otherwise, it is called {\it{non-isotropic}}. We will need the following lemma: 


\begin{lemma}(Serganova, \cite{S1})\label{oddrootbase} Let $\Pi$ and $\Pi'$ be two bases, and $\D^+(\Pi)$,  $\D^+(\Pi')$ be the corresponding sets of positive roots. If $\Pi'=r_\a(\Pi)$, for some root $\a\in\Pi$. Then 

$$\D^+(\Pi')=\D^+(\Pi)\cap \{ -\a \}\setminus\{\a\}, $$
or
$$\D^+(\Pi')=\D^+(\Pi)\cap \{ -\a, -2\a \}\setminus\{\a,2\a \}, $$

depending on whether $2\a$ is a root. 
\end{lemma}

\subsection{Representations of Lie superalgebras}


The following definitions and results can be found in \cite{Musson}. 

A {\it{linear representation}} $\r$ of $\g=\g_{\0}\oplus\g_{\1}$ is a superspace $V=V_0\oplus V_1$, such that the graded action of $\g$ on $V$ preserves parity, i.e. $\g_{\bar{i}}(V_j)\subset V_{i+j}$ for $i,j\in \Z_2$ and $[g_1,g_2]v=g_1(g_2(v))-(-1)^{p(g_1)p(g_2)}g_2(g_1(v))$, where $g(v):=\r(g)(v)$. Then, we call $V$ a {\it{$\g$-module}}.

A $\g$-module is a {\it{weight module}}, if $\h$ acts semisimply on $V$. Then we can write $V=\bigoplus_{\m\in\h^*}V_{\m}$, where $V_{\m}=\{ m\in V\ |\ hm=\m(h)m, \forall h\in\h\}$. The elements of $P(V)=\{\m\in\h^*\ |\ V_{\m}\neq 0\}$ are called {\it{weights}} of $V$. 

For a fixed Cartan subalgebra $\h$ of $\g$, we fix $\bb$ to be a Borel containing $\h$. We have $\bb=\h\oplus\nn^+$. Let $\l\in\h^*$, we define one dimensional even $\bb$-module $C_\l=\langle v_{\l}\rangle$ by letting $h(v_{\l})=\l(h)v_{\l}$, $\forall h\in\h$ and $\nn^+(v_{\l})=0$ with $deg(v_{\l})=\0$. We define the {\it{Verma module}} with highest weight $\l$ as the induced module $$M_\l=Ind^{\g}_{\bb}C_\l=U(\g)\otimes_{U(\bb)}C_\l.$$

The $\g$-module $M_\l$ has a unique maximal submodule $I_\l$. The module $L_\l=M_\l/I_\l$ is called an {\it{irreducible representation with highest weight $\l$}}. It is proven in \cite{K2} that $L_{\l_1}$ and $L_{\l_2}$ are isomorphic iff $\l_1=\l_2$ and that any finite dimensional irreducible representation of $\g$ is one of $L_\l$. 

A weight $\l\in\h^*$ is called {\it{integral}} if $a_i\in\Z$ for all $i\neq s$, where s corresponds to an odd isotropic root in a distinguished base. 

We denote by $\Lambda$ the lattice of integral weights. It is the same as the weight lattice of the $\g_{\0}$. The root lattice will be a sublattice of $\Lambda$ and is denoted by $Q$. We know that any simple finite-dimensional $\g$-module that is simisimple over $\mathfrak{h}$ and has weights in $\Lambda$, is a quotient of the Verma module with highest weight $\l\in\Lambda$ by a maximal submodule. $\l$ is called {\it{dominant}} if this quotient is finite-dimensional. 

Thus, for every dominant weight, there are two simple modules, that can be obtained from each other by change of parity. In order to avoid "parity chasing", the {\it{parity function}} is defined $p:\Lambda\rightarrow\mathbb{Z}_2$, such that $p(\l+\a)=p(\l)+p(\a)$ for all $\a\in\D$ and extend it linearly to all weights.

For a $\g$-module $V$, there is a functor $\pi$ such that $\pi(V)$ is the module with shifted parity, i.e. $\pi(V)_0=V_1$ and $\pi(V)_1=V_0$. We have $\mathcal{C}=\F\oplus\pi(\F)$, where $\mathcal{C}$ is the category of finite-dimensional representations of $\g$ and $\F$ is the full subcategory of $\CC$ consisting of modules such that the parity of any weight space coincides with the parity of the corresponding weight.



The {\it{Dynkin labels}} of a linear function $\l\in\h^*$ are defined by $a_s=(\l,\a_s)$, if $\a_s$ is an odd isotropic root in a distinguished base and $a_i=\frac{2(\l,\a_i)}{(\a_i,\a_i)}$ for other roots in the distinguished base. 

The following result from \cite{K1} is analogous to the theorem on the highest weights for finite-dimensional irreducible representations of Lie algebras. We state it only in the case $\g=F(4)$ or $G(3)$:

\begin{lemma} (Kac,\cite{K1})\label{Kac} For a distinguished Borel subalgebra of $\g=F(4)$ or $G(3)$, let $e_i$, $f_i$, $h_i$ be standard generators of $\g$. Let $\l\in\g^*$ and $a_i=\l(h_i)$. Then the representation $L_\l$ is finite dimensional if and only if the following conditions are satisfied:

For $\g=F(4)$,

1) $a_i\in \Z_+$; 

2) $k=\frac{1}{3}(2a_1-3a_2-4a_3-2a_4)\in \Z_+$;

3) $k<4$: $a_i=0$ for all $i$ if $k=0$; $k\neq 0$; $a_2=a_4=0$ for $k=2$; $a_2=2a_4+1$ for $k=3$.

For $\g=G(3)$,

1) $a_i\in \Z_+$; 

2) $k=\frac{1}{2}(a_1-2a_2-3a_3)\in\Z_+$ for $\g=G(3)$;

3) $k<3$: $a_i=0$ for all $i$ if $k=0$; $k\neq 0$; $a_2=0$ for $k=2$. 
\end{lemma}

For a base $\S$, we denote $L^\S_\l$ the simple $\g$-module with highest weight $\l$ corresponding to the triangular decomposition obtained from $\S$. 


\begin{lemma} (Serganova, \cite{S1})\label{dominant weights} A weight $\l$ is dominant integral if and only if for any base $\S$ obtained from $\Pi$ by a sequence of odd reflections, and for any $\b\in\S$ such that $\b(h_\b)=2$, we have $\l'(h_\b)\in\mathbb{Z}_{\geq 0}$ if $\b$ is even and $\l'(h_\b)\in2\mathbb{Z}_{\geq 0}$ if $\b$ is odd. Here, $L^\S_{\l'}\cong L^{\Pi}_\l$. 
\end{lemma}

An irreducible finite-dimensional representation of $\g$ is called {\it{typical}} if it splits as a direct summand in any finite dimensional representation of $\g$. It is known that $\l$ is a highest weight of a typical representation if $(\l+\r,\a)\neq 0$ for any isotropic $\a\in\D^+_1$. 

\section{Structure and blocks for the exceptional Lie superalgebras $F(4)$ and $G(3)$}

\subsection{Description of $F(4)$} 

Let $\g$ be the exceptional Lie superalgebra $F(4)$. In this section, we descrive some structure of $\g$ that can be found in \cite{K1}. 

The Lie superalgebra $\g=F(4)$ has dimension 40 and rank 4. The even part $\g_\0$ is $B_3\oplus A_1=\mathfrak{o}(7)\oplus \mathfrak{sl}(2)$ and the odd part $\g_\1$ is isomorphic to $\mathfrak{spin}_7\otimes \mathfrak{sl}_2$ as a $\g_\0$-module. Here $\mathfrak{spin}_7$ is the eight dimensional spinor representation of $\mathfrak{so}(7)$ and $\mathfrak{sl}_2$ is the two dimensional representation of $\mathfrak{sl}(2)$. The even part $\g_{\0}$ has dimension 24. The odd part $\g_{\1}$ has dimension 16.

Its root system can be written in the space $\mathfrak{h}^*=\mathbb{C}^4$ in terms of the basis vectors $\{ \epsilon_1, \epsilon_2, \epsilon_3, \delta \}$ that satisfy the relations: $$(\epsilon_i,\epsilon_j)=\delta_{ij}\text{ , }(\delta,\delta)=-3\text{ , }(\epsilon_i,\delta)=0\text{ for all }i,j.$$ 

With respect to this basis, the root system $\D=\D_\0\oplus\D_\1$ is given by $$\Delta_\0=\{\pm\epsilon_i\pm\epsilon_j; \pm\epsilon_i; \pm\delta\}_{i\neq j}\text{ and }\Delta_\1=\{\frac{1}{2}(\pm\epsilon_1\pm\epsilon_2\pm\epsilon_3\pm\delta)\}.$$

For $F(4)$, the isotropic roots are all odd roots. 

We choose the simple roots to be $\Pi=\{\alpha_1=\frac{1}{2}(-\epsilon_1-\epsilon_2-\epsilon_3+\delta);\alpha_2=\epsilon_3;\alpha_3=\epsilon_2-\epsilon_3;\alpha_4=\epsilon_1-\epsilon_2\}$. This will correspond to the following Dynkin diagram and Cartan matrix:

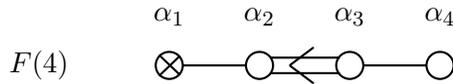
\begin{center}
  \begin{tikzpicture}[scale=.4][->,>=stealth',shorten >=1pt,auto,node distance=3cm,
  thick,main node/.style={circle,fill=blue!20,draw}]
  
    \draw (-6,0) node[anchor=east]  {$F(4)$};
    \draw (-3,1) node[above]  {$\a_1$};
    \draw (0,1) node[above]  {$\a_2$};
    \draw (3,1) node[above]  {$\a_3$};
    \draw (6,1) node[above]  {$\a_4$};

    \draw[thick] (-3 cm ,0) circle (.45 cm);
    \draw[thick] (0 ,0) circle (.45 cm);
    \draw[thick] (3 cm,0) circle (.45 cm);
    \draw[thick] (6 cm,0) circle (.45 cm);
    \draw[thick] (35: 4.5mm) -- +(2.3 cm, 0); {second above}
    \draw[xshift=-3.15 cm,thick] (0: 6 mm) -- +(2.1 cm, 0); {first line}
    \draw[thick] (-35: 4.5mm) -- +(2.3 cm, 0); {second below}
    \draw[xshift=3.15 cm,thick] (0: 3 mm) -- +(2.1 cm, 0); {last line}
    \draw[thick] (1, 0mm) -- +(0.8 cm, 0.6); {arrow above}
    \draw[thick] (1, 0mm) -- +(0.8 cm, -0.6); {arrow below}
    \draw[thick] (-3.34, -3.5mm) -- +(0.7, 6.7mm); {cross}
    \draw[thick] (-3.34, 3.5mm) -- +(0.7, -6.7mm); {cross}
  \end{tikzpicture}

\[
 \text{Cartan matrix}=A_{\Pi}=
 \begin{pmatrix}
  0 & 1 & 0 & 0 \\
  -1 & 2 & -2 & 0 \\
  0  & -1  & 2 & -1  \\
  0 & 0 & -1 & 2
 \end{pmatrix}
\]

\end{center}

Up to $W$-equivalence, we have the following six simple root systems for $F(4)$ with $\Sigma$ being the standard basis. 

$\Sigma=\Pi=\{\alpha_1=\frac{1}{2}(-\epsilon_1-\epsilon_2-\epsilon_3+\delta);\alpha_2=\epsilon_3;\alpha_3=\epsilon_2-\epsilon_3;\alpha_4=\epsilon_1-\epsilon_2\}$;

$\Sigma'=\{\alpha_1=\frac{1}{2}(\epsilon_1+\epsilon_2+\epsilon_3-\delta);\alpha_2=\frac{1}{2}(-\epsilon_1-\epsilon_2+\epsilon_3+\delta);\alpha_3=\epsilon_2-\epsilon_3;\alpha_4=\epsilon_1-\epsilon_2\}$;

$\Sigma''=\{\alpha_1=\epsilon_3;\alpha_2=\frac{1}{2}(\epsilon_1+\epsilon_2-\epsilon_3-\delta);\alpha_3=\frac{1}{2}(-\epsilon_1+\epsilon_2-\epsilon_3+\delta);\alpha_4=\epsilon_1-\epsilon_2\}$;

$\Sigma'''=\{\alpha_1=\frac{1}{2}(-\epsilon_1+\epsilon_2+\epsilon_3+\delta);\alpha_2=\epsilon_2-\epsilon_3;\alpha_3=\frac{1}{2}(\epsilon_1-\epsilon_2+\epsilon_3-\delta);\alpha_4=\frac{1}{2}(\epsilon_1-\epsilon_2-\epsilon_3+\delta)\}$;

$\Sigma^{(4)}=\{\alpha_1=\frac{1}{2}(\epsilon_1-\epsilon_2-\epsilon_3-\delta);\alpha_2=\epsilon_2-\epsilon_3;\alpha_3=\epsilon_3;\alpha_4=\delta\}$;

$\Sigma^{(5)}=\{\alpha_1=\delta;\alpha_2=\epsilon_2-\epsilon_3;\alpha_3=\epsilon_1-\epsilon_2;\alpha_4=\frac{1}{2}(-\epsilon_1+\epsilon_2+\epsilon_3-\delta)\}$.

The following odd roots will be used later:  

$\b=\frac{1}{2}(-\epsilon_1-\epsilon_2-\epsilon_3+\delta$;

$\b'=\frac{1}{2}(-\epsilon_1-\epsilon_2+\epsilon_3+\delta)$;

$\b''=\frac{1}{2}(-\epsilon_1+\epsilon_2-\epsilon_3+\delta)$;

$\b'''=\frac{1}{2}(-\epsilon_1+\epsilon_2+\epsilon_3+\delta)$;

$\b^{(4)}=\d$.

With respect to the root system $\Sigma$, the positive roots are $\Delta^+=\Delta^+_0\cup\Delta^+_1$, where $\Delta^+_0=\{\delta, \epsilon_i, \epsilon_i\pm\epsilon_j\ | i<j\}$ and $\Delta^+_1=\{\frac{1}{2}(\pm\epsilon_1\pm\epsilon_2\pm\epsilon_3+\delta)\}$. The Weyl vector is $\rho=\rho_0-\rho_1=\frac{1}{2}(5\epsilon_1+3\epsilon_2+\epsilon_3-3\delta)$, where $\rho_0=\frac{1}{2}(5\epsilon_1+3\epsilon_2+\epsilon_3+\delta)$ and $\rho_1=2\delta$.

The integral weight lattice, which is spanned by fundamental weights $\l_1=\e_1$, $\l_2=\e_1+\e_2$, $\l_3=\frac{1}{2}(\e_1+\e_2+\e_3)$, and $\l_4=\frac{1}{2}\d$ of $\g_\0$, is $\L=\frac{1}{2}\mathbb{Z}(\e_1+\e_2+\e_3)\oplus\mathbb{Z}\e_1\oplus\mathbb{Z}\e_2\oplus\frac{1}{2}\mathbb{Z}\d$. Also, $\L/Q\cong\Z_2$, where $Q$ is the root lattice. We can define parity function on $\L$, by setting  $p(\frac{\e_i}{2})=0$ and $p(\frac{\d}{2})=1$.

The Weyl groups $W$ is generated by six reflections that can be defined on basis vectors as follows: for an arbitrary permutation $(ijk)\in S_3$, we get three possible permutations $\s_i(e_i)=e_i$, $\s_i(e_j)=e_k$, $\s_i(e_k)=e_j$, and other three defined $\tau_i(e_i)=-e_i$, $\tau_i(e_j)=e_j$, $\tau_i(e_k)=e_k$ all six fixing $\d$, also one permutation $\s(e_i)=e_i$ for all $i$ and $s(\d)=-\d$. The Weyl group in this case is $W=((\Z/2\Z)^3\rtimes S_3)\oplus \Z/2\Z$.

\begin{lemma}\label{dominant integral weights of F(4)} A weight $\l=a_1\e_1+a_2\e_2+a_3\e_3+a_4\d\in X^+$ is dominant integral weight of $\g$ if and only if $\l+\rho\in \{(b_1,b_2,b_3|b_4)\in
\frac{1}{2}\mathbb{Z}\times\frac{1}{2}\mathbb{Z}\times\frac{1}{2}\mathbb{Z}\times\frac{1}{2}
\mathbb{Z}\ |\ b_1>b_2>b_3>0;\ b_4\geq-\frac{1}{2};\ b_1-b_2\in\Z_{>0}; b_2-b_3\in\Z_{>0};\ b_4=-\frac{1}{2}\ \implies b_1=b_2+1\ \&\ b_3=\frac{1}{2};\ b_4=0\ \implies b_1-b_2-b_3=0\}$. 
\end{lemma}

\begin{proof} 
Let $\l=a_1\e_1+a_2\e_2+a_3\e_3+a_4\d\in X^+$. Since the even roots in $\Pi$ are $\b=\e_3$, $\e_2-\e_3$, $\e_1-\e_2$. The relations $\l(\b)\in\Z_{\geq 0}$ imply $a_1\geq a_2 \geq a_3\geq 0$ or equivalently $b_1>b_2>b_3>0$ and $b_1-b_2\in\Z_{>0}$, $ b_2-b_3\in\Z_{>0}$. 

Using \rlem{dominant weights}, we apply odd reflections with respect to odd roots $\b, \b',\b'',\b'''$ to $\l$ we obtain conditions on $a_4$ or equivalently on $b_4$.

The following are the only possibilities:

(1) If $\l(\b)\neq 0$, $\l'(\b')\neq 0$, $\l''(\b'')\neq 0$, $\l'''(\b''')\neq 0$, $\l^{(4)}(\d)=2a_4-4\in\Z_{\geq 0}$, then $a_4\geq 2$ or $a_4\in\frac{1}{2}\Z_{\geq 0}$. Or, $b_4\geq \frac{1}{2}$.

(2) If $\l(\b)\neq 0$, $\l'(\b')\neq 0$, $\l''(\b'')\neq 0$, $\l'''(\b''')= 0$, $\l^4=\l'''$ and $\l^{(4)}(\d)=2a_4-3\in\Z_{\geq 0}$, implying $a_4\geq \frac{3}{2}$ and $a_4\in\frac{1}{2}\Z_{\geq 0}$. Only $a_4=\frac{3}{2}$ is possible and we have $a_1-a_2-a_3=-\frac{1}{2}$. Or, $b_4=0$ and $b_1-b_2-b_3=0$.

(3) $\l(\b)\neq 0$, $\l'(\b')\neq 0$, $\l''(\b'')=0$ and $\l'''=\l''$, $\l'''(\b''')= 0$, $\l^4=\l'''$ and $\l^{(4)}(\d)=2a_4-3\in\Z_{\geq 0}$, implying $a_4\geq 1$ and $a_4\in\frac{1}{2}\Z_{\geq 0}$. Only $a_4=1$ is possible and we have $a_1=a_2$ and $a_3=0$. Or, $b_4=-\frac{1}{2}$ and $b_1=b_2+1$, $b_3=\frac{1}{2}$. 
 
\end{proof}

\subsection{Description of $G(3)$} 

Let $\g$ be the exceptional Lie superalgebra $G(3)$. We describe some structure for $G(3)$ that has been studied in \cite{K1}. 

The Lie superalgebra $\g=G(3)$ is a 31-dimensional exceptional Lie superalgebra of defect 1. We have $\g_{\bar{0}}=G_2\oplus A_1$, where $G_2$ is the exceptional Lie algebra, and an irreducible $\g_{\0}$-module $\mathfrak{g}_{\1}$ that is isomorphic to $\mathfrak{g}_2\otimes \mathfrak{sl}_2$, where  $\mathfrak{g}_2$ is the seven dimensional representation of $G_2$ and $\mathfrak{sl}_2$ is the two dimensional representation of $\mathfrak{sl}_(2)$. The $\g_{\0}$ has dimension 17 and rank 3. And $\g_{\1}$ has dimension 14. 

We can realize its root system in the space $\mathfrak{h}^*=\mathbb{C}^3$ endowed with basis $\{\e_1, \e_2, \e_3, \d \}$ with $\e_1+\e_2+\e_3=0$ and with the bilinear form defined by: $$(\e_1,\e_1)=(\e_2,\e_2)=-2(\e_1,\e_2)=-(\d,\d)=2.$$

With respect to the above basis, the root system $\D=\D_\0\oplus\D_\1$ is given by $$\Delta_\0=\{ \pm\e_i; \pm2\d; \epsilon_i-\epsilon_j\}_{i\neq j}\text{ and }\Delta_\1=\{ \pm\d; \pm\e_i\pm\d\}.$$ 

Up to $W$ equivalence, there is are five systems of simple roots for $G(3)$ given by:

$\Pi=\{\a_1=\e_3+\d; \a_2=\e_1; \a_3=\e_2-\e_1 \}$,

$\Pi'=\{-\e_3-\d; -\e_2+\d; \e_2-\e_1\}$,

$\Pi''=\{\e_1; \e_2-\d; -\e_1+\d\}$,

$\Pi'''=\{\d;  \e_1-\d; \e_2-\e_1 \}$.






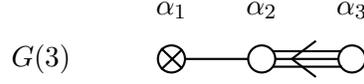
\begin{center}
  \begin{tikzpicture}[scale=.4]
  
    \draw (-6,0) node[anchor=east]  {$G(3)$};
    \draw (-3,1) node[above]  {$\a_1$};
    \draw (0,1) node[above]  {$\a_2$};
    \draw (3,1) node[above]  {$\a_3$};

    \draw[thick] (-3 cm ,0) circle (.45 cm);
    \draw[thick] (0 ,0) circle (.45 cm);
    \draw[thick] (3 cm,0) circle (.45 cm);
    \draw[thick] (35: 4.5mm) -- +(2.3 cm, 0); {second above}
    \draw[xshift=-3.15 cm,thick] (0: 6 mm) -- +(2.1 cm, 0); {first line}
    \draw[xshift=-0.15 cm,thick] (0: 6 mm) -- +(2.1 cm, 0); {second midle line}

    \draw[thick] (-35: 4.5mm) -- +(2.3 cm, 0); {second below}
    \draw[thick] (1, 0mm) -- +(0.8 cm, 0.6); {arrow above}
    \draw[thick] (1, 0mm) -- +(0.8 cm, -0.6); {arrow below}
    \draw[thick] (-3.34, -3.5mm) -- +(0.7, 6.7mm); {cross}
    \draw[thick] (-3.34, 3.5mm) -- +(0.7, -6.7mm); {cross}
  \end{tikzpicture}

\[
 \text{Cartan matrix}=A_{\Pi}=
 \begin{pmatrix}
  0 & 1 & 0  \\
  -1 & 2 & -3 \\
  0  & -1  & 2   
 \end{pmatrix}
\]

\end{center}

The positive roots with respect to $\S$ are $\Delta^+=\Delta^+_0\cup\Delta^+_1$, where
$$\Delta^+_0=\{\e_1;\ \e_2;\ -\e_3;\ 2\d;\ \epsilon_2-\epsilon_1;\  \e_1-\e_3;\ \e_2-\e_3 \}\text{ and }\Delta^+_1=\{ \d; \pm\e_i+\d\}.$$ 

The Weyl vector is $\rho=\rho_0-\rho_1=2\e_1+3\e_2-\frac{5}{2}\d$, where $\r_0=2\e_1+3\e_2+\d$ and $\r_1=\frac{7}{2}\d$. The Weyl groups $W$ is the group $W=D_6\oplus\Z/2\Z$, where $D_6$ is the dihedral group of order 12. It is generated by four reflections: for an arbitrary permutation $(ijk)\in S_3$, we get three $\s_i(e_i)=e_i$, $\s_i(e_j)=e_k$, $\s_i(e_k)=e_j$, one reflection defined by $\tau(e_i)=-e_i$ for all $i$ and $\tau(\d)=\d$, also one reflection $\s(e_i)=e_i$ for all $i$ and $\s(\d)=-\d$. 


The integral weight lattice for $\g_{\0}$ is $\L=\Z\e_1\oplus\Z\e_2\oplus \Z\d $, which is the lattice spanned by the fundamental weights $\w_1=\d$, $\w_2=\e_1+\e_2$, $\w_3=\e_1+2\e_2$ of $\g_\0$. Also, $\L/Q\cong \{1\}$, where $Q$ is the root lattice. We can define the parity function on $\L$, by setting  $p(\e_i)=0$ and $p(\d)=1$.

Using \rlem{dominant weights} and \rlem{Kac}, it is convenient to write down dominant weights in terms of basis $\{\e_1,\e_2, \d\}$: 
\begin{lemma} A weight $\l$ is a dominant integral weight of $\g=G(3)$ if and only if $\l+\rho\in\{(b_1,b_2,b_3)\in\Z\times\Z\times(\frac{1}{2}+\Z)| 2b_1>b_2>b_1>0;\ \text{either}\ b_3>0;\ \text{or if}\ b_3=-\frac{1}{2},\ \text{then}\ b_2=2b_1-1;\ b_3\neq-\frac{3}{2};\ \text{if}\ b_3=-\frac{5}{2},\ \text{then}\ b_1=2\ \text{and}\ b_2=3\}$. 

\end{lemma}

\begin{proof}
Using \rlem{dominant weights} as in the case of $F(4)$.
\end{proof}

\subsection{Associated variety and the fiber functor}
Let $G_0$ be simply-connected connected Lie group with Lie algebra $\g_{\0}$, for a Lie superalgebra $\g$. Let $X=\{ x\in\g_{\1}| [x,x]=0\}$. Then $X$ is a $G_0$-invariant Zariski closed set in $\g_{\1}$, called the {\it{self-commuting cone}} in $\g_{\1}$, see \cite{DuSe}. 

Let $S$ to be the set of subsets of mutually orthogonal linearly independent isotropic roots of $\D_1$. So the elements of $S$ are $A=\{\a_1,...,\a_k|(\a_i,\a_j)=0\}$. Let $S_k=\{A\in S|\ |A|=k\}$ and $S_0=\emptyset$. 

It was proven in \cite{DuSe} that every $G_0$-orbit on $X$ contains an element $x=X_{\a_1}+\cdots+X_{\a_k}$ with $X_i\in\g_{\a_i}$ for some set $\{\a_1,..., \a_k\}\in S$. The number $k$ is called {\it{rank}} of $x$. It was also proven in \cite{DuSe} that there are finitely many $G_0$-orbits on $X$ and they are in one-to-one correspondence with $W$-orbits in $S$. For $\g=F(4)$ or $G(3)$, $k$ is equal to 1.

 \begin{lemma} For an exceptional Lie superalgebra $\g=F(4)$ or $G(3)$, the {\it{rank}} of $x\in X \setminus \{0\}$ is 1. And every $x$ is $G_0$ conjugate to some $X_{\a}\in\g_\a$ for some isotropic root $\a$ with $[h,X_{\a}]=\a(h)X_{\a}$ for all $h\in\h$. 
 \end{lemma}
 
 \begin{proof} For exceptional Lie superalgebras, $X$ has two $G_{\0}$-orbits: $\{0\}$ and the orbit of a highest vector in $\g_{\1}$. The set $S$ also consists of two $W$-orbits: $\emptyset$ and the set of all isotropic roots in $\D$. For $F_4$, the set of all isotropic roots is $\D_1$. For $G(3)$, this set is $\D_1\setminus \{\d\}$. 
 \end{proof}

Let $X_k=\{x\in X, rank{x}=k\}$. Then $X=\cup_{k\leq def{\g}}X_k$ and $\bar{X}=\cup_{j\leq k}X_j$. For a $\g$-module $M$ and for $x\in X$, define the fiber $M_x=Ker\,{x}/Im\,{x}$ as the cohomology of $x$ in $M$ as in \cite{S3}. The {\it{associated variety}} $X_M$ of $M$ is defined in \cite{S3} by setting $X_M=\{x\in X|M_x\neq 0\}$.

From \cite{DuSe}, we know that if $M$ is  finite dimensional, then $X_M$ is a $G_0$-invariant Zariski closed subset of $X$. Also, if $M$ is finite dimensional $\g$-module, then for all $x\in X$, $sdim{M}=sdim{M_x}$. 

We can assume that $x=\sum X_{\a_i}$ with $X_{\a_i}\in\g_{\a}$ for $i=1,\dots,n$. Then, there is a base containing the roots $\a_i$ for $i=1,\dots,n$. We define quotient as in \cite{DuSe} by $\g_x=C_\g(x)/[x,\g]$, where $C_\g(x)=\{a\in\g | [a,x]=0\}$ is the centralizer of $x$ in $\g$, since $[x,\g]$ is an ideal in $C_\g(x)$. The superalgebra $\g_x$ has a Cartan subalgra $\h_x=(Ker\,{\a_1}\cap\cdots\cap Ker\,{\a_k})/(kh_{\a_1}\oplus\cdots\oplus kh_{\a_k})$ and a root system is equal to $\D_x=\{\a\in\D| (\a,\a_i)=0\ \text{for}\  \a\neq\pm\a_i \ \text{and}\ i=1,..., k\}$.

Since $Ker\,{x}$ is $C_\g(x)$-invariant and $[x,\g]Ker\,{x}\subset Im\,{x}$, $M_x$ has a structure of a $\g_x$-module. We can define $U(\g)^x$ to be subalgebra of $ad_x$-invariants. Then we have an isomorphism $U(\g_x)\cong U(\g)^x/[x,U(\g)]$, which is given by $U(\g_x)\rightarrow U(\g)^x\rightarrow U(\g_x)/[x,U(\g)]$. The corresponding projection $\phi: U(\g)^x\rightarrow U(\g_x)$ is such that $\phi(Z(\g))\subset Z(\g_x)$ and thus it can be restricted to a homomorphism of rings $\phi: Z(\g)\rightarrow Z(\g_x)$. The dual of this map is denoted by $\check{\phi}: Hom(Z(\g_x), \mathbb{C})\rightarrow Hom(Z(\g),\mathbb{C})$.

Thus, $M\rightarrow M_x$ defines a functor from the category of $\g$-modules to the category of $\g_x$-modules, which is called the {\it{fiber functor}}. By construction, if central character of $M$ is equal to $\chi$, then the central character of $M_x$ is in $\check{\phi}^{-1}\{\chi\}$. It was proven in \cite{DuSe} that for a finite dimensional $\g$-module with central character $\chi$ and $at(\chi)=k$, we have $X_M\subset \bar{X}_k$.

For $x=\sum X_{\a_i}$ with $X_{\a_i}\in\g_{\a}$ for $i=1,\dots,n$, we can chose a base containing the roots $\a_i$ for $i=1,\dots,n$. This gives $\h_x^*=(\C\a_1\oplus\cdots\oplus\C\a_k)^{\perp}/(\C\a_1\oplus\cdots\oplus\C\a_k)$ and a natural projection $p: (\C\a_1\oplus\cdots\oplus\C\a_k)^{\perp}\rightarrow \h_x^*$. Then $\nu$, $\nu' \in p^{-1}(\m)$ imply $\chi_{\nu}=\chi_{\nu'}$ and $\check{\phi}^{-1}(\chi_{\mu})=\chi_{\nu}$, see \cite{S2}. 

\subsection{Blocks}

Let $\g=F(4)$ or $G(3)$. Consider a graph with vertices the elements of $X^+$ and arrows between each two vertices if they have a non-split extension. The connected components of this graph are called {\it{blocks}}. All the simple components of an indecomposable module belong to the same block, then we say that the indecomposable module itself belongs to this block.

For Lie superalgebras, the generalized central character may correspond to more than one simple $\g$-module. The category $\F$ decomposes into direct sum of full subcategories called $\F^\chi$, where $\F^\chi$ consists of all finite dimensional modules with (generalized) central character $\chi$. Let $F^{\chi}$ be the set of all weights $\l$ such that $L_\l\in\F^{\chi}$. We will call the subcategories $\F^\chi$ blocks, since we will prove they are blocks in the above sense.

In this section, we describe all integral dominant weights in the {\it{atypical}} blocks, which are blocks containing more than one simple $\g$-module. Denote $\l^w:=w(\l+\r)-\r$. 

\begin{lemma}(Serganova, \cite{S2})\label{weights in block} There is a set of odd roots $\alpha_1,\ldots,\alpha_k\in\Delta_1$ and $\mu\in \h^*$ a weight, such that $(\alpha_i,\alpha_j)=0$ and $(\mu+\rho,\alpha_i)=0$. Then $m_\chi=\{\mu\in \h^*| \chi=\chi_{\mu}\}=\cup_{w\in W}(\mu+\mathbb{C}\alpha_1+\dots+\mathbb{C}\alpha_k)^w$.
\end{lemma}

If $k=0$ in the above lemma, then $\chi$ is called \textit{typical}, and if $k>0$, then it is
\textit{atypical}. The number $k$ is called \textit{the degree of atypicality} of $\mu$.

\begin{lemma}\label{atypicality 1} The degree of atypicality of any weight for $\g$ is $\leq1$. 
\end{lemma}

Recall that $X$ is the cone of self-commuting elements defined above. For a $\g$-module $M$ and for $x\in X$, recall that $M_x=Ker\,{x}/Im\,{x}$ is the fiber as the cohomology of $x$ in $M$ and $\g_x=C_\g(x)/[x,\g]$.

Denote by $\r_l$ the Weyl vector and by $\w_1=\frac{1}{3}(2\b_1+\b_2)$ and $\w_2=\frac{1}{3}(\b_1+2\b_2)$ the fundamental weights for $\mathfrak{sl}(3)$, where $\b_1=\e_1-\e_2$ and $\b_2=\e_2-\e_3$ are the simple roots of $\mathfrak{sl}(3)$. The following lemma allows us to parametrize the atypical blocks by of $\g=F(4)$ and label them $\F^{(a,b)}$.

\begin{lemma}\label{dominant integral weights of F(4)} If $\g=F(4)$ and $x\in X$, then $\g_x\cong\mathfrak{sl}(3)$. For any simple module $M\in \F^{\chi}$ for atypical $\chi$, we have $$M_x\cong L_{a,b}^{\oplus m_1}\oplus L_{b,a}^{\oplus m_2}\oplus \Pi(L_{a,b})^{\oplus m_3}\oplus \Pi(L_{b,a})^{\oplus m_4},$$ where $L_{a,b}$ is a simple $\mathfrak{sl}(3)$-module with highest weight $\m$ of $\mathfrak{sl}(3)$ such that $\m+\r_l=a\w_1+b\w_2$. Here, $a=3n+b$ with $(a,b)\in \mathbb{N}\times \mathbb{N}$ and $n\in\mathbb{Z}_{\geq 0}$ such that $a=b$ or $a>b$. 

\end{lemma}

\begin{proof}  By \rlem{weights in block} and \rlem{atypicality 1}, we can take $x=X_{\a_1}$ with $X_{\a_1}\in\g_{\a_1}$ and $\a_1=\frac{1}{2}(\e_1+\e_2+\e_3-\d)\in\D_1$. Then the root system for $\g_x$ is $\D_x=\{\e_i-\e_j\}_{i\neq j}$, $i,j=1,2,3$ and it correspond to the root system of $\mathfrak{sl}(3)$ proving the first part.

Let $M \in \F^{\chi}$ be the simple $\g$-module with highest weight $\l$, then $(\l+\r,\b)=0$ for some $\b\in\D$. We choose $w\in W$, with $w(\b)=\a_1$, then $(w(\l+\r),\a_1)=0$ and $w(\l+\r)-\r$ is dominant with respect to $\mathfrak{sl}(3)$.

Let $w(\l+\r)=\l'+\r=a_1\e_1+a_2\e_2+a_3\e_3+a_4\d$. Also, let $a+\frac{1}{2}=(w(\l+\r),\b_1)=a_1-a_2$ and $b+\frac{1}{2}=(w(\l+\r),\b_2)=a_2-a_3$. 


Now we have $w(\l+\r)=\l'+\r=a_1\e_1+a_2\e_2+a_3\e_3+a_4\d=(a_1+a_4, a_2+a_4,a_3+a_4|0)-2a_4\a_1=(\frac{2a+b}{3}+\frac{1}{2}, \frac{-a+b}{3},-\frac{a+2b}{3}-\frac{1}{2}|0)-2a_4\a_1=a\w_1+b\w_2+\r_l-2a_4\a_1$. 

Similarly, there is $\s\in W$ such that $\s(w(\l+\r))=\s(\l'+\r)=\l''+\r=\s(a_1\e_1+a_2\e_2+a_3\e_3+a_4\d)=(-a_3, -a_2,-a_1|a_4)=(-a_4-a_3, -a_4-a_2,-a_4-a_1|0)+2a_4\a_1=(\frac{2b+a}{3}+\frac{1}{2}, \frac{-b+a}{3},-\frac{b+2a}{3}-\frac{1}{2}|0)+2a_4\a_1=b\w_1+a\w_2+\r_l+2a_4\a_1$.

This implies, $\l'\in p^{-1}(a\w_1+b\w_2)$ or $\l''\in p^{-1}(b\w_1+a\w_2)$, which correspond to the dominant integral weights of $\g_x=\mathfrak{sl}(3)$ since $a$ and $b$ are positive integers. Also, $a-b=-3(a_2+a_4)$, implying that $a=3n+b$. 

From above, we have that $\l',\l'' \in p^{-1}(\m)$, where $\m=a\w_1+b\w_2$ or $b\w_1+a\w_2$ is a dominant integral weight of $\g_x=\mathfrak{sl}(3)$ such that $a=3n+b$. From \rlem{weights in block}, we have $\chi_\l=\chi_{\l'}=\chi_{\l''}$. By construction of $\check{\phi}$ above, the central character of $M_x$ is in the set $\check{\phi}^{-1}\{\chi_\l \}$. Also, if $\l\in p^{-1}(\m)$, then $\check{\phi}(\chi_\m)=\chi_\l$. Therefore, $M_x$ contains Verma modules over $\g_x$ with highest weights in $p(\l')$ for any $\l'$ such that $\chi_\l=\chi_{\l'}$, proving the lemma. 



Conversely, for  $(a,b)\in \mathbb{N}\times \mathbb{N}$ with $a-b=3n$, there is a dominant weight $\l\in F^{\chi}$ with $\l+\r=(a+b+1)\e_1+(b+1)\e_2+\e_3+(\frac{a+2b}{3}+1)\d$, such that $p(\l)=a\w_1+b\w_2$.  
\end{proof}

Similarly, denote by $\r_l$ is the Weyl vector and by $\w_1=\frac{1}{2}\b_1$ be the fundamental weight of $\mathfrak{sl}(2)$, where $\b_1=\e_1-\e_2$ is the simple root of $\mathfrak{sl}(2)$.

The following lemma allows parametrize the atypical blocks by of $\g=G(3)$ by $a=2n+1$, with $n\in\Z_{\geq 0}$ and label them $\F^{a}$. 

\begin{lemma}\label{dominant integral weights of G(3)}
If $\g=G(3)$ and $x\in X$, then $\g_x\cong\mathfrak{sl}(2)$. For any simple $M\in \F^{\chi}$ for atypical $\chi$, we have $$M_x\cong L_{a}^{\oplus m_1}\oplus\Pi(L_{a)})^{\oplus m_2},$$ where $L_{a}$ is a simple $\mathfrak{sl}(2)$-module with dominant weight $\m$ with $\m+\r_l=a\w_1$. Here, $a=2n+1$ with $n\in\mathbb{Z}_{\geq 0}$. 
\end{lemma}

\begin{proof}  Similar to the proof of \rlem{dominant integral weights of F(4)} for $F(4)$. 
\end{proof}

We can describe the dominant integral weights in the atypical blocks. 




In the following two theorems, for every $c$, we describe a unique dominant weight $\l_c$ in $\F^{(a,a)}$ (or $\F^{a}$), such that $c$ is equal to the last coordinate of $\l_c+\r$. For $\l_c$ in $\F^{(a,b)}$ with $a\neq b$, $c$ is equal to the last coordinate of $\l_c+\r$ if $c$ is positive and to the last coordinate of $\l_c+\r$ with negative sign if $c$ is negative.

For $F(4)$, we denote: $t_1=\frac{2a+b}{3},t_2=\frac{a+2b}{3},t_3=\frac{a-b}{3}$.  Note that if $a=b$, $t_1=t_2=a$ and $t_3=0$.

\begin{theorem}\label{cf4} Let $\g=F(4)$.

(1) It is possible to parametrize the dominant weights $\l$ with $L_\l\in \F^{(a,a)}$ by $c\in\frac{1}{2}\Z_{\geq -1}\setminus \{a, \frac{a}{2}, 0\}$ for $a>1$ and by $c\in\frac{1}{2}\Z_{\geq 3}\cup \{\frac{-3}{2}\}$ for $a=1$, such that $(\l+\r,\d)=3c$.


(2) Similarly, it is possible to parametrize the dominant weights $\l$ with $L_\l\in \F^{(a,b)}$ by $c\in\frac{1}{2}\mathbb{Z}\setminus\{t_2,\frac{t_1}{2}, t_3, -\frac{t_3}{2}, -\frac{t_2}{2}, -t_1,\}$, such that $(\l+\r,\d)=3sign(c)\,c$.


\end{theorem}

\begin{proof}  
To prove (1), take $\l+\r=(2a+1)\e_1+(a+1)\e_2+\e_3+(a+1)\d$, then $\lambda \in F^{(a,a)}$ by \rlem{dominant integral weights of F(4)}, and $(\l+\r, \a)=0$ for $\a=\frac{1}{2}(\e_1+\e_2+\e_3+\d)$. By \rlem{weights in block} and \rlem{atypicality 1}, all dominant integral weights in $F^{(a,a)}$ are in $A=\{w(\l+\r+k\a)-\r| w\in W, k\in\Z\}$. The last coordinate of dominant integral weights in $A$ is $c:=a+1+\frac{k}{2}\in\frac{1}{2}\Z$.

Since for $c\in\{\pm{a}, \pm\frac{a}{2}, 0\}$, the element of $A$ with $k=2(c-a-1)$ is not dominant for any $w\in W$, we consider the following eight intervals for $c$: (1) $a<c$; (2)$\frac{a}{2}<c<a$; (3) $0<c<\frac{a}{2}$; (4) $c=-\frac{1}{2}$; (5) $c<-a$; (6) $-a<c<-\frac{a}{2}$; (7) $-\frac{a}{2}<c<0$; (8) $c=\frac{1}{2}$.

For every $c$, in the above intervals, we define corresponding Weyl group element as follows: (1) $w_c=id$, (2) $w_c=\tau_3$, (3) $w_c=\s_1\tau_3$, (4) $w_c=\s_1\s_2\tau_2\tau_3$; (5) $w_c=\s$, (6) $w_c=\tau_3\s$, (7) $w_c=\s_1\tau_3\s$, (8) $w_c=\s_1\s_2\tau_2\tau_3\s$. The last four cases give us same dominant weights as in the first four cases.

Since $\l+\r=a(e_1-e_3)+(a+1)\a$, the dominant integral weight $\l_c$ corresponding to this $c$ can be written as follows:

For $c\in J_1=(a,\infty)$, $\l_c+\r=a(e_1-e_3)+2c\b_1$, where $\b_1=\frac{1}{2}(\e_1+\e_2+\e_3+\d)=w_c(\a)$;

For $c\in J_2=(\frac{a}{2},a)$, $\l_c+\r=a(e_1+e_3)+2c\b_2$, where $\b_2=\frac{1}{2}(\e_1+\e_2-\e_3+\d)=w_c(\a)$;

For $c\in J_3=(0,\frac{a}{2})$, $\l_c+\r=a(e_1+e_2)+2c\b_3$, where $\b_3=\frac{1}{2}(\e_1-\e_2+\e_3+\d)=w_c(\a)$;

We also have the following cases: 

Let $a=1$. For $c=-\frac{3}{2}$, $\l_c+\r=e_1-e_3-2c\b_0$, where $\b_0=\frac{1}{2}(\e_1+\e_2+\e_3-\d)$.

Let $a>1$. For $c=-\frac{1}{2}$, $\l_c+\r=a(e_1+e_2)-2c\b_0$, where $\b_0=\frac{1}{2}(\e_1-\e_2+\e_3-\d)$.


 For (2), we take $\l\in F^{(a,b)}$, such that $\l+\r=t_1\e_1+t_2\e_2+t_3\e_3$. By \rlem{dominant integral weights of F(4)},  $\l\in X^+$ and $(\l+\r,\a)=0$ for $\a=\frac{1}{2}(\e_1-\e_2-\e_3+\d)$.

 By \rlem{weights in block} and \rlem{atypicality 1}, all dominant integral weights in $F^{(a,a)}$ are in $A=\{w(\l+\r+\frac{k}{2}\a)-\r| w\in W, k\in\Z\}$. Let $c:=\frac{k}{2}\in\frac{1}{2}\Z$.

Since for $c\in\{t_2,\frac{t_1}{2}, t_3, -\frac{t_3}{2}, -\frac{t_2}{2}, -t_1,\}$, the element of $A$ with $k=2c$ is not dominant for any $w\in W$, we consider the following eight intervals for $c$:

(1) $t_2<c$; (2)$\frac{t_1}{2}<c<t_2$; (3) $t_3<c<\frac{t_1}{2}$; (4) $0\leq c<t_3$; (5) $-\frac{t_3}{2}<c<0$; (6) $-\frac{t_2}{2}<c<-\frac{t_3}{2}$; (7) $-t_1<c<-\frac{t_2}{2}$; (8) $c<-t_1$.

For every $c$, in the above intervals, we define corresponding Weyl group element as follows: (1) $w_c=\tau_3\s_1\tau_3$, (2) $w_c=\s_1\tau_3$, (3) $w_c=\tau_3$, (4) $w_c=id$; (5) $w_c=\s$, (6) $w_c=\s_3\s$, (7) $w_c=\s_1\s_3\s$, (8) $w_c=\tau_3\s_1\s_3\s$.

Then, it is easy to check using \rlem{dominant integral weights of F(4)} that $w_c\in W$ is the unique element such that $\l_c+\r=w_c(\l+\r+c\d)\in X^+$.


In each case, we list the dominant integral weights in $F^{(a,b)}$, parametrized by $c$:

For $c\in I_1=(t_2,\infty)$, $\l_c+\r=t_1\e_1-t_3\e_2-t_2\e_3+2c\b_1$, where $\b_1=\frac{1}{2}(\e_1+\e_2+\e_3+\d)=w_c(\a)$;

For $c\in I_2=(\frac{t_1}{2},t_2)$, $\l_c+\r=t_1\e_1-t_3\e_2+t_2\e_3+2c\b_2$, where $\b_2=\frac{ 1}{2}(\e_1+\e_2-\e_3+\d)=w_c(\a)$;

For $c\in I_3=(t_3,\frac{t_1}{2})$, $\l_c+\r=t_1\e_1+t_2\e_2-t_3\e_3+2c\b_3$, where $\b_3=\frac{1}{2}(\e_1-\e_2+\e_3+\d)=w_c(\a)$;

For $c\in I_4=[0,t_3)$, $\l_c+\r=t_1\e_1+t_2\e_2+t_3\e_3+2c\b_4$, where $\b_4=\frac{1}{2}(\e_1-\e_2-\e_3+\d)=w_c(\a)$;

For $c\in I_5=(-\frac{t_3}{2},0)$, $\l_c+\r=t_1\e_1+t_2\e_2+t_3\e_3-2c\b_5$, where $\b_5=\frac{1}{2}(-\e_1+\e_2+\e_3+\d)=-w_c(\a)$;

For $c\in I_6=(-\frac{t_2}{2},-\frac{t_3}{2})$, $\l_c+\r=t_2\e_1+t_1\e_2+t_3\e_3-2c\b_6$, where $\b_6=\b_3=-w_c(\a)$;

For $c\in I_7=(-t_1,-\frac{t_2}{2})$, $\l_c+\r=t_2\e_1+t_3\e_2+t_1\e_3-2c\b_7$, where $\b_7=\b_2=-w_c(\a)$;

For $c\in I_8=(-\infty,-t_1)$, $\l_c+\r=t_2\e_1+t_3\e_2-t_1\e_3-2c\b_8$, where $\b_8=\b_1=-w_c(\a)$.

The uniqueness of $\l_c$ in both cases follows from \rlem{dominant integral weights of F(4)}.

\end{proof}

\begin{remark}\label{remark} For every $\lambda_c \in F^{(a,a)}$ or $F^{(a,b)}$, such that $(\l+\r,\d)=3c$ or $(\l+\r,\d)=-3c$ with $c\in J_i$ or $I_i$, we have corresponding $\b_i\in\D^+$ is such that $(\l+\r,\b_i)=0$. This $\b_i=w_c(\a)$, where $w_c$'s are defined in the above proof. 
\end{remark} 


\begin{theorem}\label{cg3} Let $\g=G(3)$.

(1) It is possible to parametrize the dominant weights $\l$ with $L_\l\in \F^{1}$ by $c\in(\frac{1}{2}+\Z_{\geq 2})\cup \{-\frac{5}{2}\}$, such that $(\l+\r,\d)=3c$.

(2) Similarly, for $a>0$, it is possible to parametrize the dominant weights $\l$ with $L_\l\in \F^a$ by $c\in (-\frac{1}{2}+\Z_{\geq 0})\setminus\{0,\frac{a}{2}, \frac{3a}{2}\}$, such that $(\l+\r,\d)=3c$.


\end{theorem}

\begin{proof} (1) Let $a=0$. In this case, take $\l+\r=2\e_1+3\e_2+\frac{5}{2}\d$, then $\lambda \in F^{1}$ by \rlem{dominant integral weights of G(3)}, and $(\l+\r, \a)=0$ for $\a=\e_1+\e_2+\d$. By \rlem{weights in block} and \rlem{atypicality 1}, all dominant integral weights in $F^{1}$ are in $A=\{w(\l+\r+k\a)-\r| w\in W, k\in\Z\}$. The last coordinate of dominant integral weights in $A$ is $c:=\frac{5}{2}+k\in\frac{1}{2}+\Z$, so $k\in\Z$. 

Since for $c=\pm\frac{3}{2}$, the element of $A$ with $k=c-\frac{5}{2}$ is not dominant for any $w\in W$, we consider the following intervals for $c$: (1) $\frac{3}{2}<c$; (2) $c=-\frac{5}{2}$; (3) $c<-\frac{3}{2}$; (4) $c=\frac{5}{2}$.

For every $c$, in the above intervals, we define corresponding Weyl group element as follows: (1) $w_c=id$, (2) $w_c=\s_3\tau$, (3) $w_c=\s_3\tau\s$, (4) $w_c=\s$. The last two cases correspond to the same dominant weights as in the first two cases.

Since $\l+\r=\frac{1}{2}(e_2-e_1)+\frac{5}{2}\a$, the dominant integral weight  $\l_c$ with last coordinate $c$ can be written as follows:

$c\in J_1=(\frac{3}{2},\infty)$, then $\l_c+\r=\frac{1}{2}(e_2-e_1)+c\a$, $\b=\e_1+\e_2+\d=w_c(\a)$;

$c=-\frac{5}{2}$, $\l_c+\r=(2, 3, -\frac{5}{2})$, $\b=-\e_1-\e_2+\d=w_c(\a)$.

(2) Let $a>0$. In this case, take $\l=\e_1+(a+1)\e_2+(1+\frac{a}{2})\d-\r$, then $\lambda \in F^{a}$ by \rlem{dominant integral weights of G(3)}, and $(\l+\r, \a)=0$ for $\a=\e_1+\e_2+\d$. By \rlem{weights in block} and \rlem{atypicality 1}, all dominant integral weights in $F^{a}$ are in $A=\{w(\l+\r+k\a)-\r| w\in W, k\in\Z\}$. The last coordinate of dominant integral weights in $A$ is $c:=\frac{a}{2}+1+k\in\frac{1}{2}+\Z$, so $k\in\Z$.

Since for $c\in\{\pm\frac{a}{2}, \pm\frac{3a}{2}\}$, the element of $A$ with $k=c-1-\frac{a}{2}$ is not dominant for any $w\in W$, we consider the following intervals for $c$: (1) $\frac{3a}{2}<c$; (2) $\frac{a}{2}<c<\frac{3a}{2}$; (3) $0<c<\frac{a}{2}$; (4) $c=-\frac{1}{2}$; (5) $c<-\frac{3a}{2}$; (6) $-\frac{3a}{2}<c<-\frac{a}{2}$; (7) $-\frac{a}{2}<c<0$; (8) $c=\frac{1}{2}$.

For every $c$, in the above intervals, we define corresponding Weyl group element as follows: (1) $w_c=id$, (2) $w_c=\s_1\tau$, (3) $w_c=\s_1\s_2\tau$, (4) $w_c=\s_2$; (5) $w_c=\tau\s_3\s$, (6) $w_c=\s_3\s_1\s$, (7) $w_c=\s_2\s$, (8) $w_c=\s_3\s_2\tau\s$. The last four cases correspond to the same dominant weights as in the first four cases.

Since $\l_c+\r=\frac{a}{2}(e_2-e_1)+c\a$, the dominant integral weight $\l_c$ corresponding to $c$ can be written as follows:

For $c\in J_1=(\frac{3}{2}a,\infty)$, $\l_c+\r=(c-\frac{a}{2},c+\frac{a}{2}, c)$, $\b=\e_1+\e_2+\d=w_c(\a)$;

For $c\in J_2=(\frac{1}{2}a,\frac{3}{2}a)$, $\l_c+\r=(a, c+\frac{a}{2}, c)$, $\b=\e_2+\d=w_c(\a)$;

For $c\in J_3=(0,\frac{1}{2}a)$, $\l_c+\r=(c+\frac{a}{2}, a, c)$, $\b=\e_1+\d=w_c(\a)$;

For $c=-\frac{1}{2}$, $\l_c+\r=(\frac{a}{2}+\frac{1}{2}, a, -\frac{1}{2})$, $\b=-\e_1+\d=w_c(\a)$.

The uniqueness of $\l_c$ in both cases follows from \rlem{dominant integral weights of G(3)}.

\end{proof}

\begin{remark}\label{remark1} For every $\l_c \in F^1$ or $F^a$, such that $(\l_c+\r,\d)=2c$, we have corresponding $\b=w_c(\a)\in\D^+$ is such that $(\l_c+\r,\b)=0$, where $w_c$'s are defined in the above proof. 
\end{remark} 

From \rlem{dominant integral weights of F(4)} and \rlem{dominant integral weights of G(3)}, we obtain the proof of the first parts of \rthm{CBF} and \rthm{CBG}.

\section{Geometric induction and translation functor}

\subsection{Geometric induction}

We fix a Borel subalgebra $\bb$ of $\g$, and let $V$ be a $\bb$-module. Denote by $\mathcal{V}$ the vector bundle $G\times_BV$ over the generalized grassmannian $G/B$. The space of sections of $\mathcal{V}$ has a natural structure of a $\g$-module, in other words the sheaf of sections of $\mathcal{V}$ is a $\g$-sheaf.

Let $C_\l$ denote the one dimensional representation of $B$ with weight $\l$. Denote by $\mathcal{O}_\l$ the line bundle $G\times_B C_\l$ on the flag (super)variety $G/B$. See \cite{BGGSeGr}. The functor $\G_i$ from the category of $\bb$-modules, to the category of $\g$-modules can be defined by $\G_i(G/B, V)=\G_i(G/B, \V):=(H^i(G/B,\V^*))^*$as in \cite{Ser}.

Denote by $\varepsilon(\l)$ the {\it{Euler characteristic}} of the sheaf $\O_\l$ belonging to the category $\F$:

$$\varepsilon(\l)=\sum_{i=0}^{dim(G/B)_0}(-1)^i[\G_i(G/B,\O_\l):L_\m][L_\m].$$

The following properties of this functor will be useful:

\begin{lemma}\label{1} (\cite{GrSe}) If 
\begin{displaymath}
    \xymatrix{
        0\ar[r] & U \ar[r] & V \ar[r] & W\ar[r] & 0 }
\end{displaymath}
is a short exact sequence of $B$-modules, then one has the following long exact sequence
\begin{displaymath}
    \xymatrix{
        ...\ar[r] & \G_1(G/B, W) \ar[r] & \G_0(G/B, U)\ar[r] & \G_0(G/B, V)\ar[r] & \G_0(G/B, W)\ar[r] & 0 }
\end{displaymath}
\end{lemma}

\begin{lemma}\label{2} (\cite{GrSe}) The module $\G_0(G/B, \V)$ is the maximal finite-dimensional quotient of $\U(\g)\otimes_{\U(\bb)}V$.
\end{lemma}


\begin{corollary}(\cite{GrSe})\label{z} For every dominant weight $\l$, the module $L_{\l}$ is a quotient of $\G_0(G/B,\O_{\l})$ with $[\G_0(G/B,\O_{\l}):L_{\l}]=1$. 
\end{corollary}

\begin{lemma}\label{3} (\cite{GrSe}) If $L_{\m}$ occurs in $\G_i(G/B, \O_\l)$ with non-zero mulitiplicity, then $\m+\r=w(\l+\r)-\sum_{\a\in I} \a$ for some $w\in W$ of length $i$ and $I\subset \D^+_1$. 
\end{lemma}

\begin{lemma}\label{4}(\cite{Penkov}) Assume for an even root $\gamma$ in the base of $B$, $\b+\r=r_{\gamma}(\a+\r)$. Then $\G^i(G/B,\O_{\a})\cong \G^{i+1}(G/B,\O_{\b})$. 
\end{lemma}

\begin{lemma}\label{7} If $L_{\l}\in \F^{\chi}$, then $$\sum_i (-1)^i sdim{\G_i(G/B, \O_{\l})}=0.$$
\end{lemma}

\begin{proof} We follow similar argument as in lemma 5.2 in \cite{S3}. Let $\l\in F^{\chi}$, then for $t\in\mathbb{Z}$, the weight $\l+t\d$ is integral. The weight $\l+t\d$ is typical for almost all $t$. Since for $\l$ typical weight, \rthm{BWB} holds ( see \cite{Penkov}), we have $\G_i(G/B, L_{\l})=0$ for $i>0$ and $\G_0(G/B, L_{\l})=L_{\l}$. Also, since $\l+t\d$ is typical we have:
$$\sum_i(-1)^isdim\G_i(G/B, L_{\l+t\d})=sdim(\l+t\d)=0.$$

On the other hand, we have $chL_{\l+t\d}=e^{t\d}chL_{\l}$. Therefore, from \rthm{chcoh} we have:

$$\sum_i(-1)^isdim\G_i(G/B, L_{\l+t\d})=p(t)$$

for some polynomial $p(t)$. We have $p(t)=0$ for almost all $t\in\mathbb{Z}$. Thus, $p(0)=0$. 
\end{proof}


\begin{lemma}\label{sums} (\cite{BGGSeGr}) For every dominant weight $\l$, let $p(w)$ be the parity of $w$, then

$$\sum_i (-1)^ich{\G_i(G/B, \O_{\l})}=(-1)^{p(w)}\sum_i (-1)^i ch{\G_i(G/B, \O_{w(\l+\r)-\r})}$$
\end{lemma}

Let $\bb=\h\oplus \nn$, where $\g=\nn^-\oplus\h\oplus\nn$, and $\nn$ is the nilpotent part of $\bb$. Consider the projection

$$\phi: U(\g)=U(\nn^-)U(\h)U(\nn)\rightarrow U(\nn^-)U(\h)$$ with kernel $U(\g)\nn$. The restriction of $\phi$ to $Z(\g)$ induces the injective homomorphism of centers $Z(\g)\rightarrow Z(\h)$. Denote the dual map by $\Phi: Hom(Z(\h), \C)\rightarrow Hom(Z(\g),\C)$. 

\begin{lemma} (\cite{GrSe}) If $V$ is an irreducible $\bb$-module admitting a central character $\chi$, then the $\g$-module $\G_i(G/B, V)$ admits the central character $\Phi(\chi)$. 
\end{lemma} 

Let $M^{\chi}=\{m\in M | (z-\chi(z))^Nv=0, z\in Z\}$.

\begin{corollary}\label{1a} (\cite{GrSe}) For any finite-dimensional $\g$-module $M$, let $M^{\chi}$ denote the component with generalized central character $\chi$. Then
$$\G_i(G/B, (V\otimes M)^{\Phi^{-1}(\chi)})=(\G_i(G/B, V)\otimes M)^{\chi}.$$
\end{corollary}

\subsection{Translation functor}

A {\it{translation functor}} $T_{\chi,\tau}:\F^{\chi}\rightarrow \F^{\tau}$ is defined by

$$T_{\chi,\tau}(V)=(V\otimes\g)^{\tau},\text{ for }V\in\F^{\chi}.$$

Here, $(M)^{\tau}$ denotes the projection of $M$ to the block $\F^{\tau}$. Since $\g\cong \g^*$, the left adjoint functor of $T_{\chi,\tau}$ is defined by 
 
$$T^*_{\chi,\tau}(V)=(V\otimes\g)^{\chi},\text{ for }V\in\F^{\tau}.$$

For convenience, when its clear, we will denote $T:=T_{\chi,\tau}$.

For $\l\in \F^{\chi}$, we also define $$T_{\chi,\tau}(\O_\l)=(\O_{\l}\otimes\g)^{\Phi^{-1}(\chi)},$$ where $V^{\Phi^{-1}(\chi)}$ is the component with generalized character lying in $\Phi^{-1}(\chi)$.

\begin{lemma}\label{6} (\cite{GrSe}) We have $\G^i(G/B, T(\O_{\l}))=T(\G^i(G/B,\O_{\l}))$, where $T$ is a translation functor. 
\end{lemma}

\begin{lemma}\label{w} Assume $T(\O_{\l})$ has a filtration with quotients $\O_{\sigma_i}$, $i=1,2$ with $\sigma_1$ is dominant and $\sigma_2$  acyclic. Then for all $i\geq 0$, we have $\G_i(G/B,T(\O_{\l}))=\G_i(G/B,\O_{\sigma_1})$. 
\end{lemma}

\begin{proof} We have an exact sequence of vector bundles:

$$0\rightarrow  \O_{\s_2}\rightarrow T(\O_{\l_1})\rightarrow \O_{\s_1}\rightarrow 0,$$

Since $\s_2$ is acyclic, $\G_i(G/B,\O_{\sigma_2})=0$ for all $i\geq 0$. Thus $\G_i(G/B,T(\O_{\l}))=\G_i(G/B,\O_{\sigma_1})$.

\end{proof}

\begin{lemma}\label{surjection} Let $X$ is an indecomposable $\g$-module with unique simple quotient $L_\l$, such that if $L_\sigma$ is a subquotient of $X$ implies $\sigma<\l$, then there is a surjection $\G_0(G/B, \O_\l)\rightarrow X$. 
\end{lemma}
\begin{proof} Follows from \rlem{2}. 
\end{proof}

\begin{lemma}\label{eqcoho} For $\l\in F^{(a,b)}$ (or $F^{a}$) let $T(L_\l)=L_{\l'}$ and $T$ an equivalence of categories $\F^{(a,b)}$ and $\F^{(a+1,b+1)}$ (or $F^{a}$ and $F^{a+2}$) preserving the order on weights. We have $T(\G_0(G/B, \O_{\l}))=\G_0(G/B, \O_{\l'})$. 
\end{lemma}

\begin{proof} From \rlem{2}, $\G_0(G/B, \O_{\l})$ is a maximal indecomposable module with quotient $L_\l$. Since $T$ is an equivalence of categories, $T(\G_0(G/B, \O_{\l}))$ is an indecomposable module with quotient $L_\l'$. All other simple subquotients of $T(\G_0(G/B, \O_{\l}))$ are $L_\s$ with $\s<\l'$.

By \rlem{surjection}, we have a surjection $\G_0(G/B, \O_{\l'})\rightarrow T(\G_0(G/B, \O_{\l}))$. In a similar way we have a surjection $\G_0(G/B, \O_{\l})\rightarrow T^*(\G_0(G/B, \O_{\l'}))$. This proves the equality. 
\end{proof}


\begin{lemma}\label{zero} Let $T=T_{\chi,\tau}$. For $\g=F(4)$, let $\chi=(a,b)$ and $\tau=(a+1,b+1)$ and for $\g=G(3)$, let $\chi=a$ and $\tau=a+2$. Then $T(L_{\l})\neq 0$ for any $\l\in\F^{\chi}$. 
\end{lemma}

\begin{proof} From definition of translation functor, we have $T(L_{\l})=(L_{\l}\otimes\g)^{\tau}$. From \cite{S3}, we have $(M\otimes\g)_x=M_x\otimes\g_x$ for any $\g$-module $M$. Thus, $T(L_{\l})_x=(L_{\l}\otimes\g)^{\tau}_x=((L_{\l})_x\otimes \g_x)^{\Phi^{-1}(\tau)}$, where $(L_{\l})_x$ is an $\g_x$-module. And $\g_x\cong sl(3)$ or $sl(2)$. This implies $T(L_{\l})\neq 0$. 
\end{proof}

\begin{lemma}\label{acyclicx} Let $\l\in F^{\chi}$ be dominant. Assume there is exactly one dominant weight $\m\in F^{\tau}$ of the form $\l+\gg$ with $\gg\in \D$. Then we have $T(L_{\l})=L_{\m}$.
\end{lemma}

\begin{proof} By definition, $T(L_{\l})=(L_{\l}\otimes\g)^{\tau}$. By assumption, $\m$ is the only $\bb$-singular weight in $T(L_\l)$. Since $T(L_\l)$ is contregradient, $T(L_\l)=L_\l\oplus M$. If $M\neq 0$, it must have another $\bb$-singular vector. Hence, $M=0$ and the statement follows. 

 



\end{proof}

\begin{theorem}\label{eqgeneral} Assume, for every $L_{\l}\in \F^{\chi}$, there is a unique $L_{\l'}=T(L_{\l})\in \F^{\tau}$. Also assume for each $L_{\l'}\in \F^{\tau}$, there are at most two weights $\l_1$ and $\l_2$ in $F^{\chi}$ such that $\l'+\ga=\l_i$, $i=1,2$ with $\l_1=\l >\l_2$ and $\g\in \D$. Then the categories $\F^{\chi}$ and $\F^{\tau}$ are equivalent. 
\end{theorem}

\begin{proof} We show that translation functor $T$ defined by $T(L_{\l})=(L_{\l}\otimes\g)^{\tau}$ is an equivalence of categories $\F^{\chi}$ and $\F^{\tau}$. 

It is sufficient to show that we have exact and mutually adjoint functors $T$ and $T^*$, which induce bijection between simple modules. Since we already have that $T$ maps simple modules in $\F^{\chi}$ to simple modules in $\F^{\tau}$, we just need to show that $T^*$ also maps simple modules to simple modules such that $T \cdot T^*=id_{\F^{\tau}}$ and $T^*\cdot T=id_{\F^{\chi}}$. 

Thus, we just show that $T^*(L_{\l'})=L_{\l}$ for each $\l'\in F^{\tau}$. We have $Hom_{\mathfrak{g}}(T^*(L_{\l'}),L_{\m}) = Hom_{\mathfrak{g}}(L_{\l'},T(L_{\m})) =Hom_{\mathfrak{g}}(L_{\l'},L_{\m'}) = \mathbb{C}\text{ for }\l=\m\text{ and }0\text{ otherwise}.$ Similarly, we have $Hom_{\mathfrak{g}}(L_{\m},T^*(L_{\l'})) = Hom_{\mathfrak{g}}
(T(L_{\m}),L_{\l'}) = Hom_{\mathfrak{g}}(L_{\m'},L_{\l'})=\mathbb{C}$ for $\m=\l$ and $0$ otherwise.

The $\bb$-singular vectors in $T^*(L_{\l'})$ have weights of the form $\l=\l'+\gamma$ with $\gamma\in\D$. By assumption of the theorem, all $\bb$-singular vectors in $T^*(L_{\l'})$ are less than or equal to $\l$ in the standard order. Since $T^*(L_{\l'})$ is contragradient and the multiplicity of $L_\l$ in $T^*(L_{\l'})$ is one, we must have $T^*(L_{\l'})=L_\l\oplus M$ for some module $M$. Since $Hom(M, L_{\xi})=0$ for any $\xi\in F^{\chi}$, we have $M=0$. 







 \end{proof}

\begin{lemma} For the distinguished Borel $B$ and dominant weight $\l$, we have $$\G_i(G/B, \O_{\l})=0\text{ for }i>1.$$ 
\end{lemma}
\begin{proof} Consider the bundle $\pi: G/B\rightarrow G/P$, where $P$ is the parabolic subgroup obtained from $B$ by adding all negative even simple roots. The even dimension of $G/P$ equals 1.

On the other hand, $\pi_*^0(\O_{\l})=L_\l(\mathfrak{p})$ and $\pi_*^i(\O_{\l})=0$, since $\l$ is dominant. Hence, by Leray spectral sequence (see \cite{GrSe}), we have $$\G_i(G/B, \O_{\l})=\G_i(G/P, L_{\l}(\mathfrak{p}))=0,$$ where $\mathfrak{p}$ is the corresponding parabolic subalgebra. 
\end{proof}

\section{Generic weights}

\subsection{Character and superdimension formulae for generic weights}


For $\g$-module $M_\l$ and $V=\bigoplus_{\m\in\h^*}V_{\m}$ the weight decomposition of its quotient, we define the {\it{character}} of $V$ by $$ch{V}=\sum_{\m\in P(V)}(dim{V_{\m}})e^{\m}.$$ 

If $\l\in\L^+$ is a typical weight, then the following character formula is proven by Kac and it holds for the exceptional Lie superalgebras:
$$chL_\l=\frac{D_1\cdot e^{\r}}{D_0}\cdot\sum_{w\in W} sign{w}\cdot e^{w(\l+\r)},$$

where $D_0=\prod_{\a\in\D^+_0}(1-e^{-\a})^{dim{\g_{\a}}}$ and $D_1=\prod_{\a\in\D^+_1}(1+e^{-\a})^{dim{\g_{\a}}}$.

The {\it{generic}} weights are defined in \cite{P} to be the weights far from the walls of the Weyl chamber. Here is a more precise definition:


\begin{definition} We define $\l_c\in\F^{\chi}$ with $\chi=(a,b)$ or $\chi=a$ to be a {\it{generic weight}} if $c>\frac{a+2b}{3}+\frac{3}{2}$ or $c<-\frac{3}{2}-\frac{2a+b}{3}$ for $F(4)$ and if $c>\frac{3a}{2}-2$ for $G(3)$.
\end{definition} 

The following theorems will be used later in the proofs: 

\begin{theorem}(Penkov, \cite{P})\label{chgeneric} For a generic weight $\l$, the following formula holds: 

$$chL_\l=S(\l)=\frac{D_1\cdot e^{\r}}{D_0}\cdot\sum_{w\in W} sign\,{w}\cdot w(\frac{e^{\l+\r}}{\prod_{\a\in A(\l)}(1+e^{-\a})}),$$ where $A(\l)$ is the maximal set of mutually orthogonal linearly independent real isotropic roots $\a$ such that $(\l+\r,\a)=0$. The set $A(\l)$ is one-element set for $F(4)$ and $G(3)$. 
\end{theorem}

\begin{theorem}(Penkov, \cite{P})\label{chcoh} For a finite-dimensional $\bb$-module $V$, the following formula holds: 

$$\sum_i(-1)^i ch(H^i(G/B,V^*)^*)=\frac{D_1\cdot e^{\r}}{D_0}\cdot\sum_{w\in W} sign\,{w}\cdot w(ch(V)e^{\r}),$$

\end{theorem}


We first prove the following theorem for generic weights. In later section, we establish it for all weights. 

\begin{theorem}\label{character formula generic weights} Let $\g=F(4)$ (or $G(3)$). 
Let $\l\in F^{(a,b)}$ (or $F^{a}$) be a generic dominant weight and and $\m+\r_l=a\m_1+b\m_2$ (or and $\m+\r_l=a\m_1$), then following superdimension formula holds:

$$sdim\,{L_\l}=(-1)^{p(\m)}2dim\,{L_\m(\g_x)}.$$

\end{theorem}

\begin{proof} From \rthm{chgeneric}, we have:

$$chL_{\l}=S(\l)=\sum_{\m\in S}(-1)^{p(\m)} ch{L_{\m}(\g_{\0})},$$ where $S=\{\m=\l-\sum\a | \a\in\D^+_\1 \text{ , } \a\neq\b=\frac{1}{2}(\e_1+\e_2+\e_3+\d) \}$.

Computing using definition of $sdim{V}$, the above formula and the classical Weyl character formula we get

$$sdim{L_\l}=\sum_{\m\in S} (-1)^ldim{L_\m}(\g_{\0}),$$ where $S=\{\m=\l-\sum\a | \a\in\D^+_\1 \text{ , } \a\neq\b=\frac{1}{2}(\e_1+\e_2+\e_3+\d) \}$ and $l$ is the number of roots $\a$ in the expression of $\m$. This is true since for all generic $\l$, we have $(\l+\r,\b)=0$.

Computing the formula above, using computer program (see Appendix), we have:

$$sdim\,{L_\l}=(-1)^{p(\m)}2dim\,{L_\m(\g_x)}.$$

\end{proof}

\subsection{Cohomology groups for generic weights for $F(4)$ and $G(3)$}


\begin{lemma}\label{root for each weight} For a generic weight $\l \in F^{\chi}$, there is a unique $\a\in\D^+$ such that $\l-\a \in F^{\chi}$ and $(\l+\r,\a)=0$. 
\end{lemma}

\begin{proof} From \rthm{cf4} and \rthm{cg3}, there is a unique $c$ corresponding to $\l$. Since $\l-\a$ will correspond to $c-\frac{1}{2}$ for $F(4)$ and to $c-1$ for $G(3)$, there is a unique such possible $\l-\a$. We take $\a=(\frac{1}{2},\frac{1}{2},\frac{1}{2}|\frac{1}{2})$ for $\g=F(4)$ and $\a=(1,1|1)$ for $\g=G(3)$, then it follows from \rthm{cf4} and \rthm{cg3} that $\l-\a\in F^{\chi}$. 
\end{proof}

\begin{lemma}\label{lemma 3 generic weights} Let $\l \in F^{(a, b)}$ (or $F^{a}$) be generic weight and $\a\in\D^+$ such that $\l-\a \in F^{(a, b)}$ (or $F^{a}$) and $(\l+\r,\a)=0$, then $$ [\G_0(G/B, \mathcal{O}_\l): L_{\l-\a}]\leq 1\text{ and }$$ $$[\G_0(G/B, \mathcal{O}_\l): L_{\xi}]=0\text{ if }\xi\neq\l-\a.$$ For $i>0$, we have $\G_i(G/B, \mathcal{O}_\l)=0$. 
\end{lemma}

\begin{proof} If $\l$ is a generic weight, than the only weights obtained in the form $\m+\r=w(\l+\r)-\sum\a$ are $\l$ and $\l-\a$. One can see this from \rlem{cf4} and \rlem{cg3}.

Thus, the lemma follows from \rlem{3} and \rlem{root for each weight}, since there is a unique root $\a=(\frac{1}{2},\frac{1}{2},\frac{1}{2}|\frac{1}{2})$ with $\l-\a\in F^{(a, b)}$ (or $F^{a}$). And $w=id$ is the only possibility. 
\end{proof}

\begin{lemma}\label{generic} Let $\l \in F^{(a, b)}$ (or $F^{a}$) be generic weight, then we have the exact sequence
\begin{displaymath}
    \xymatrix{
        0\ar[r] & L_{\l-\a} \ar[r] & \G_0(G/B, \mathcal{O}_\l)\ar[r] & L_\l \ar[r] & 0 }
\end{displaymath}
for $\a\in\D_1^+$ such that $(\l+\r,\a)=0$. 
\end{lemma}

\begin{proof} We know $\G_0(G/B, \mathcal{O}_\l)$ is the maximal finite dimensional quotient of the Verma module $M_{\l}$ with highest weight $\l$. Therefore, $[\G_0(G/B, \mathcal{O}_\l): L_{\l}]=1$. By \rlem{lemma 3 generic weights}, we have $[\G_0(G/B, \mathcal{O}_\l): L_{\l-\a}]\leq 1$. To prove the exact sequence, it is enough to show $[\G_0(G/B, \mathcal{O}_{\l}): L_{\l-\a}]\neq 0$.

From \rlem{lemma 3 generic weights}, we have $0=sdim {\G_0(G/B, \mathcal{O}_\l)}= sdim{L_{\l}}+[\G_0(G/B, \mathcal{O}_\l): L_{\l-\a}] sdim{L_{\l-\a}}$. From \rlem{character formula generic weights}, since $\l$ is generic we have that $sdim{L_{\l}}\neq 0$. Thus, $[\G_0(G/B, \mathcal{O}_{\l}): L_{\l-\a}]\neq 0$.

\end{proof}

\section{Equivalence of symmetric blocks in $F(4)$}

\subsection{Equivalence of blocks $\F^{(1, 1)}$ and $\F^{(2, 2)}$} 

Let $\g=F(4)$. We prove the equivalence of the symmetric blocks $\F^{(1, 1)}$ and $\F^{(2, 2)}$ as the first step of mathematical induction in $a$ of proving the equivalence of the symmetric blocks  $\F^{(a, a)}$ and $\F^{(a+1, a+1)}$.

The following is the picture of translator functor from block $\F^{(1, 1)}$ to $\F^{(2, 2)}$. It is defined by $T(L_{\l})=(L_{\l}\otimes\g)^{(2,2)}$. The non-filled circles represent the non-dominant weights in the block occurring on the walls of the Weyl chamber. The filled circles represent dominant weights in the block. The horizontal arrows are maps $\l\mapsto \l+\gamma$, with $\gamma\in\D$ is the root above the arrow. 

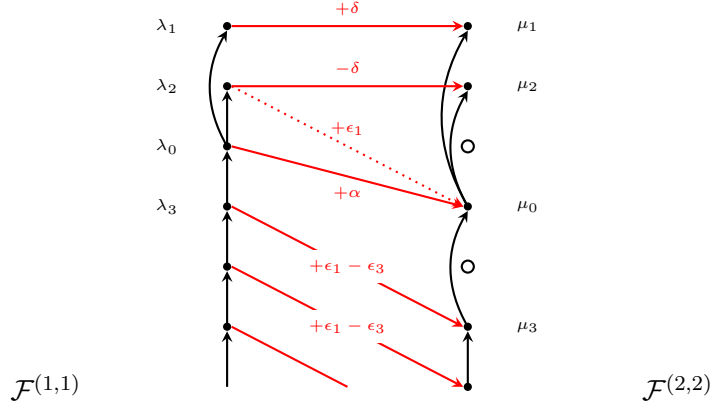
\begin{center}
  \begin{tikzpicture}[scale=.4]
  
    \draw (-2.5,0) node[anchor=east]  {$\F^{(1, 1)}$}; 
    \draw (+18.5,0) node[anchor=east]  {$\F^{(2, 2)}$}; 
    
  \draw[xshift=1 cm,thick,fill=black] (1, 12 cm) circle (.1cm);
    \draw[xshift=1 cm,thick,fill=black] (1, 10 cm) circle (.1cm);
      \draw[xshift=1 cm,thick,fill=black] (1, 8 cm) circle (.1cm);
        \draw[xshift=1 cm,thick,fill=black] (1, 6 cm) circle (.1cm);
        \draw[xshift=1 cm,thick,fill=black] (1, 4 cm) circle (.1cm);
        \draw[xshift=1 cm,thick,fill=black] (1, 2 cm) circle (.1cm);
        
   \draw[xshift=0 cm] (0, 12 cm) node[anchor=center]  {\tiny $\l_1$};
    \draw[xshift=0 cm] (0, 10 cm) node[anchor=center]  {\tiny$\l_2$};
     \draw[xshift=0 cm] (0, 8 cm) node[anchor=center]  {\tiny$\l_0$};
      \draw[xshift=0 cm] (0, 6 cm) node[anchor=center]  {\tiny$\l_3$};
      
  \draw[xshift=5 cm,thick,fill=black] (5, 12 cm) circle (.1cm);
    \draw[xshift=5 cm,thick,fill=black] (5, 10 cm) circle (.1cm);
     \draw[xshift=5 cm,thick, black] (5, 8 cm) circle (.2cm);
      \draw[xshift=5 cm,thick,fill=black] (5, 6 cm) circle (.1cm);
        \draw[xshift=5 cm,thick, black] (5, 4 cm) circle (.2cm);
         \draw[xshift=5 cm,thick,fill=black] (5, 2 cm) circle (.1cm);
                \draw[xshift=5 cm,thick,fill=black] (5, 0 cm) circle (.1cm);
        
   \draw[xshift=6 cm] (6, 12 cm) node[anchor=center]  {\tiny$\m_1$};
    \draw[xshift=6 cm] (6, 10 cm) node[anchor=center]  {\tiny$\m_2$};
        \draw[xshift=6 cm] (6, 8 cm) node[anchor=center]  {$$};
     \draw[xshift=6 cm] (6, 6 cm) node[anchor=center]  {\tiny$\m_0$};
         \draw[xshift=6 cm] (6, 4 cm) node[anchor=center]  {$$};
      \draw[xshift=6 cm] (6, 2 cm) node[anchor=center]  {\tiny$\m_3$};
      
       \foreach \y in {3,...,4}
    \draw[-stealth, xshift=1 cm,thick] (1, 2*\y cm) -- +(0, 1.85 cm);
     \foreach \y in {0,...,2}
    \draw[-stealth, xshift=1 cm,thick] (1, 2*\y cm) -- +(0, 1.85 cm);
    
    \path [-stealth, solid,black,thick,draw=none] (2, 8 cm) edge[bend left] (2, 11.85 cm);
    \path [-stealth, solid,black,thick,draw=none]  (10, 6 cm) edge[bend left] (10, 9.85 cm);
    \path [-stealth, solid,black,thick,draw=none]  (10, 2 cm) edge[bend left] (10, 5.85 cm);
    \path [-stealth, solid,black,thick,draw=none]  (10, 6 cm) edge[bend left] (10, 11.85 cm);
    \path [-stealth, solid,black,thick,draw=none]  (10, 0 cm) edge (10, 1.85 cm);
    
     \path [-stealth, solid,red,thick,draw=none]  (2.15, 12 cm) edge node[above=0mm]{\tiny $+\d$} (9.85, 12 cm) ;
    
   \path [-stealth, solid,red,thick,draw=none]  (2.15, 10 cm) edge node[above=0]{\tiny $-\d$} (9.85, 10 cm) ;
    
    \path [-stealth, solid,red,dotted,thick,draw=none]  (2.15, 10 cm) edge node[above=0.1mm,draw=none,rectangle]{\tiny $+\e_1$} (9.85, 6 cm) ;
    
    \path [-stealth, solid,red,thick,draw=none]  (2.15, 8 cm) edge node[below=0]{\tiny $+\a$} (9.85, 6 cm) ;
    
    \path [-stealth,solid,red,thick,draw=none]  (2.15, 6 cm) edge node[draw=none,fill=white,rectangle]{\tiny $+\e_1-\e_3$} (9.85, 2 cm) ;
    
     \path [-stealth,solid,red,thick,draw=none]  (2.15, 4 cm) edge node[fill=white,draw=none,rectangle]{\tiny $+\e_1-\e_3$} (9.85, 0 cm) ;
    
    \path [solid,red,thick,draw=none]  (2.15, 2 cm) edge (6, 0 cm) ;

  \end{tikzpicture}
\end{center}

In this section, we will show that the solid arrows represent the maps $L_{\l}\mapsto T(L_{\l})$ and use this to prove the equivalence of symmetric blocks.

In the above picture $\a=(\frac{1}{2},-\frac{1}{2},-\frac{1}{2}|-\frac{1}{2})$ and 
$\l_1+\r=(\frac{5}{2},\frac{3}{2},\frac{1}{2}|-\frac{3}{2})$;
$\l_2+\r=(\frac{5}{2},\frac{3}{2},\frac{1}{2}|\frac{3}{2})$;
$\l_0+\r=(3,2,1|2)$;
$\l_3+\r=(\frac{7}{2},\frac{5}{2},\frac{3}{2}|\frac{5}{2})$;
$\m_1+\r=(\frac{5}{2},\frac{3}{2},\frac{1}{2}|-\frac{1}{2})$;
$\m_2+\r=(\frac{5}{2},\frac{3}{2},\frac{1}{2}|\frac{1}{2})$;
$\m_0+\r=(\frac{7}{2},\frac{3}{2},\frac{1}{2}|\frac{3}{2})$;
$\m_3+\r=(\frac{9}{2},\frac{5}{2},\frac{1}{2}|\frac{5}{2})$. 
(Note that the indices here are different from the index c, which corresponds to the last coordinate of $\l+\r$.)

\begin{lemma}\label{k1122} Any dominant weight $\l\in F^{(1,1)}$ with $\l\neq\l_1$ and $\l_2$ can be obtained from $\l_0$ by adding root $\b=(\frac{1}{2},\frac{1}{2},\frac{1}{2}|\frac{1}{2})$ finitely many times. 

\end{lemma}

\begin{proof} From \rthm{cf4}, if $a=1$, then $J_2, J_3=\emptyset$. Since $c\neq \pm\frac{3}{2}$, we have $\l=\l_0+c\b$, where $\b=\frac{1}{2}(\e_1+\e_2+\e_3+\d)$.
\end{proof}

\begin{lemma}\label{i1122} For a dominant weight $\l\in F^{(1,1)}$ with $\l\neq\l_i$ for $i=1,2$, we have $\G_i(G/B,\O_{\l})=0$ for $i>0$. 
\end{lemma}

\begin{proof} Assume $\l\neq\l_s$ for $s=1,2$ and $\G_i(G/B,\O_{\l})\neq0$ for $i>0$. There is $\m\in F^{(1,1)}$ dominant weight such that $L_\m$ occurs in $\G_i(G/B,\O_{\l})$ with non-zero multiplicity. 

For $\l\neq\l_s$ for $s=1,2$, we have by \rlem{k1122}, $\l+\r=\l_0+\r+n\b=(3+\frac{n}{2},2+\frac{n}{2},1+\frac{n}{2}|2+\frac{n}{2})$. 
By \rlem{3}, we have $\m+\r=w(\l+\r)-\sum_{\a\in I} \a$ for $w\in W$ of length $i$. The last coordinate of $\m+\r$ is in $$[\frac{n}{2}-2, \frac{n}{2}+2]\cap \frac{1}{2}\Z_{\geq4} \text{ or }\pm\frac{3}{2}.$$

Assume $n=0$. The last coordinate of $\m+\r$ is $2$ or $\pm\frac{3}{2}$. By \rthm{cf4} and computation there are only three possibilities $\m=\l_i$ with $i=0,1,2$ and in each case $w=id$. This implies $\G_i(G/B,\O_{\l_0})=0$ for $i>0$.

Assume $n=1$. The last coordinate of $\m+\r$ is $$2, \frac{5}{2}, \pm\frac{3}{2}.$$ By computation there are only four possibilities $\m=\l_i$ with $i=0,1,2,3$ and in each case either $w=id$ or doesn't exist. This implies $\G_i(G/B,\O_{\l_3})=0$ for $i>0$.

Assume $n\geq 1$. The last coordinate of $\m+\r$ is in $$[\frac{n}{2}-2, \frac{n}{2}+2]\cap \frac{1}{2}\Z_{\geq4}.$$ By computation, only $w=id$ is possible when $\m+\r$ has last coordinate equal the last coordinate of $\l+\r$ minus $\frac{1}{2}$. Thus, $\G_i(G/B,\O_{\l})=0$ for $i>0$.
\end{proof}

\begin{lemma}\label{coh011} For a dominant weight $\l\in F^{(1,1)}$ with $\l\neq\l_i$ for $i=0,1,2$, we have $[\G_0(G/B,\O_{\l}):L_{\l-\a}]=1$ for a unique $\a\in\D$ such that $\l-\a\in F^{(1,1)}$.

Also, we have $[\G_0(G/B,\O_{\l}):L_{\m}]=0$ for $\m\neq \l$ and $\m\neq\l-\a$.

\end{lemma}

\begin{proof} As in the previous lemma, by \rlem{k1122}, $\l+\r=\l_0+\r+n\b=(3+\frac{n}{2},2+\frac{n}{2},1+\frac{n}{2}|2+\frac{n}{2})$.

The first part of the lemma follows from \rlem{lemma 3 generic weights}. Assume $n=0$. The last coordinate of $\m+\r$ is $2$ or $\pm\frac{3}{2}$. By computation, there are only three possibilities $\m=\l_i$ with $i=0,1,2$ and in each case $w=id$. This implies $[\G_0(G/B,\O_{\l_0}),L_{\m}]=0$ for $\m\neq \l_0-\a$.

Assume $n=1$. The last coordinate of $\m+\r$ is $2, \frac{5}{2}, \pm\frac{3}{2}$. By computation there are only four possibilities $\m=\l_i$ with $i=0,1,2,3$. For $i=0,2$, there is unique possible $w=id$ and set $I$. This implies $[\G_0(G/B,\O_{\l_3}),L_{\m}]=0$ for $\m\neq \l_3-\a$.

Assume $n> 1$. The last coordinate of $\m+\r$ is in $[\frac{n}{2}-2, \frac{n}{2}+2]\cap \frac{1}{2}\Z_{\geq4}$. By computation and \rlem{cf4}, only $w=id$ is possible when $\m+\r$ has last coordinate equal the last coordinate of $\l+\r$ minus $\frac{1}{2}$ or $\m=\l$, in each case there is a unique set $I$. Thus, $[\G_0(G/B,\O_{\l}):L_{\m}]=0$ for $\m\neq \l$ and $\m\neq\l-\a$ for any $\a\in\D_{\1}$.
\end{proof}

\begin{lemma} For a dominant weight $\l\in F^{(1,1)}$, we have $sdim{L_{\l}}=\pm 2$ if $\l\neq\l_i$ for $i=1,2$. 
\end{lemma}

\begin{proof} We prove this by induction starting with a generic weight $\l\in F^{(1,1)}$. From generic formula for superdimension, we have $sdim{L_{\l}}=a$ with $a=\pm 2$. The weights in $F^{(1,1)}$ can be obtained successively from $\l$ by subtracting odd root $\b$ from \rlem{k1122}.

By \rlem{7} and \rlem{i1122}, we have $$0=sdim{\G_0(G/B,\O_{\l})}=sdim{L_{\l}}+[\G_0(G/B,\O_{\l}):L_{\l-\a}]sdim{L_{\l-\a}}.$$ Since $sdim{L_{\l}}=\pm 2$ and $[\G_0(G/B,\O_{\l}):L_{\l-\a}]\leq 1$ from proof of previous lemma, we must have $[\G_0(G/B,\O_{\l}):L_{\l-\a}]=1$ and $sdim{L_{\l-\a}}=\mp 2$. By induction, this way from generic weight we obtain $L_{\l_0}$. Thus, $sdim{L_{\l_0}}=\pm 2$. 
\end{proof}

\begin{lemma}\label{0l1} We have $\G_0(G/B,\O_{\l_1})=L_{\l_1}$. 
\end{lemma}

\begin{proof} From \rlem{3} and \rthm{cf4}, if $L_{\sigma}$ occurs in $\G_0(G/B,\O_{\l_1})$, then $\m\leq \l$. Thus, $[\G_0(G/B,\O_{\l_1}):L_{\sigma}]=0$ for $\sigma\neq \l_1$.

We know $[\G_0(G/B,\O_{\l_1}):L_{\l_1}]=1$ from \rlem{z}. 
\end{proof}

\begin{lemma}\label{1l1} We have $\G_1(G/B,\O_{\l_1})=L_{\l_2}$. 
\end{lemma}

\begin{proof} We have $$0=sdim{\G_0(G/B,\O_{\l_1})}-sdim{\G_1(G/B,\O_{\l_1})}$$ and $$sdim{\G_0(G/B,\O_{\l_1})}=sdim{L_{\l_1}}=1.$$ 

This implies that $sdim{\G_1(G/B,\O_{\l_1})}=1$. Thus, we either have $\G_1(G/B,\O_{\l_1})=L_{\l_1}$ or  $\G_1(G/B,\O_{\l_1})=L_{\l_2}$. This is true since $\G_1(G/B,\O_{\l_1}):L_{\sigma}]=0$ for all $\s\neq\l_1,\l_2$.

We have $$ch{\G_0(G/B,\O_{\l_1})}-ch{\G_1(G/B,\O_{\l_1})}=\frac{D_1e^{\r}}{D_0}\sum_{w\in W}sgn(w)e^{w(\l_1+\r)}.$$ The expression on the right is not zero, since one can compute that the lowest degree term in the numerator is not zero. This implies $\G_0(G/B,\O_{\l_1})\neq \G_1(G/B,\O_{\l_1})$. Thus, $\G_1(G/B,\O_{\l_1})=L_{\l_2}$. 

\end{proof}

\begin{lemma} We have $sdimL_{\l_1}=sdimL_{\l_2}=1$. 
\end{lemma}

\begin{proof} This follows from previous two lemmas and since

$$sdim{\G_0(G/B,\O_{\l_1})} =sdim{\G_1(G/B,\O_{\l_1})}.$$ 
\end{proof}

\begin{lemma}\label{coho0} The cohomology group $\G_0(G/B,\O_{\l_0})$ has a filtration with quotients $L_{\l_0}$, $L_{\l_1}$, and $L_{\l_2}$.We know that $L_{\l_0}$ is a quotient of $\G_0(G/B,\O_{\l_0})$. The kernel of that quotient has a filtration with subquotients $L_{\l_1}$, $L_{\l_2}$. Also, $sdim{L_{\l_0}}=-2$.
\end{lemma}

\begin{proof} From previous lemmas, we have $sdim{L_{\l_0}}=\pm 2$, $sdim{L_{\l_1}}=sdim{L_{\l_2}}=1$. We also know from \rlem{coh011}, $[\G_0(G/B,\O_{\l_0}): L_\sigma]=0$, unless $\sigma=\l_i$ with $i=0, 1, 2$. From \rlem{3}, we have $[\G_0(G/B,\O_{\l_0}): L_{\l_0}]=1$, $[\G_0(G/B,\O_{\l_0}): L_{\l_1}]\leq1$, $[\G_0(G/B,\O_{\l_0}): L_{\l_2}]\leq1$.

We have $$0=sdim{\G_0(G/B,\O_{\l_0})}=sdim{L_{\l_0}}+[\G_0(G/B,\O_{\l_0}): L_{\l_1}]sdim{L_{\l_1}}+$$ $$+[\G_0(G/B,\O_{\l_0}): L_{\l_2}]sdim{L_{\l_2}}.$$ This implies that $[\G_0(G/B,\O_{\l_0}): L_{\l_1}]=[\G_0(G/B,\O_{\l_0}): L_{\l_2}]=1$, and $sdim{L_{\l_0}}=-2$.
\end{proof}

\begin{lemma} We have $\G_0(G/B,\O_{\l_2})=L_{\l_2}$ and $\G_1(G/B,\O_{\l_2})=L_{\l_1}$. 
\end{lemma}

\begin{proof} From \rlem{3}, we have $[\G_0(G/B,\O_{\l_2}):L_{\sigma}]=0$ for $\sigma\neq \l_i$ with $i=1,2$. We know $[\G_0(G/B,\O_{\l_2}):L_{\l_2}]=1$ from \rlem{z}. We need to show $[\G_0(G/B,\O_{\l_2}):L_{\l_1}]=0$.

From \rlem{sums}, since $\l_2=w(\l_1+\r)-\r$, with $w$ reflection with respect to root $\d$, we have $$ch{\G_0(G/B,\O_{\l_1})}-ch{\G_1(G/B,\O_{\l_1})}=-ch{\G_0(G/B,\O_{\l_2})}+ch{\G_1(G/B,\O_{\l_2})}.$$

From \rlem{0l1}, we have $\G_0(G/B,\O_{\l_1})=L_{\l_1}$. From \rlem{1l1}, we have $\G_1(G/B,\O_{\l_1})=L_{\l_2}$. From \rlem{3}, we know that $[\G_1(G/B,\O_{\l_2}):L_{\l_2}]=0$. We also know that $[\G_0(G/B,\O_{\l_2}):L_{\l_2}]=1$. The above equation gives $$[\G_1(G/B,\O_{\l_2}):L_{\l_1}]-[\G_0(G/B,\O_{\l_2}):L_{\l_1}]=1.$$

We show that $\G_1(G/B,\O_{\l_2})=L_{\l_1}$, which together with previous equality implies $[\G_0(G/B,\O_{\l_2}):L_{\l_1}]=0$ and proves the lemma.

Consider the typical weight $\m$, with $\m+\r=(3,2,1|1)$. The module $(L_{\m}\otimes\g)^{(1,1)}$ has a filtration with quotients $\O_\l$ with $\l=\l_i$ with $i=0,2$. As $\l_2<\l_0$, we have an exact sequence:

$$0\rightarrow\O_{\l_0}\rightarrow (\O_{\m}\otimes\g)^{\Phi^{-1}(\chi)}\rightarrow\O_{\l_2}\rightarrow 0.$$

Applying \rlem{1}, gives the following long exact sequence: 

$$0\rightarrow\G_1(G/B,\O_{\l_2})\rightarrow\G_0(G/B,\O_{\l_0})\rightarrow (L_{\m}\otimes\g)^{\chi}\rightarrow\G_0(G/B,\O_{\l_2})\rightarrow 0.$$

From previous lemma, we have $[\G_0(G/B,\O_{\l_0}):L_{\l_1}]=1$. From the long exact sequence we have $[\G_1(G/B,\O_{\l_2}):L_{\l_1}]\leq [\G_0(G/B,\O_{\l_0}):L_{\l_1}]=1$. Since $sdim\G_1(G/B,\O_{\l_2})=sdim\G_0(G/B,\O_{\l_2})\neq0$, we have $[\G_1(G/B,\O_{\l_2}):L_{\l_1}]\neq 0$. This proves the lemma. 
\end{proof}

\begin{lemma}\label{x1122} We have $T(L_{\l_i})=L_{\m_i}$, for all $i\neq 2$.
\end{lemma}

\begin{proof} By definition, $T(L_{\l_i})=(L_{\l_i}\otimes\g)^{(2,2)}$. For each $i\neq 2$, there is a unique dominant weight $\m_i$ in the block $\F^{(2, 2)}$ of the form $\l_i+\gg$ with $\gg\in \D$ as its shown in the picture. Thus, the lemma follows from \rlem{acyclicx}.

\end{proof}

\begin{lemma}\label{special1122} We have $T(L_{\l_2})=L_{\m_2}$.
\end{lemma}
 
\begin{proof} By definition, $T(L_{\l_2})=(L_{\l_2}\otimes\g)^{(2,2)}$. The only dominant weights in $F^{(2, 2)}$ of the form $\l_2+\gg$ with $\gg\in \D$ are $\m_2$ and $\m_0$.

It suffices to prove that $T(L_{\l_2})$ does not have a subquotient $L_{\m_0}$.


We know that $L_{\l_0}$ is a quotient of $\G_0(G/B,\O_{\l_0})$ from \rlem{z}. The kernel of that quotient has a filtration with subquotients $L_{\l_1}$, $L_{\l_2}$ (see \rlem{coho0}). We have the following exact sequence: 

$$0\rightarrow S\rightarrow \G_0(G/B,\O_{\l_0})\rightarrow L_{\l_0}\rightarrow 0.$$

Since $T$ is an exact functor, we get the following exact sequence: 

$$0\rightarrow T(S)\rightarrow T(\G_0(G/B,\O_{\l_0}))\rightarrow T(L_{\l_0})\rightarrow 0.$$

From \rlem{x1122}, we have $T(L_{\l_0})=L_{\m_0}$. The kernel $T(S)$ of that quotient has a filtration with subquotients $T(L_{\l_1})$, $T(L_{\l_2})$. By \rlem{6} and \rlem{w}, we have $T(\G_0(G/B,\O_{\l_0}))=\G_0(G/B,T(\O_{\l_0}))=\G_0(G/B,\O_{\m_0})$. The later module has a unique quotient $L_{\m_0}$. Therefore, $T(S)$ has no simple subquotient $L_{\m_0}$. Hence, $T(L_{\l_2})$ also does not have a subquotient $L_{\m_0}$.



\end{proof} 

\begin{corollary} For any $\l\in F^{(1,1)}$, the module $T(L_{\l})\in F^{(2,2)}$ is irreducible of highest weight $\l+\a$ for some $\a\in\D$. Conversely, any irreducible module in $F^{(2,2)}$ is obtained this way. 
\end{corollary}

\begin{proof} For any dominant weight $\l\in F^{(1,1)}$, with $\l\neq \l_2$, there is a unique $\a\in\D$ with dominant weight $\l+\a\in F^{(2,2)}$. Thus, $T(L_{\l})$ is an irreducible with highest weight $\l+\a$. From previous lemma, the corollary follows. 



\end{proof}

\begin{theorem}\label{eq1122} The blocks $\F^{(1, 1)}$ and $\F^{(2, 2)}$ are equivalent as categories. 
\end{theorem}

\begin{proof} From above corollary, for each $\l_i\in F^{(1,1)}$, let $L_{\m_i}=T(L_{\l_i})$ be the simple module with highest weight $\m_i\in F^{(2,2)}$. We show that $T^*(L_{\m_i})=L_{\l_i}$ for each $\m_i\in F^{(2,2)}$.

For all $\m\neq \m_0$, we have a unique $\gamma\in\D$, such that $\m+\gamma\in F^{(1, 1)}$. For $\m=\m_0$, there are two possible $\gamma\in\D$ such that $\m_0+\gamma\in F^{(1, 1)}$. From the picture above, we have $\ga=-(\frac{1}{2},-\frac{1}{2},-\frac{1}{2}|-\frac{1}{2})$ or $\ga=-\e_1$, such that $\m_0+\gamma=\l_0$ or $\l_1$. The theorem follows from \rthm{eqgeneral}.

 \end{proof}
 
\subsection{Equivalence of blocks $\F^{(a, a)}$ and $\F^{(a+1, a+1)}$} 

In this section, we prove the inductive step of the equivalence of all the symmetric blocks. 
Let $V$ be a finite-dimensional $\g$-module. We define translator functor $T(V)_{\chi,\tau}: F_{\chi}\rightarrow F_{\tau}$ by $T(V )_{\chi,\tau} (M) =(M\otimes V)^{\tau}$ as before.

The following is the picture of translator functor from block $\F^{(a,a)}$ to $\F^{(a+1, a+1)}$. It is defined by $T(L_{\l})=(L_{\l}\otimes\g)^{(a+1, a+1)}$. The non-filled circles represent the non-dominant weights in the block occurring on the walls of the Weyl chamber. The filled circles represent dominant weights in the block. The horizontal arrows are maps $\l\mapsto \l+\gamma$, with $\gamma\in\D$ is the root above the arrow. In this section, we will show that the solid arrows represent the maps $L_{\l}\mapsto T(L_{\l})$.

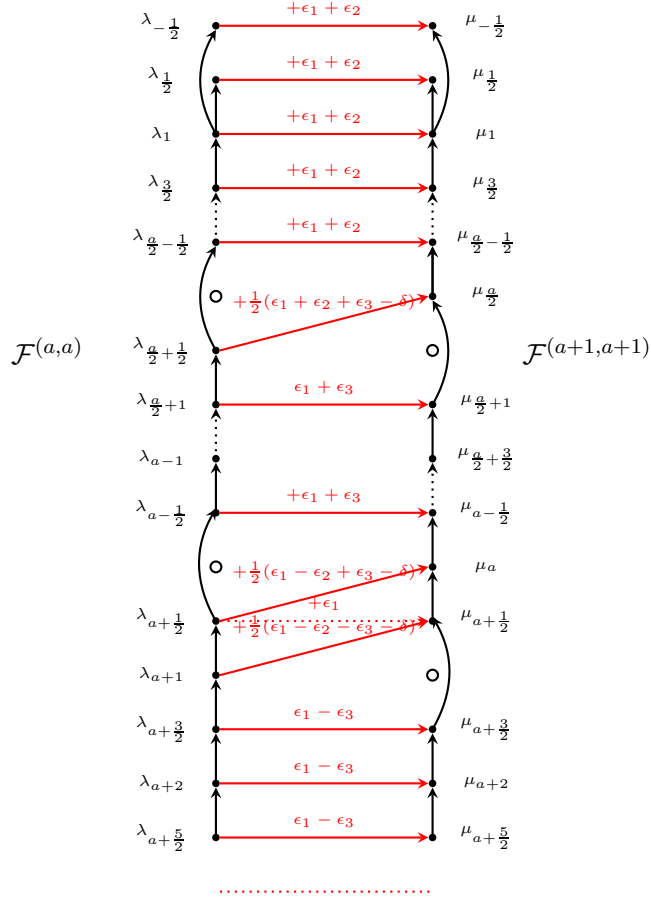
\begin{center}
  \begin{tikzpicture}[scale=.36]
  
    \draw (-2.5,0) node[anchor=east]  {$\F^{(a, a)}$}; 
    \draw (+18.5,0) node[anchor=east]  {$\F^{(a+1, a+1)}$}; 
    
  \draw[xshift=1 cm,thick,fill=black] (1, 12 cm) circle (.1cm);
    \draw[xshift=1 cm,thick,fill=black] (1, 10 cm) circle (.1cm);
      \draw[xshift=1 cm,thick,fill=black] (1, 8 cm) circle (.1cm);
        \draw[xshift=1 cm,thick,fill=black] (1, 6 cm) circle (.1cm);
        \draw[xshift=1 cm,thick,fill=black] (1, 4 cm) circle (.1cm);
        \draw[xshift=1 cm,thick,black] (1, 2 cm) circle (.2cm);
                \draw[xshift=1 cm,thick,fill=black] (1, 0 cm) circle (.1cm);
        \draw[xshift=1 cm,thick,fill=black] (1, -2 cm) circle (.1cm);
        \draw[xshift=1 cm,thick,fill=black] (1, -4 cm) circle (.1cm);
        \draw[xshift=1 cm,thick,fill=black] (1, -6 cm) circle (.1cm);
        \draw[xshift=1 cm,thick,black] (1, -8 cm) circle (.2cm);
        \draw[xshift=1 cm,thick,fill=black] (1, -10 cm) circle (.1cm);
        \draw[xshift=1 cm,thick,fill=black] (1, -12 cm) circle (.1cm);
        \draw[xshift=1 cm,thick,fill=black] (1, -14 cm) circle (.1cm);
        \draw[xshift=1 cm,thick,fill=black] (1, -16 cm) circle (.1cm);
        \draw[xshift=1 cm,thick,fill=black] (1, -18 cm) circle (.1cm);

   \draw[xshift=0 cm] (0, 12 cm) node[anchor=center]  {{\tiny $\l_{-\frac{1}{2}}$}};
    \draw[xshift=0 cm] (0, 10 cm) node[anchor=center]  {{\tiny$\l_{\frac{1}{2}}$}};
     \draw[xshift=0 cm] (0, 8 cm) node[anchor=center]  {{\tiny$\l_{1}$}};
      \draw[xshift=0 cm] (0, 6 cm) node[anchor=center]  {{\tiny$\l_{\frac{3}{2}}$}};
   \draw[xshift=0 cm] (0, 4 cm) node[anchor=center]  {{\tiny$\l_{\frac{a}{2}-\frac{1}{2}}$}};
    \draw[xshift=0 cm] (0, 0 cm) node[anchor=center]  {{\tiny$\l_{\frac{a}{2}+\frac{1}{2}}$}};
     \draw[xshift=0 cm] (0, -2 cm) node[anchor=center]  {{\tiny$\l_{\frac{a}{2}+1}$}};
         \draw[xshift=0 cm] (0, -4 cm) node[anchor=center]  {{\tiny$\l_{a-1}$}};
    \draw[xshift=0 cm] (0, -6 cm) node[anchor=center]  {{\tiny$\l_{a-\frac{1}{2}}$}};
      \draw[xshift=0 cm] (0, -10 cm) node[anchor=center]  {{\tiny$\l_{a+\frac{1}{2}}$}};
         \draw[xshift=0 cm] (0, -12 cm) node[anchor=center]  {{\tiny$\l_{a+1}$}};
    \draw[xshift=0 cm] (0, -14 cm) node[anchor=center]  {{\tiny$\l_{a+\frac{3}{2}}$}};
     \draw[xshift=0 cm] (0, -16 cm) node[anchor=center]  {{\tiny$\l_{a+2}$}};
      \draw[xshift=0 cm] (0, -18 cm) node[anchor=center]  {{\tiny$\l_{a+\frac{5}{2}}$}};
            
  \draw[xshift=5 cm,thick,fill=black] (5, 12 cm) circle (.1cm);
    \draw[xshift=5 cm,thick,fill=black] (5, 10 cm) circle (.1cm);
     \draw[xshift=5 cm,thick, fill=black] (5, 8 cm) circle (.1cm);
      \draw[xshift=5 cm,thick,fill=black] (5, 6 cm) circle (.1cm);
        \draw[xshift=5 cm,thick, fill=black] (5, 4 cm) circle (.1cm);
         \draw[xshift=5 cm,thick,fill=black] (5, 2 cm) circle (.1cm);
                \draw[xshift=5 cm,thick,black] (5, 0 cm) circle (.2cm);
                                \draw[xshift=5 cm,thick,fill=black] (5, -2 cm) circle (.1cm);
                \draw[xshift=5 cm,thick,fill=black] (5, -4 cm) circle (.1cm);
                \draw[xshift=5 cm,thick,fill=black] (5, -6 cm) circle (.1cm);
                \draw[xshift=5 cm,thick,fill=black] (5, -8 cm) circle (.1cm);

                \draw[xshift=5 cm,thick,fill=black] (5, -10 cm) circle (.1cm);
                \draw[xshift=5 cm,thick,black] (5, -12 cm) circle (.2cm);
                \draw[xshift=5 cm,thick,fill=black] (5, -14 cm) circle (.1cm);
                \draw[xshift=5 cm,thick,fill=black] (5, -16 cm) circle (.1cm);
                \draw[xshift=5 cm,thick,fill=black] (5, -18 cm) circle (.1cm);

   \draw[xshift=6 cm] (6, 12 cm) node[anchor=center]  {{\tiny $\m_{-\frac{1}{2}}$}};
    \draw[xshift=6 cm] (6, 10 cm) node[anchor=center]  {{\tiny $\m_{\frac{1}{2}}$}};
        \draw[xshift=6 cm] (6, 8 cm) node[anchor=center]  {{\tiny $\m_{1}$}};
     \draw[xshift=6 cm] (6, 6 cm) node[anchor=center]  {{\tiny $\m_{\frac{3}{2}}$}};
         \draw[xshift=6 cm] (6, 4 cm) node[anchor=center]  {{\tiny $\m_{\frac{a}{2}-\frac{1}{2}}$}};
      \draw[xshift=6 cm] (6, 2 cm) node[anchor=center]  {{\tiny $\m_{\frac{a}{2}}$}};
            \draw[xshift=6 cm] (6, 0 cm) node[anchor=center]  {{\tiny $$}};
                  \draw[xshift=6 cm] (6, -2 cm) node[anchor=center]  {{\tiny $\m_{\frac{a}{2}+1}$}};
      \draw[xshift=6 cm] (6, -4 cm) node[anchor=center]  {{\tiny $\m_{\frac{a}{2}+\frac{3}{2}}$}};
      \draw[xshift=6 cm] (6, -6 cm) node[anchor=center]  {{\tiny $\m_{a-\frac{1}{2}}$}};
      \draw[xshift=6 cm] (6, -8 cm) node[anchor=center]  {{\tiny $\m_{a}$}};
      \draw[xshift=6 cm] (6, -10 cm) node[anchor=center]  {{\tiny $\m_{a+\frac{1}{2}}$}};
      \draw[xshift=6 cm] (6, -14 cm) node[anchor=center]  {{\tiny $\m_{a+\frac{3}{2}}$}};
      \draw[xshift=6 cm] (6, -16 cm) node[anchor=center]  {{\tiny $\m_{a+2}$}};
      \draw[xshift=6 cm] (6, -18 cm) node[anchor=center]  {{\tiny $\m_{a+\frac{5}{2}}$}};

       \foreach \y in {-3,-1}
    \draw[-stealth, xshift=1 cm,thick] (1, 2*\y cm) -- +(0, 1.85 cm);  
       \foreach \y in {-9,...,-6}
    \draw[-stealth, xshift=1 cm,thick] (1, 2*\y cm) -- +(0, 1.85 cm);            
       \foreach \y in {3,...,4}
    \draw[-stealth, xshift=1 cm,thick] (1, 2*\y cm) -- +(0, 1.85 cm);
     \foreach \y in {-2,2}
    \draw[-stealth, xshift=1 cm,dotted,thick] (1, 2*\y cm) -- +(0, 1.85 cm);
    
          \foreach \y in {-2,-5,-4,-9,-8}
    \draw[-stealth, xshift=9 cm,thick] (1, 2*\y cm) -- +(0, 1.85 cm);
     \foreach \y in {-3}
    \draw[-stealth, xshift=9 cm,dotted,thick] (1, 2*\y cm) -- +(0, 1.85 cm);

    \path [-stealth, solid,black,thick,draw=none] (2, 8 cm) edge[bend left] (2, 11.85 cm);
    \path [-stealth, solid,black,thick,draw=none]  (10, 8 cm) edge (10, 9.85 cm);
    \path [-stealth, solid,black,thick,draw=none]  (10, -2 cm) edge[bend right] (10, 1.85 cm);
    \path [-stealth, solid,black,thick,draw=none]  (10, 8 cm) edge[bend right] (10, 11.85 cm);
        \path [-stealth, solid,black,thick,draw=none]  (10, -14 cm) edge[bend right] (10, -9.85 cm);

        \path [-stealth, solid,black,thick,draw=none]  (2, 0 cm) edge[bend left] (2, 3.85 cm);
        \path [-stealth, solid,black,thick,draw=none]  (2, -10 cm) edge[bend left] (2, -5.85 cm);

    \path [-stealth, solid,black,thick,draw=none]  (10, 2 cm) edge (10, 3.85 cm);
        \path [-stealth, solid,black,dotted,thick,draw=none]  (10, 4 cm) edge (10, 5.85 cm);
    \path [-stealth, solid,black,thick,draw=none]  (10, 6 cm) edge (10, 7.85 cm);
    \path [-stealth, solid,black,thick,draw=none]  (10, 2 cm) edge (10, 3.85 cm);

     \path [-stealth, solid,red,thick,draw=none]  (2.15, 12 cm) edge node[above=0mm]{\tiny $+\e_1+\e_2$} (9.85, 12 cm) ;
    
   \path [-stealth, solid,red,thick,draw=none]  (2.15, 10 cm) edge node[above=0]{\tiny $+\e_1+\e_2$} (9.85, 10 cm) ;

    \path [-stealth, solid,red,thick,draw=none]  (2.15, 8 cm) edge node[above=0mm]{\tiny $+\e_1+\e_2$} (9.85, 8 cm) ;
    
    \path [-stealth,solid,red,thick,draw=none]  (2.15, 6 cm) edge node[draw=none,rectangle,above=0mm]{\tiny $+\e_1+\e_2$} (9.85, 6 cm) ;
        \path [-stealth,solid,red,thick,draw=none]  (2.15, 4 cm) edge node[draw=none,rectangle,above=0mm]{\tiny $+\e_1+\e_2$} (9.85, 4 cm) ;
           \path [-stealth,solid,red,thick,draw=none]  (2.15, 0 cm) edge node[draw=none,rectangle,above=0mm]{\tiny $+\frac{1}{2}(\e_1+\e_2+\e_3-\d)$} (9.85, 2 cm) ;
            \path [-stealth,solid,red,thick,draw=none]  (2.15, -2 cm) edge node[draw=none,rectangle,above=0mm]{\tiny $\e_1+\e_3$} (9.85, -2 cm) ;
            \path [-stealth,solid,red,thick,draw=none]  (2.15, -10 cm) edge node[draw=none,rectangle,above=0mm]{\tiny $+\frac{1}{2}(\e_1-\e_2+\e_3-\d)$} (9.85, -8 cm) ;
            
 \path [-stealth,solid,red,thick,draw=none]  (2.15, -12 cm) edge node[draw=none,rectangle,above=0mm]{\tiny $+\frac{1}{2}(\e_1-\e_2-\e_3-\d)$} (9.85, -10 cm) ;
     \path [-stealth,solid,red,thick,draw=none]  (2.15, -14 cm) edge node[draw=none,rectangle,above=0mm]{\tiny $\e_1-\e_3$} (9.85, -14 cm) ;
      \path [-stealth,solid,red,thick,draw=none]  (2.15, -16 cm) edge node[draw=none,rectangle,above=0mm]{\tiny $\e_1-\e_3$} (9.85, -16 cm) ;
        \path [-stealth,solid,red,thick,draw=none]  (2.15, -18 cm) edge node[draw=none,rectangle,above=0mm]{\tiny $\e_1-\e_3$} (9.85, -18 cm) ;
      \path [-stealth,solid,red,thick,draw=none]  (2.15, -6 cm) edge node[draw=none,rectangle,above=0mm]{\tiny $+\e_1+\e_3$} (9.85, -6 cm) ;
    
      \path [-stealth, solid,red,dotted,thick,draw=none]  (2.15, -10 cm) edge node[above=0.1mm,draw=none,rectangle]{\tiny $+\e_1$} (9.85, -10 cm) ;
    
    
    \path [solid,red,thick,dotted,draw=none,above=0mm]  (2.15, -20 cm) edge (10, -20 cm) ;

  \end{tikzpicture}
\end{center}

\begin{lemma}\label{aa1coho} For $\l\in F^{(a,a)}$, let $T$ be an equivalence of categories $\F^{(a,a)}$ and $\F^{(a+1,a+1)}$ and $T(L_\l)=L_{\l'}$, then $\G_i(G/B, \O_{\l'})$ has a subquotients $L_{\l'_s}$ with  $$[\G_i(G/B, \O_{\l'}):L_{\l'_s}]=[\G_i(G/B, \O_\l):L_{\l_s}].$$  
\end{lemma}

\begin{proof} Assume $i=0$. Then $\G_0(G/B, \O_{\l'})=T(\G_0(G/B, \O_\l)$ from \rlem{eqcoho}. Assume $i>0$. For $\l\neq\l_t$ with $t=1,2$, we have $\G_i(G/B, \O_{\l})=0$ for $i>0$ from computation using \rlem{3}. For $t=1,2$, we know from \rlem{eqcoho}, $\G_0(G/B, \O_{\l_t})=L_{\l_t}$ since all other submodules in $\G_0(G/B, \O_{\l_t})$ have highest weight $< \l_t$ and this is impossible. Thus, we have $sdim\G_1(G/B, \O_{\l_t})=sdim\G_0(G/B, \O_{\l_t})=sdimL_{\l_t}$.

For $s\neq 1,2$, $sdimL_{\l_s} > sdimL_{\l_1}$, which implies $\G_1(G/B, \O_{\l_t})=L_{\l_k}$ for $t,k=1,2$. We have $$ch{\G_0(G/B,\O_{\l_i})}-ch{\G_1(G/B,\O_{\l_i})}=\frac{D_1e^{\r}}{D_0}\sum_{w\in W}sgn(w)e^{w(\l_i+\r)}.$$ The expression on the right is not zero, since one can compute that the lowest degree term in the numerator is not zero. Thus, $ch\G_1(G/B, \O_{\l_i})\neq ch\G_0(G/B, \O_{\l_i})$ and we must have $\G_1(G/B, \O_{\l_i})=L_{\l_s}$ with $s\neq i$. This proves the lemma. 

\end{proof}

\begin{lemma} Let $\l\in F^{(a,a)}$ be dominant, then there is unique
$\ga\in\Delta$ such that $\l+\ga\in\F^{(a+1,a+1)}$ is dominant,
unless $\l+\r=(2a+\frac{1}{2},a+\frac{1}{2},\frac{1}{2}|a+\frac{1}{2})$.
\end{lemma}

\begin{proof} From \rlem{cf4}, for given $c\geq-\frac{1}{2}$, there is at most one dominant
$\l\in F^{(a,a)}$ with $\l+\r=(b_1,b_2,b_3|c)$. Assume $\ga\in\D$ is
such that $\l+\ga\in F^{(a+1,a+1)}$, then $\l+\r+\ga$ must have last
coordinate $c\pm 1$, $c\pm\frac{1}{2}$, or $c$.

Thus in generic cases, the last coordinate of $\l+\ga+\r$ and $\l+\rho$ are in the same interval $J_i$. The few exceptional cases, when the last coordinates are in the distinct intervals, occur around walls of the Weyl chamber, when $c=a+\frac{1}{2}$, $a+1$, $\frac{a}{2}+\frac{1}{2}$, $\frac{a}{2}+1$. And only for $c=a+\frac{1}{2}$, there are two possible $\ga$.

We show that the last coordinates of $\l+\ga+\r$ and $\l+\r$ are the same in
generic cases, and thus, there is at most one such $\ga$, proving the
uniqueness.

Note that for generic $\l$, $(\l+\r,\a)=0$ and $(\l+\ga+\r,\a)=0$ are true
for the same $\a\in\D_{\bar{1}}^+$ (see \rrem{remark} above). That implies
$(\ga,\a)=0$. This is impossible for $\ga=\d$. If $\ga$ is odd then $(\ga,\a)=0$
implies $\ga=\pm\a$, which is impossible for $\l$ and $\l+\ga$ would be in
the same block. For even root $\ga\neq\d$ the statement is clear.

For the existence, for each $\l_c$ the root $\ga$ described in the picture above above each arrow. 
\end{proof}

\begin{lemma}\label{aa1x} We have $T(L_{\l_i})=L_{\l_i+\gamma}$, for all $i\neq a+\frac{1}{2}$ and for the unique $\gamma\in\D$ in the previous lemma.
\end{lemma}

\begin{proof} By definition, $T(L_{\l_i})=(L_{\l_i}\otimes\g)^{(a+1,a+1)}$. For each $\l_i$, there is a unique dominant weight $\m_i$ in the block $\F^{(a+1, a+1)}$ of the form $\l_i+\gg$ with $\gg\in \D$. Thus, the lemma follows from \rlem{acyclicx}.

\end{proof}

\begin {lemma}\label{eqaa1} Assume for each $\l\in F^{(a,a)}$, $T(L_{\l})$ is a simple module in $\F^{(a+1,a+1)}$. Then categories  $\F^{(a,a)}$ and $\F^{(a+1,a+1)}$ are equivalent. 
\end {lemma}

\begin{proof} By hypothesis, for each $\l_i\in F^{(a,a)}$, $T(L_{\l_i})$ is a simple module in $\F^{(a+1,a+1)}$, we denote $L_{\m_i}=T(L_{\l_i})$ the simple module with highest weight $\m_i\in F^{(a+1,a+1)}$. We show that $T^*(L_{\m_i})=L_{\l_i}$ for each $\m_i\in F^{(a+1,a+1)}$.

For all $\m\neq \m_{a+\frac{1}{2}}$, we have a unique $\gamma\in\D$, such that $\m+\gamma\in F^{(a,a)}$. For $\m=\m_{a+\frac{1}{2}}$, there are two possible $\gamma\in\D$ such that $\m+\gamma\in F^{(a,a)}$. From the picture picture above, we have $\ga=-(\frac{1}{2},-\frac{1}{2},-\frac{1}{2}|-\frac{1}{2})$ or $\ga=-\e_1$, such that $\m_{a+\frac{1}{2}}+\gamma=\l_{a+\frac{1}{2}}$ or $\l_{a+1}$. The statement follows from \rthm{eqgeneral}








 \end{proof}

\begin{lemma}\label{aa1} Let $\g=F(4)$ and $\l\in F^{(a,a)}$ such that $\l=(2a+\frac{1}{2},a+\frac{1}{2},\frac{1}{2}|a+\frac{1}{2})-\r$.
If $a=1$, let $\a=\d$, and if $a>1$, let $\a=(-\frac{1}{2},\frac{1}{2},-\frac{1}{2}|\frac{1}{2})$. Then $T(L_{\l})=L_{\l-\a}$. 

\end{lemma}

\begin{proof} We will assume that blocks $\F^{(c,c)}$ for $c\leq a$ are all equivalent. Then using this assumption we will prove the lemma. This lemma implies the equivalence of $\F^{(a,a)}$ and $\F^{(a+1,a+1)}$. Thus, we use a complicated induction in $a$.

For $a=1$, we have the statement from \rlem{special1122}. Let $a>1$. From our assumption and \rlem{aa1coho}, we obtain all cohomology groups for $\F^{(a,a)}$, since we know them for $\F^{(1,1)}$ from previous section.

From definition, we have $\l=\l_{a+\frac{1}{2}}$ and $T(L_{\l_{a+\frac{1}{2}}})=(L_{\l_{a+\frac{1}{2}}}\otimes\g)^{(a+1,a+1)}$. Thus, the only dominant weights in $\F^{(a+1,a+1)}$ of the form $\l_{a+\frac{1}{2}}+\gg$ with $\gg\in \D$ are $\m_{a+\frac{1}{2}}$ and $\m_a$ as its shown in the picture.

It will suffice to prove that $T(L_{\l_{a+\frac{1}{2}}})$ does not have a subquotient $L_{\m_{a+\frac{1}{2}}}$. Thus, $T(L_{\l_{a+\frac{1}{2}}})=L_{\m_a}$ as required.


We know that $L_{\l_{a+1}}$ is a quotient of $\G_0(G/B,\O_{\l_{a+1}})$. From inductive assumption, \rlem{coh011}, and \rlem{aa1coho}, we have the following exact sequence: 

$$0\rightarrow L_{\l_{a+\frac{1}{2}}} \rightarrow \G_0(G/B,\O_{\l_{a+1}})\rightarrow L_{\l_{a+1}}\rightarrow 0.$$

Since $T$ is an exact functor, we obtain the following exact sequence: 

$$0\rightarrow T(L_{\l_{a+\frac{1}{2}}})\rightarrow T(\G_0(G/B,\O_{\l_{a+1}}))\rightarrow T(L_{\l_{a+1}})\rightarrow 0.$$

From \rlem{aa1x}, we have $T(L_{\l_{a+1}})=L_{\m_{a+\frac{1}{2}}}$. By \rlem{6} and \rlem{w}, we have $$T(\G_0(G/B,\O_{\l_{a+1}}))=\G_0(G/B,T(\O_{\l_{a+1}}))=\G_0(G/B,\O_{\m_{a+\frac{1}{2}}}).$$ The module $\G_0(G/B,\O_{\m_{a+\frac{1}{2}}})$ has a unique quotient $L_{\m_{a+\frac{1}{2}}}$. Hence, $T(L_{\l_{a+\frac{1}{2}}})$ does not have a subquotient $L_{\m_{a+\frac{1}{2}}}$.
\end{proof}
 
\begin {theorem} The categories  $\F^{(a,a)}$ and $\F^{(a+1,a+1)}$ are equivalent for all $a\geq 1$. 
\end {theorem}

\begin{proof} This follows from \rthm{eqgeneral} together with \rlem{aa1x} and \rlem{aa1}.


\end{proof}

\section{Equivalence of non-symmetric blocks in $F(4)$}

\subsection{Equivalence of blocks $\F^{(4, 1)}$ and $\F^{(5, 2)}$} 

Let $\g=F(4)$. The following is the picture of translator functor from block $\F^{(4, 1)}$ to $\F^{(5, 2)}$. It is defined by $T(L_{\l})=(L_{\l}\otimes\g)^{(2,2)}$. The non-filled circles represent the non-dominant weights in the block occurring on the walls of the Weyl chamber. The filled circles represent dominant weights in the block. The vertical arrows are maps $\l\mapsto \l+\gamma$, with $\gamma\in\D$ is the root above the arrow. In this section, we will show that the solid arrows represent the maps $L_{\l}\mapsto T(L_{\l})$. 

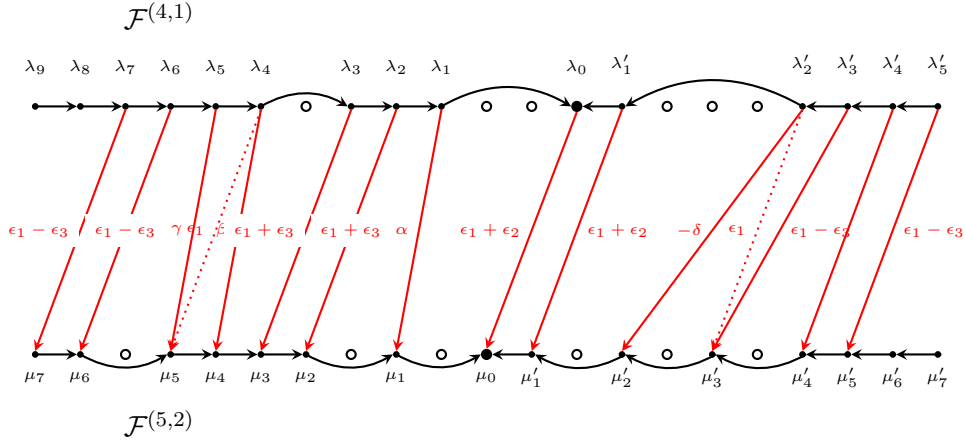
\begin{center}
  \begin{tikzpicture}[scale=.3]
  
    \draw (-2.5,-1) node[anchor=north]  {$\F^{(5, 2)}$}; 
    \draw (-2.5,15) node[anchor=south]  {$\F^{(4, 1)}$}; 
    
           \draw[xshift=-6 cm,thick,fill=black] (38 cm,12) circle (.1cm);
       \draw[xshift=-6 cm,thick,fill=black] (36 cm,12) circle (.1cm);
    \draw[xshift=-6 cm,thick,fill=black] (34 cm,12) circle (.1cm);
    \draw[xshift=-6 cm,thick,fill=black] (32 cm,12) circle (.1cm);
    \draw[xshift=-6 cm,thick] (30 cm,12) circle (.2cm);
    \draw[xshift=-6 cm,thick] (28cm,12) circle (.2cm);
    \draw[xshift=-6 cm,thick] (26 cm,12) circle (.2cm);
    \draw[xshift=-6 cm,thick,fill=black] (24 cm,12) circle (.1cm);
    \draw[xshift=-6 cm,thick,fill=black] ( 22 cm,12) circle (.2cm);
    \draw[xshift=-6 cm,thick] (20 cm,12) circle (.2cm);
    \draw[xshift=-6 cm,thick] (18 cm,12) circle (.2cm);
    \draw[xshift=-6 cm,thick,fill=black] (16cm,12) circle (.1cm);
    \draw[xshift=-6 cm,thick,fill=black] (14 cm,12) circle (.1cm);
  \draw[xshift=-6 cm,thick,fill=black] (12 cm,12) circle (.1cm);
    \draw[xshift=-6 cm,thick] (10 cm,12) circle (.2cm);
      \draw[xshift=-6 cm,thick,fill=black] (8 cm,12) circle (.1cm);
        \draw[xshift=-6 cm,thick,fill=black] (6 cm,12) circle (.1cm);
        \draw[xshift=-6 cm,thick,fill=black] (4 cm,12) circle (.1cm);
        \draw[xshift=-6 cm,thick,fill=black] (2 cm,12) circle (.1cm);
                \draw[xshift=-6 cm,thick,fill=black] (0 cm,12) circle (.1cm);
        \draw[xshift=-6 cm,thick,fill=black] (-2 cm,12) circle (.1cm);
    
       \draw[xshift=-6 cm,thick,fill=black] (38 cm,1) circle (.1cm);
       \draw[xshift=-6 cm,thick,fill=black] (36 cm,1) circle (.1cm);
    \draw[xshift=-6 cm,thick,fill=black] (34 cm,1) circle (.1cm);
    \draw[xshift=-6 cm,thick,fill=black] (32 cm,1) circle (.1cm);
    \draw[xshift=-6 cm,thick] (30 cm,1) circle (.2cm);
    \draw[xshift=-6 cm,thick,fill=black] (28cm,1) circle (.1cm);
    \draw[xshift=-6 cm,thick] (26 cm,1) circle (.2cm);
    \draw[xshift=-6 cm,thick,fill=black] (24 cm,1) circle (.1cm);
    \draw[xshift=-6 cm,thick] ( 22 cm,1) circle (.2cm);
    \draw[xshift=-6 cm,thick,fill=black] (20 cm,1) circle (.1cm);
    \draw[xshift=-6 cm,thick,fill=black] (18 cm,1) circle (.2cm);
    \draw[xshift=-6 cm,thick] (16cm,1) circle (.2cm);
    \draw[xshift=-6 cm,thick,fill=black] (14 cm,1) circle (.1cm);
  \draw[xshift=-6 cm,thick] (12 cm,1) circle (.2cm);
    \draw[xshift=-6 cm,thick,fill=black] (10 cm,1) circle (.1cm);
      \draw[xshift=-6 cm,thick,fill=black] (8 cm,1) circle (.1cm);
        \draw[xshift=-6 cm,thick,fill=black] (6 cm,1) circle (.1cm);
        \draw[xshift=-6 cm,thick,fill=black] (4 cm, 1) circle (.1cm);
        \draw[xshift=-6 cm,thick] (2 cm, 1) circle (.2cm);
         \draw[xshift=-6 cm,thick,fill=black] (0 cm, 1) circle (.1cm);  
        \draw[xshift=-6 cm,thick,fill=black] (-2 cm, 1) circle (.1cm);
        
     \draw[xshift=-6 cm] (38 cm, 0) node[anchor=center]  {{\tiny $\m_7'$}};
      \draw[xshift=-6 cm] (36 cm, 0) node[anchor=center]  {{\tiny $\m_6'$}}; 
     \draw[xshift=-6 cm] (34 cm, 0) node[anchor=center]  {{\tiny $\m_5'$}};      
   \draw[xshift=-6 cm] (32 cm, 0) node[anchor=center]  {{\tiny $\m_4'$}};
    \draw[xshift=-6 cm] (30 cm, 0) node[anchor=center]  {$$};
     \draw[xshift=-6 cm] (28 cm, 0) node[anchor=center]  {{\tiny $\m_3'$}};
      \draw[xshift=-6 cm] (26 cm, 0) node[anchor=center]  {$$};  
     \draw[xshift=-6 cm] (24 cm, 0) node[anchor=center]  {{\tiny $\m_2'$}};      
   \draw[xshift=-6 cm] (22 cm, 0) node[anchor=center]  {$$};
    \draw[xshift=-6 cm] (20 cm, 0) node[anchor=center]  {{\tiny $\m_1'$}};
     \draw[xshift=-6 cm] (18 cm, 0) node[anchor=center]  {{\tiny $\m_0$}};
      \draw[xshift=-6 cm] (16 cm, 0) node[anchor=center]  {$$};          
   \draw[xshift=-6 cm] (14 cm, 0) node[anchor=center]  {{\tiny $\m_1$}};      
   \draw[xshift=-6 cm] (12 cm, 0) node[anchor=center]  {$$};
    \draw[xshift=-6 cm] (10 cm, 0) node[anchor=center]  {{\tiny $\m_2$}};
     \draw[xshift=-6 cm] (8 cm, 0) node[anchor=center]  {{\tiny $\m_3$}};
      \draw[xshift=-6 cm] (6 cm, 0) node[anchor=center]  {{\tiny $\m_4$}};
     \draw[xshift=-6 cm] (4 cm, 0) node[anchor=center]  {{\tiny $\m_5$}};
      \draw[xshift=-6 cm] (2 cm, 0) node[anchor=center]  {$$};      
            \draw[xshift=-6 cm] (0 cm, 0) node[anchor=center]  {{\tiny $\m_6$}};  
                        \draw[xshift=-6 cm] (-2 cm, 0) node[anchor=center]  {{\tiny $\m_7$}}; 
                
     \draw[xshift=-6 cm] (38 cm, 13) node[anchor=south]  {{\tiny $\l_5'$}};
      \draw[xshift=-6 cm] (36 cm, 13) node[anchor=south]  {{\tiny $\l_4'$}}; 
     \draw[xshift=-6 cm] (34 cm, 13) node[anchor=south]  {{\tiny $\l_3'$}};      
   \draw[xshift=-6 cm] (32 cm, 13) node[anchor=south]  {{\tiny $\l_2'$}};
    \draw[xshift=-6 cm] (30 cm, 13) node[anchor=south]  {$$};
     \draw[xshift=-6 cm] (28 cm, 13) node[anchor=south]  {$$};
      \draw[xshift=-6 cm] (26 cm, 13) node[anchor=south]  {$$};  
     \draw[xshift=-6 cm] (24 cm, 13) node[anchor=south]  {{\tiny $\l_1'$}};      
   \draw[xshift=-6 cm] (22 cm, 13) node[anchor=south]  {{\tiny $\l_0$}};
    \draw[xshift=-6 cm] (20 cm, 13) node[anchor=south]  {$$};
     \draw[xshift=-6 cm] (18 cm, 13) node[anchor=south]  {$$};
      \draw[xshift=-6 cm] (16 cm, 13) node[anchor=south]  {{\tiny $\l_1$}};          
   \draw[xshift=-6 cm] (14 cm, 13) node[anchor=south]  {{\tiny $\l_2$}};      
   \draw[xshift=-6 cm] (12 cm, 13) node[anchor=south]  {{\tiny $\l_3$}};
    \draw[xshift=-6 cm] (10 cm, 13) node[anchor=south]  {$$};
     \draw[xshift=-6 cm] (8 cm, 13) node[anchor=south]  {{\tiny $\l_4$}};
      \draw[xshift=-6 cm] (6 cm, 13) node[anchor=south]  {{\tiny $\l_5$}};
     \draw[xshift=-6 cm] (4 cm, 13) node[anchor=south]  {{\tiny $\l_6$}};
      \draw[xshift=-6 cm] (2 cm, 13) node[anchor=south]  {{\tiny $\l_7$}};   
     \draw[xshift=-6 cm] (0 cm, 13) node[anchor=south]  {{\tiny $\l_8$}};
      \draw[xshift=-6 cm] (-2 cm, 13) node[anchor=south]  {{\tiny $\l_9$}}; 

  \path [-stealth, solid,red,thick,draw=none] ( -4 cm,11.85) edge node[left, draw=none,fill=white,rectangle]{\tiny $\e_1-\e_3$}(-8 cm,1.15);
  \path [-stealth, solid,red,thick,draw=none] ( -2 cm,11.85) edge node[draw=none,fill=white,rectangle]{\tiny $\e_1-\e_3$}(-6 cm,1.15);
  \path [-stealth, solid,red,thick,draw=none] ( 0 cm,11.85) edge node[left=0]{\tiny $\gg$}(-2 cm,1.15);
  \path [-stealth, solid,red,thick,draw=none] ( 2 cm,11.85) edge node[left=0]{\tiny $\b$}(0 cm,1.15);
    \path [-stealth, solid,red,thick, dotted,draw=none] ( 2 cm,11.85) edge node[left=0]{\tiny $\e_1$}(-2 cm,1.15);
  \path [-stealth, solid,red,thick,draw=none] ( 6 cm,11.85) edge node[left,draw=none,fill=white,rectangle]{\tiny $\e_1+\e_3$}(2 cm,1.15);
 \path [-stealth, solid,red,thick,draw=none] ( 8 cm,11.85) edge node[draw=none,fill=white,rectangle]{\tiny $\e_1+\e_3$}(4 cm,1.15);
 \path [-stealth, solid,red,thick,draw=none] ( 10 cm,11.85) edge node[left=0]{\tiny $\a$}(8 cm,1.15); 
 \path [-stealth, solid,red,thick,draw=none] ( 16 cm,11.75) edge node[left=0]{\tiny $\e_1+\e_2$}(12 cm,1.25); 
 \path [-stealth, solid,red,thick,draw=none] ( 18 cm,11.85) edge node[right=0]{\tiny $\e_1+\e_2$}(14 cm,1.15);
 \path [-stealth, solid,red,thick,draw=none] ( 26 cm,11.85) edge node[left=0]{\tiny $-\d$}(18 cm,1.15);
  \path [-stealth, solid,red,thick,dotted,draw=none] ( 26 cm,11.85) edge node[left=0]{\tiny $\e_1$}(22 cm,1.15);
 \path [-stealth, solid,red,thick,draw=none] ( 28 cm,11.85) edge node[right=0]{\tiny $\e_1-\e_3$}(22 cm,1.15);
 \path [-stealth, solid,red,thick,draw=none] ( 30 cm,11.85) edge node[right=0.1]{\tiny $$}(26 cm,1.15);
 \path [-stealth, solid,red,thick,draw=none] ( 32 cm,11.85) edge node[right=0]{\tiny $\e_1-\e_3$}(28 cm,1.15);

  \path [-stealth,solid,black,thick,draw=none] ( -8 cm,12) edge node[left=0]{}(-6.15 cm,12);
  \path [-stealth,solid,black,thick,draw=none] ( -6 cm,12) edge node[left=0]{}(-4.15 cm,12);
    \path [-stealth,solid,black,thick,draw=none] ( -4 cm,12) edge node[left=0]{}(-2.15 cm,12);
  \path [-stealth,solid,black,thick,draw=none] ( -2 cm,12) edge node[left=0]{}(-0.15 cm,12);
      \path [-stealth,solid,black,thick,draw=none] ( 0 cm,12) edge node[left=0]{}(1.85 cm,12);
  \path [-stealth,solid,black,thick,draw=none] ( 2 cm,12) edge[bend left] node[left=0]{}(5.85 cm,12);
  \path [-stealth,solid,black,thick,draw=none] ( 6 cm,12) edge node[left=0]{}(7.85 cm,12);
      \path [-stealth,solid,black,thick,draw=none] ( 8 cm,12) edge node[left=0]{}(9.85 cm,12);  
  \path [-stealth,solid,black,thick,draw=none] ( 10 cm,12) edge[bend left] node[left=0]{}(15.75 cm,12);
      \path [-stealth,solid,black,thick,draw=none] ( 18 cm,12) edge node[left=0]{}(16.25 cm,12);  
  \path [-stealth,solid,black,thick,draw=none] ( 26 cm,12) edge[bend right] node[left=0]{}(18.15 cm,12);
      \path [-stealth,solid,black,thick,draw=none] ( 28 cm,12) edge node[left=0]{}(26.15 cm,12);  
      \path [-stealth,solid,black,thick,draw=none] ( 30 cm,12) edge node[left=0]{}(28.15 cm,12);  
      \path [-stealth,solid,black,thick,draw=none] ( 32 cm,12) edge node[left=0]{}(30.15 cm,12);  

       \path [-stealth,solid,black,thick,draw=none] ( -8 cm,1) edge node[left=0]{}(-6.15 cm,1);  
  \path [-stealth,solid,black,thick,draw=none] ( -6 cm,1) edge[bend right] node[left=0]{}(-2.15 cm,1);
       \path [-stealth,solid,black,thick,draw=none] ( -2 cm,1) edge node[left=0]{}(-0.15 cm,1);  
       \path [-stealth,solid,black,thick,draw=none] ( 0 cm,1) edge node[left=0]{}(1.85 cm,1);  
       \path [-stealth,solid,black,thick,draw=none] ( 2 cm,1) edge node[left=0]{}(3.85 cm,1);  
       \path [-stealth,solid,black,thick,draw=none] ( 4 cm,1) edge[bend right]  node[left=0]{}(7.85 cm,1);  
       \path [-stealth,solid,black,thick,draw=none] ( 8 cm,1) edge[bend right]  node[left=0]{}(11.75 cm,1);  
       \path [-stealth,solid,black,thick,draw=none] ( 14.15 cm,1) edge node[left=0]{}(12.25 cm,1);  
       \path [-stealth,solid,black,thick,draw=none] ( 18.15 cm,1) edge[bend left]  node[left=0]{}(14.15 cm,1);  
       \path [-stealth,solid,black,thick,draw=none] ( 22.15 cm,1) edge[bend left]  node[left=0]{}(18.15 cm,1);  
       \path [-stealth,solid,black,thick,draw=none] ( 26.15 cm,1) edge[bend left]  node[left=0]{}(22.15 cm,1);  
       \path [-stealth,solid,black,thick,draw=none] ( 28.15 cm,1) edge node[left=0]{}(26.15 cm,1);  
       \path [-stealth,solid,black,thick,draw=none] ( 30.15 cm,1) edge node[left=0]{}(28.15 cm,1);  
       \path [-stealth,solid,black,thick,draw=none] ( 32.15 cm,1) edge node[left=0]{}(30.15 cm,1);

  \end{tikzpicture}
\end{center}

In the above picture $\ga=(\frac{1}{2},-\frac{1}{2},-\frac{1}{2}|-\frac{1}{2})$ and
$\l_4'+\r=(\frac{13}{2},\frac{5}{2},\frac{3}{2}|\frac{7}{2})$;
$\l_3'+\r=(6,2,1|3)$;
$\l_2'+\r=(\frac{11}{2},\frac{3}{2},\frac{1}{2}|\frac{5}{2})$;
$\l_1'+\r=(\frac{7}{2},\frac{3}{2},\frac{1}{2}|\frac{1}{2})$;
$\l_0+\r=(3,2,1|0)$;
$\l_1+\r=(\frac{7}{2},\frac{5}{2},\frac{3}{2}|\frac{3}{2})$;
$\l_2+\r=(4,3,1|2)$;
$\l_3+\r=(\frac{9}{2},\frac{7}{2},\frac{1}{2}|\frac{5}{2})$;
$\l_4+\r=(\frac{11}{2},\frac{9}{2},\frac{1}{2}|\frac{7}{2})$;
$\l_5+\r=(6,5,1|4)$;
$\l_6+\r=(\frac{13}{2},\frac{11}{2},\frac{3}{2}|\frac{9}{2})$;
$\m_4'+\r=(\frac{15}{2},\frac{5}{2},\frac{1}{2}|\frac{7}{2})$;
$\m_3'+\r=(\frac{13}{2},\frac{3}{2},\frac{1}{2}|\frac{5}{2})$;
$\m_2'+\r=(\frac{11}{2},\frac{3}{2},\frac{1}{2}|\frac{3}{2})$;
$\m_1'+\r=(\frac{9}{2},\frac{5}{2},\frac{1}{2}|\frac{1}{2})$;
$\m_0+\r=(4,3,1|0)$;
$\m_1+\r=(4,3,2|1)$;
$\m_2+\r=(5,3,2|2)$;
$\m_3+\r=(\frac{11}{2},\frac{7}{2},\frac{3}{2}|\frac{5}{2})$;
$\m_4+\r=(6,4,1|3)$;
$\m_5+\r=(\frac{13}{2},\frac{9}{2},\frac{1}{2}|\frac{7}{2})$;
$\m_6+\r=(\frac{15}{2},\frac{11}{2},\frac{1}{2}|\frac{9}{2})$. 
Note that the indices for $\l$ above are different from the index $c$ which represents the last coordinate of $\l+\r$. 

\begin{lemma}\label{coh4152} For a dominant weight $\l\in F^{(4,1)}$ with $\l=\l_c$ such that $c>\frac{5}{2}$ or $c<-\frac{7}{2}$, we have $[\G_0(G/B,\O_{\l}):L_{\l-\a}]=1$ for a unique $\a\in\D$ such that $\l-\a\in F^{(4,1)}$.

Also, we have $[\G_0(G/B,\O_{\l}):L_{\m}]=0$ for $\m\neq \l$ and $\m\neq\l-\a$. 
\end{lemma}
\begin{proof} Follows from \rlem{3} and \rlem{lemma 3 generic weights}. 
\end{proof}

\begin{lemma}\label{x4152} We have $T(L_{\l_i})=L_{\m_i}$, for all $i\neq 4$ and $T(L_{\l_i'})=L_{\m_i'}$, for all $i\neq 2$.
\end{lemma}

\begin{proof} By definition, $T(L_{\l_i})=(L_{\l_i}\otimes\g)^{(5,2)}$. For each $\l_i$, we have a unique dominant weight $\m_i$ in the block $\F^{(5, 2)}$ of the form $\l_i+\gg$ with $\gg\in \D$. Thus, the lemma follows from \rlem{acyclicx}.
\end{proof}

\begin{lemma} We have $T(L_{\l_4})=L_{\m_4}$ and $T(L_{\l_2'})=L_{\m_2'}$.
\end{lemma}
 
\begin{proof} By definition, $T(L_{\l_4})=(L_{\l_4}\otimes\g)^{(5,2)}$. The only dominant weights in the block $\F^{(5, 2)}$ of the form $\l_4+\gg$ with $\gg\in \D$ are $\m_4$ and $\m_5$, as its shown in the picture above. The proof is similar to the proof of \rlem{special1122} using \rlem{x4152}.

\end{proof} 

\begin{theorem}\label{eq4152} We have an equivalence of categories $\F^{(4, 1)}$ and $\F^{(5, 2)}$. 
\end{theorem}

\begin{proof} The proof is similar to the proof of \rlem{eq1122} using the above two lemmas. 







 \end{proof}

\subsection{Equivalence of blocks $\F^{(a, b)}$ and $\F^{(a+1, b+1)}$}

Let $\g=F(4)$. The following is the picture of translator functor from block $\F^{(a, b)}$ to $\F^{(a+1, b+1)}$. It is defined by $T(L_{\l})=(L_{\l}\otimes\g)^{(a+1, b+1)}$. The non-filled circles represent the non-dominant weights in the block occurring on the walls of the Weyl chamber. The filled circles represent dominant weights in the block. The vertical arrows are maps $\l\mapsto \l+\gamma$, with $\gamma\in\D$ is the root above the arrow.

In the picture below, $t_i$ are as defined before and represent the indices $c$ corresponding to acyclic weights. Let $\l_i$ denote the starting vertices of the vertical arrows and $\m_i$ be the corresponding end vertices.  In this section, we will show that the solid arrows represent the maps $L_{\l}\mapsto T(L_{\l})$.

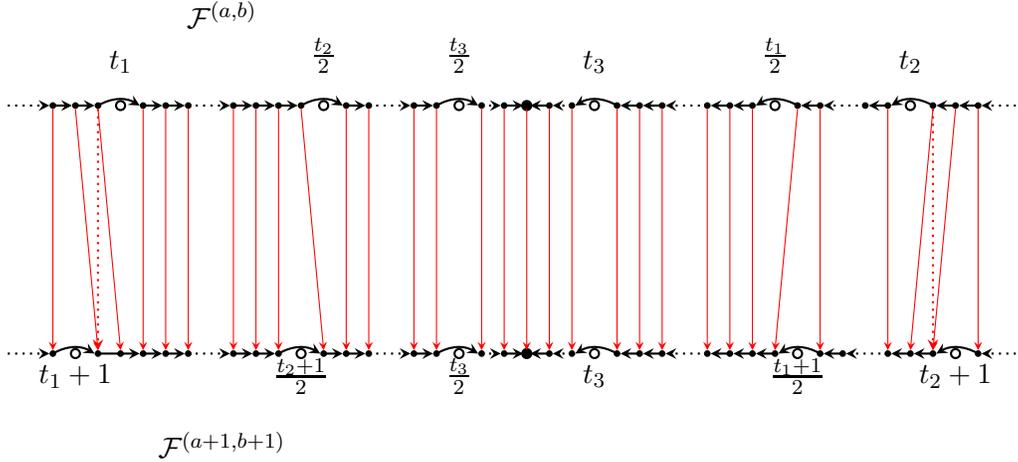
\begin{center}
  \begin{tikzpicture}[scale=.3]
  
    \draw (-2.5,-2) node[anchor=north]  {$\F^{(a+1, b+1)}$}; 
    \draw (-2.5,15) node[anchor=south]  {$\F^{(a, b)}$}; 
 
       \draw[xshift=-6 cm,thick,fill=black] (37 cm,12) circle (.1cm);
    \draw[xshift=-6 cm,thick,fill=black] (36 cm,12) circle (.1cm);
    \draw[xshift=-6 cm,thick,fill=black] (35 cm,12) circle (.1cm);
    \draw[xshift=-6 cm,thick] (34 cm,12) circle (.2cm);
    \draw[xshift=-6 cm,thick,fill=black] (33cm,12) circle (.1cm);
    \draw[xshift=-6 cm,thick,fill=black] (32 cm,12) circle (.1cm);
    \draw[xshift=-6 cm,thick,fill=black] ( 30 cm,12) circle (.1cm);
    \draw[xshift=-6 cm,thick,fill=black] (29 cm,12) circle (.1cm);
    \draw[xshift=-6 cm,thick] (28 cm,12) circle (.2cm);
    \draw[xshift=-6 cm,thick,fill=black] (27cm,12) circle (.1cm);
    \draw[xshift=-6 cm,thick,fill=black] (26 cm,12) circle (.1cm);
      \draw[xshift=-6 cm,thick,fill=black] (25 cm,12) circle (.1cm);
    \draw[xshift=-6 cm,thick,fill=black] (23 cm,12) circle (.1cm);
      \draw[xshift=-6 cm,thick,fill=black] (22 cm,12) circle (.1cm);
       \draw[xshift=-6 cm,thick,fill=black] (21 cm,12) circle (.1cm);
    \draw[xshift=-6 cm,thick] (20 cm,12) circle (.2cm);
    \draw[xshift=-6 cm,thick,fill=black] (19 cm,12) circle (.1cm);
    \draw[xshift=-6 cm,thick,fill=black] (18 cm,12) circle (.1cm);
    \draw[xshift=-6 cm,thick,fill=black] (17cm,12) circle (.2cm);
    \draw[xshift=-6 cm,thick,fill=black] (16 cm,12) circle (.1cm);
    \draw[xshift=-6 cm,thick,fill=black] (15 cm,12) circle (.1cm);
    \draw[xshift=-6 cm,thick] ( 14 cm,12) circle (.2cm);
    \draw[xshift=-6 cm,thick,fill=black] (13 cm,12) circle (.1cm);
    \draw[xshift=-6 cm,thick,fill=black] (12 cm,12) circle (.1cm);
    \draw[xshift=-6 cm,thick,fill=black] (10 cm,12) circle (.1cm);
  \draw[xshift=-6 cm,thick,fill=black] (9 cm,12) circle (.1cm);
    \draw[xshift=-6 cm,thick] (8 cm,12) circle (.2cm);
\draw[xshift=-6 cm,thick,fill=black] (7 cm,12) circle (.1cm);
    \draw[xshift=-6 cm,thick,fill=black] (6 cm,12) circle (.1cm);
    \draw[xshift=-6 cm,thick,fill=black] (5cm,12) circle (.1cm);
    \draw[xshift=-6 cm,thick,fill=black] (4 cm,12) circle (.1cm);
    \draw[xshift=-6 cm,thick,fill=black] ( 2cm,12) circle (.1cm);
    \draw[xshift=-6 cm,thick,fill=black] (1cm,12) circle (.1cm);
    \draw[xshift=-6 cm,thick,fill=black] (0cm,12) circle (.1cm);
    \draw[xshift=-6 cm,thick] (-1cm,12) circle (.2cm);
    \draw[xshift=-6 cm,thick,fill=black] (-2cm,12) circle (.1cm);
  \draw[xshift=-6 cm,thick,fill=black] (-3cm,12) circle (.1cm);
  \draw[xshift=-6 cm,thick,fill=black] (-4cm,12) circle (.1cm);
    
       \draw[xshift=-6 cm,thick,fill=black] (37 cm,1) circle (.1cm);
    \draw[xshift=-6 cm,thick] (36 cm,1) circle (.2cm);
    \draw[xshift=-6 cm,thick,fill=black] (35 cm,1) circle (.1cm);
    \draw[xshift=-6 cm,thick,fill=black] (34 cm,1) circle (.1cm);
    \draw[xshift=-6 cm,thick,fill=black] (33cm,1) circle (.1cm);
    \draw[xshift=-6 cm,thick,fill=black] (31 cm,1) circle (.1cm);
    \draw[xshift=-6 cm,thick,fill=black] ( 30 cm,1) circle (.1cm);
    \draw[xshift=-6 cm,thick] (29 cm,1) circle (.2cm);
    \draw[xshift=-6 cm,thick,fill=black] (28 cm,1) circle (.1cm);
    \draw[xshift=-6 cm,thick,fill=black] (27cm,1) circle (.1cm);
    \draw[xshift=-6 cm,thick,fill=black] (26 cm,1) circle (.1cm);
      \draw[xshift=-6 cm,thick,fill=black] (25 cm,1) circle (.1cm);
    \draw[xshift=-6 cm,thick,fill=black] (23 cm,1) circle (.1cm);
      \draw[xshift=-6 cm,thick,fill=black] (22 cm,1) circle (.1cm);
       \draw[xshift=-6 cm,thick,fill=black] (21 cm,1) circle (.1cm);
    \draw[xshift=-6 cm,thick] (20 cm,1) circle (.2cm);
    \draw[xshift=-6 cm,thick,fill=black] (19 cm,1) circle (.1cm);
    \draw[xshift=-6 cm,thick,fill=black] (18 cm,1) circle (.1cm);
    \draw[xshift=-6 cm,thick,fill=black] (17cm,1) circle (.2cm);
    \draw[xshift=-6 cm,thick,fill=black] (16 cm,1) circle (.1cm);
    \draw[xshift=-6 cm,thick,fill=black] (15 cm,1) circle (.1cm);
    \draw[xshift=-6 cm,thick] ( 14 cm,1) circle (.2cm);
    \draw[xshift=-6 cm,thick,fill=black] (13 cm,1) circle (.1cm);
    \draw[xshift=-6 cm,thick,fill=black] (12 cm,1) circle (.1cm);
    \draw[xshift=-6 cm,thick,fill=black] (10 cm,1) circle (.1cm);
  \draw[xshift=-6 cm,thick,fill=black] (9 cm,1) circle (.1cm);
    \draw[xshift=-6 cm,thick,fill=black] (8 cm,1) circle (.1cm);
    \draw[xshift=-6 cm,thick] (7 cm,1) circle (.2cm);
    \draw[xshift=-6 cm,thick,fill=black] (6 cm,1) circle (.1cm);
    \draw[xshift=-6 cm,thick,fill=black] (5cm,1) circle (.1cm);
    \draw[xshift=-6 cm,thick,fill=black] (4 cm,1) circle (.1cm);
    \draw[xshift=-6 cm,thick,fill=black] ( 2cm,1) circle (.1cm);
    \draw[xshift=-6 cm,thick,fill=black] (1cm,1) circle (.1cm);
    \draw[xshift=-6 cm,thick,fill=black] (0cm,1) circle (.1cm);
    \draw[xshift=-6 cm,thick,fill=black] (-1cm,1) circle (.1cm);
    \draw[xshift=-6 cm,thick,fill=black] (-2cm,1) circle (.1cm);
  \draw[xshift=-6 cm,thick] (-3cm,1) circle (.2cm);
    \draw[xshift=-6 cm,thick,fill=black] (-4cm,1) circle (.1cm);

     \draw[xshift=-6 cm] (36 cm, 0) node[anchor=center]  {$t_2+1$};      
   \draw[xshift=-6 cm] (29 cm, 0) node[anchor=center]  {$\frac{t_1+1}{2}$};
    \draw[xshift=-6 cm] (20 cm, 0) node[anchor=center]  {$t_3$};
     \draw[xshift=-6 cm] (14 cm, 0) node[anchor=center]  {$\frac{t_3}{2}$};
      \draw[xshift=-6 cm] (7 cm, 0) node[anchor=center]  {$\frac{t_2+1}{2}$};  
     \draw[xshift=-6 cm] (-3 cm, 0) node[anchor=center]  {$t_1+1$};

     \draw[xshift=-6 cm] (34 cm, 13) node[anchor=south]  {$t_2$};      
   \draw[xshift=-6 cm] (28 cm, 13) node[anchor=south]  {$\frac{t_1}{2}$};    
      \draw[xshift=-6 cm] (20 cm, 13) node[anchor=south]  {$t_3$};          
   \draw[xshift=-6 cm] (14 cm, 13) node[anchor=south]  {$\frac{t_3}{2}$};
    \draw[xshift=-6 cm] (8 cm, 13) node[anchor=south]  {$\frac{t_2}{2}$};
     \draw[xshift=-6 cm] (-1cm, 13) node[anchor=south]  {$t_1$};
     
       \foreach \y in {-11,-7,-6,-5,-3,-2,-1,2,3,5,6,8,9,10,11,12,14,15,16,18,19,20,23,26,30}
    \draw[-stealth, xshift=1 cm,red] (\y,11.85) -- +(0,-10.70);

  \path [-stealth, solid,red,draw=none,dotted,thick] ( -8 cm,11.85) edge (-8 cm,1.15);
  \path [-stealth, solid,red,draw=none,dotted,thick] ( 29 cm,11.85) edge (29 cm,1.15);

  \path [-stealth, solid,red,draw=none] ( -9 cm,11.85) edge (-8 cm,1.15);
  \path [-stealth, solid,red,draw=none] ( -8 cm,11.85) edge (-7 cm,1.15);
  \path [-stealth, solid,red,draw=none] ( 1 cm,11.85) edge (2 cm,1.15);
  \path [-stealth, solid,red,draw=none] ( 23 cm,11.85) edge (22 cm,1.15);
  \path [-stealth, solid,red,draw=none] ( 29 cm,11.85) edge (28 cm,1.15);
  \path [-stealth, solid,red,draw=none] ( 30 cm,11.85) edge (29 cm,1.15);

    \path [-stealth,solid,black,thick,draw=none, dotted] ( -12 cm,12) edge node[left=0]{}(-10.15 cm,12);
      \path [-stealth,solid,black,thick,draw=none] ( -10 cm,12) edge node[left=0]{}(-9.15 cm,12);
    
      \path [-stealth,solid,black,thick,draw=none] ( -9 cm,12) edge node[left=0]{}(-8.15 cm,12);
  \path [-stealth,solid,black,thick,draw=none] ( -8 cm,12) edge[bend left] node[left=0]{}(-6.15 cm,12);
    \path [-stealth,solid,black,thick,draw=none] ( -6 cm,12) edge node[left=0]{}(-5.15 cm,12);
  \path [-stealth,solid,black,thick,draw=none] ( -5 cm,12) edge node[left=0]{}(-4.15 cm,12);
      \path [-stealth,solid,black,thick,draw=none, dotted] ( -4 cm,12) edge node[left=0]{}(-2.15 cm,12);
      \path [-stealth,solid,black,thick,draw=none] ( -2 cm,12) edge node[left=0]{}(-1.15 cm,12);
      \path [-stealth,solid,black,thick,draw=none] ( -1 cm,12) edge node[left=0]{}(-0.15 cm,12);
      \path [-stealth,solid,black,thick,draw=none] ( 0 cm,12) edge node[left=0]{}(0.85 cm,12);
      \path [-stealth,solid,black,thick,draw=none] ( 1 cm,12) edge[bend left] node[left=0]{}(2.85 cm,12);
      \path [-stealth,solid,black,thick,draw=none] ( 3 cm,12) edge node[left=0]{}(3.85 cm,12);
      \path [-stealth,solid,black,thick,draw=none, dotted] ( 4 cm,12) edge node[left=0]{}(5.85 cm,12);
            \path [-stealth,solid,black,thick,draw=none] ( 6 cm,12) edge node[left=0]{}(6.85 cm,12);

      \path [-stealth,solid,black,thick,draw=none] ( 7 cm,12) edge[bend left] node[left=0]{}(8.85 cm,12);      
      \path [-stealth,solid,black,thick,draw=none,dotted] ( 9 cm,12) edge node[left=0]{}(9.85 cm,12);      
      \path [-stealth,solid,black,thick,draw=none] ( 10 cm,12) edge node[left=0]{}(10.85 cm,12);      
      
       \path [-stealth,solid,black,thick,draw=none] ( 12 cm,12) edge node[left=0]{}(11.15 cm,12);      
       \path [-stealth,solid,black,thick,draw=none,dotted] ( 13 cm,12) edge node[left=0]{}(12.15 cm,12);      
       \path [-stealth,solid,black,thick,draw=none] ( 15 cm,12) edge[bend right] node[left=0]{}(13.15 cm,12);      
       \path [-stealth,solid,black,thick,draw=none] ( 16 cm,12) edge node[left=0]{}(15.15 cm,12);      
       \path [-stealth,solid,black,thick,draw=none] ( 17 cm,12) edge node[left=0]{}(16.15 cm,12);      
       \path [-stealth,solid,black,thick,draw=none, dotted] ( 19 cm,12) edge node[left=0]{}(17.15 cm,12);      
       \path [-stealth,solid,black,thick,draw=none] ( 20 cm,12) edge node[left=0]{}(19.15 cm,12);      
       \path [-stealth,solid,black,thick,draw=none] ( 21 cm,12) edge node[left=0]{}(20.15 cm,12);      
            \path [-stealth,solid,black,thick,draw=none] ( 23 cm,12) edge[bend right] node[left=0]{}(21.15 cm,12);      
       \path [-stealth,solid,black,thick,draw=none] ( 24 cm,12) edge node[left=0]{}(23.15 cm,12);      
       \path [-stealth,solid,black,thick,draw=none, dotted] ( 26 cm,12) edge node[left=0]{}(24.15 cm,12);      
       \path [-stealth,solid,black,thick,draw=none] ( 27 cm,12) edge node[left=0]{}(26.15 cm,12);      
            \path [-stealth,solid,black,thick,draw=none] ( 29 cm,12) edge[bend right] node[left=0]{}(27.15 cm,12);      
       \path [-stealth,solid,black,thick,draw=none] ( 30 cm,12) edge node[left=0]{}(29.15 cm,12);      
       \path [-stealth,solid,black,thick,draw=none] ( 31 cm,12) edge node[left=0]{}(30.15 cm,12);      
       \path [-stealth,solid,black,thick,draw=none, dotted] ( 33 cm,12) edge node[left=0]{}(31.15 cm,12);

     \path [-stealth,solid,black,thick,draw=none, dotted] ( -12 cm,1) edge node[left=0]{}(-10.15 cm,1);
      \path [-stealth,solid,black,thick,draw=none] ( -10 cm,1) edge[bend left] node[left=0]{}(-8.15 cm,1);
  \path [-stealth,solid,black,thick,draw=none] ( -8 cm,1) edge node[left=0]{}(-6.15 cm,1);
    \path [-stealth,solid,black,thick,draw=none] ( -6 cm,1) edge node[left=0]{}(-5.15 cm,1);
  \path [-stealth,solid,black,thick,draw=none] ( -5 cm,1) edge node[left=0]{}(-4.15 cm,1);
      \path [-stealth,solid,black,thick,draw=none, dotted] ( -4 cm,1) edge node[left=0]{}(-2.15 cm,1);
      \path [-stealth,solid,black,thick,draw=none] ( -2 cm,1) edge node[left=0]{}(-1.15 cm,1);
      \path [-stealth,solid,black,thick,draw=none] ( -1 cm,1) edge node[left=0]{}(-0.15 cm,1);
      \path [-stealth,solid,black,thick,draw=none] ( 0 cm,1) edge[bend left] node[left=0]{}(1.85 cm,1);
      \path [-stealth,solid,black,thick,draw=none] ( 2 cm,1) edge node[left=0]{}(2.85 cm,1);
      \path [-stealth,solid,black,thick,draw=none] ( 3 cm,1) edge node[left=0]{}(3.85 cm,1);
      \path [-stealth,solid,black,thick,draw=none, dotted] ( 4 cm,1) edge node[left=0]{}(5.85 cm,1);
            \path [-stealth,solid,black,thick,draw=none] ( 6 cm,1) edge node[left=0]{}(6.85 cm,1);

      \path [-stealth,solid,black,thick,draw=none] ( 7 cm,1) edge[bend left] node[left=0]{}(8.85 cm,1);      
      \path [-stealth,solid,black,thick,draw=none,dotted] ( 9 cm,1) edge node[left=0]{}(9.85 cm,1);      
      \path [-stealth,solid,black,thick,draw=none] ( 10 cm,1) edge node[left=0]{}(10.85 cm,1);      
      
       \path [-stealth,solid,black,thick,draw=none] ( 12 cm,1) edge node[left=0]{}(11.15 cm,1);      
       \path [-stealth,solid,black,thick,draw=none,dotted] ( 13 cm,1) edge node[left=0]{}(12.15 cm,1);      
       \path [-stealth,solid,black,thick,draw=none] ( 15 cm,1) edge[bend right] node[left=0]{}(13.15 cm,1);      
       \path [-stealth,solid,black,thick,draw=none] ( 16 cm,1) edge node[left=0]{}(15.15 cm,1);      
       \path [-stealth,solid,black,thick,draw=none] ( 17 cm,1) edge node[left=0]{}(16.15 cm,1);      
       \path [-stealth,solid,black,thick,draw=none, dotted] ( 19 cm,1) edge node[left=0]{}(17.15 cm,1);      
       \path [-stealth,solid,black,thick,draw=none] ( 20 cm,1) edge node[left=0]{}(19.15 cm,1);      
       \path [-stealth,solid,black,thick,draw=none] ( 21 cm,1) edge node[left=0]{}(20.15 cm,1);      
            \path [-stealth,solid,black,thick,draw=none] ( 22 cm,1) edge node[left=0]{}(21.15 cm,1);      
       \path [-stealth,solid,black,thick,draw=none] ( 24 cm,1) edge[bend right] node[left=0]{}(22.15 cm,1);      
       \path [-stealth,solid,black,thick,draw=none] ( 25 cm,1) edge node[left=0]{}(24.15 cm,1);      
       \path [-stealth,solid,black,thick,draw=none,dotted] ( 27 cm,1) edge node[left=0]{}(25.15 cm,1);      
       \path [-stealth,solid,black,thick,draw=none] ( 28 cm,1) edge node[left=0]{}(27.15 cm,1);      
       
            \path [-stealth,solid,black,thick,draw=none] ( 29 cm,1) edge node[left=0]{}(28.15 cm,1);      
       \path [-stealth,solid,black,thick,draw=none] ( 31 cm,1) edge[bend right] node[left=0]{}(29.15 cm,1);      
       \path [-stealth,solid,black,thick,draw=none, dotted] ( 33 cm,1) edge node[left=0]{}(31.15 cm,1);

  \end{tikzpicture}
\end{center}

\begin{lemma} For a weight $\l\in F^{(a,b)}$, if $\l+\ga\in F^{(a+1,b+1)}$, then the corresponding$c_{\l}$ and $c_{\l+\ga}$ are either in the same interval $J_j$ or adjacent ones. 
\end{lemma}

\begin{proof}  Given $c\in\frac{1}{2}\Z$, by theorem 6.5, there is at most one dominant $\l\in
F^{(a,b)}$, with $\l_c=\l$, for $c\in I_i$. We want to show that both $c$ and $c+\ga_4$ in the
same or adjacent intervals $I_j$.

Say $\l+\r=(b_1,b_2,b_3|b_4)$, then $b_4=c$ if $i=1,2,3,4$ and $b_4=-c$ if $i=5, 6,7,8$.

Assume $b_4=c\in I_i$ with $i=1,2,3,4$, we claim that there is no $\ga\in\D$ such
that $\l+\ga\in F^{(a+1,b+1)}$ and $-(b_4+\ga_4)\in I_i$ with $i=6,7,8$. If
$b_4\in I_i$ $i=1,2,3,4$, then $b_1-b_4=\frac{2a+b}{3}$, while if $b_4\in
I_i$ $i=6,7,8$, then $b_1-b_4=\frac{a+2b}{3}$. Now, if such $\ga$ exists, we
will have $b_1+\ga_1-b_4-\ga_4=\frac{2a+b}{3}+(\ga_1-\ga_4)=\frac{a+2b}{3}+1$,
which implies $\ga_1-\ga_4=\frac{-a+b}{3}+1=-n+1$. The last number must be in
the interval $[-1,1]$, since $\ga\in\D$. But this is only possible if $n=1$.

Similarly, if $-b_4=c\in I_i$ with $i=6,7,8$ and $\ga$ is such that
$b_4+\ga_4\in I_i$ with $i=1,2,3,4$ we have $b_1-b_4=\frac{a+2b}{3}$ and
$b_1+\ga_1-b_4-\ga_4=\frac{2a+b}{3}+1$, and we get
$\ga_1-\ga_4=\frac{2a+b}{3}+1-\frac{a+2b}{3}=\frac{a-b}{3}+1=n+1$. This is a
contradiction since $\ga_1-\ga_4\in[-1,1]$, for $\ga\in\D$. It is also
not possible to have $\l\in F^{(a,b)}$ with $\l+\r=(b_1,b_2,b_3|b_4)$ and
$-b_4\in I_5$ and $\ga\in\D$ with $\l+\ga\in F^{(a+1,b+1)}$ and $b_4+\ga_4\in I_i$
with $i=1,2,3$, since if $-b_4\in I_5$ implies
$0<b_4<\frac{a-b}{6}<\frac{a-b}{3}$ implying $b_4+\ga_4\in I_4$ or $I_5$.

The case $n=1$ can be checked separately. 
\end{proof}

The following lemma justifies the above picture.

\begin{lemma} For $\l\in F^{(a,b)}$, there is a unique $\ga\in\D$ such that
$\l+\ga\in\F^{(a+1,b+1)}$ is dominant, unless
$\l+\r=(a+b+\frac{1}{2},b+\frac{1}{2},\frac{1}{2}|\frac{a+2b}{3}+\frac{1}{2
})$ or
$\l+\r=(a+b+\frac{1}{2},a+\frac{1}{2},\frac{1}{2}|\frac{2a+b}{3}+\frac{1}{2
})$. 
\end{lemma}

\begin{proof} Assume $\ga\in\D$ is such that $\l+\ga\in F^{(a+1,b+1)}$. We first show that the $c$ corresponding to $\l+\ga+\r$ and $\l+\r$ is the same in generic cases. By \rrem{remark}, this will imply that there is at most one such $\ga$, proving the
uniqueness.

Assume that the last coordinate of $\l+\r$ is $c$. Then $\l+\r+\ga$ must have last coordinate $c\pm 1$, $c\pm\frac{1}{2}$, or $c$.

Thus for generic $\l$, $(\l+\r,\a)=0$ and $(\l+\ga+\r,\a)=0$ are true for the same $\a\in\D_{\bar{1}}^+$ (see \rrem{remark} above). That implies $(\ga,\a)=0$. Thus, $\ga\neq \d$. If $\ga$ is odd then $(\ga,\a)=0$ implies $\ga=\pm\a$, which is impossible, since then $\l$ and $\l+\ga$ correspond to the same central character from \rlem{weights in block}. If $\ga\neq\d$ is even the statement is clear.

The few exceptional cases occur around walls of the Weyl chamber, when $c=\frac{a+2b}{3}+1$, $\frac{a+2b}{3}+\frac{1}{2}$,
$\frac{2a+b}{6}+\frac{1}{2}$, $\frac{a+2b}{6}+\frac{1}{2}$,
$\frac{2a+b}{3}+\frac{1}{2}$, and $\frac{2a+b}{3}+1$. We can see that
only in the second and fifth places there are two such $\ga$.
\end{proof}

\begin{lemma}\label{xab} We have $T(L_{\l_i})=L_{\m_i}$, for all $\l_i\neq \l_c$ with $c=\frac{a+2b}{3}+\frac{1}{2}$ or $\frac{2a+b}{3}+\frac{1}{2}$.
\end{lemma}

\begin{proof} By definition, $T(L_{\l_i})=(L_{\l_i}\otimes\g)^{(a+1,b+1)}$. As one can see from the picture above, for each $i$, there is a unique dominant weight $\m_i$ in $\F^{(a+1, b+1)}$ of the form $\l_i+\gg$ with $\gg\in \D$. Thus, the lemma follows from \rlem{acyclicx}.
\end{proof}

\begin{lemma} For $c=\frac{a+2b}{3}+\frac{1}{2}=t_1+\frac{1}{2}$, we have $T(L_{\l_c})=L_{\m_{c'}}$ with $c'=\frac{a+2b}{3}=t_1$. Similarly, for $c=\frac{2a+b}{3}+\frac{1}{2}=t_2+\frac{1}{2}$, we have $T(L_{\l_c})=L_{\m_{c'}}$ with $c'=\frac{2a+b}{3}=t_2$.
\end{lemma}
 
\begin{proof} By definition, $T(L_{\l_c})=(L_{\l_c}\otimes\g)^{(a+1,b+1)}$. The only dominant weights with central character corresponding to block $\F^{(a+1, b+1)}$ of the form $\l_c+\gg$ with $\gg\in \D$ are $\m_c$ and $\m_{c'}$. The proof is similar to the proof of \rlem{special1122} using previous lemma.


\end{proof} 

\begin{theorem}\label{eqab} We have an equivalence between categories $\F^{(a, b)}$ and $\F^{(a+1, b+1)}$. 
\end{theorem}

\begin{proof} From previous lemma, for each $\l_i\in F^{(a+1,b+1)}$, $T(L_{\l_i})$ is a simple module in $\F^{(a+1,b+1)}$, we denote $L_{\m_i}=T(L_{\l_i})$ the simple module with highest weight $\m_i\in F^{(a+1,b+1)}$. We show that the the conditions of \rthm{eqgeneral} are satisfied.

For all $\m\neq \l_c\in F^{(a+1,b+1)}$ with $c=t_2+\frac{1}{2}$ or $t_1+\frac{1}{2}$, we have a unique $\gamma\in\D$, such that $\m+\gamma\in F^{(a,b)}$.

From the picture above, for $\m=\l_c\in F^{(a+1,b+1)}$ with $c=t_2+\frac{1}{2}$ or $t_1+\frac{1}{2}$, there are two possible $\gamma\in\D$ such that $\m+\gamma\in F^{(a,b)}$.

Here, $\m+\gamma=\l_{t_2+\frac{1}{2}}$ and $\l_{t_2+1}$ or $\m+\gamma=\l_{t_1+\frac{1}{2}}$ and $\l_{t_1+1}$ correspondingly such that $\l_{t_2+\frac{1}{2}}<\l_{t_2+1}$ and $\l_{t_1+\frac{1}{2}}<\l_{t_1+1}$. The theorem follows from \rthm{eqgeneral}. 






 \end{proof}

\subsection{Cohomology groups in the block $F^{(a, b)}$ with $a=b+3$.}

We let $b=1$. In the block $F^{(4, 1)}$, the dominant weights close to the walls of the Weyl chamber are denoted:

$\l_7+\r=(\frac{11}{2},\frac{3}{2},\frac{1}{2}|\frac{5}{2})$;

$\l_6+\r=(\frac{7}{2},\frac{3}{2},\frac{1}{2}|\frac{1}{2})$;

$\l_0+\r=(3,2,1|0)$;

$\l_1+\r=(\frac{7}{2},\frac{5}{2},\frac{3}{2}|\frac{3}{2})$;

$\l_2+\r=(4,3,1|2)$;

$\l_3+\r=(\frac{9}{2},\frac{7}{2},\frac{1}{2}|\frac{5}{2})$;

$\l_4+\r=(\frac{11}{2},\frac{9}{2},\frac{1}{2}|\frac{7}{2})$.

(Note that indices for $\l$ above are different from the index $c$ that corresponds to the last coordinate of $\l+\r$. )

\begin{lemma} For all $\l\in F^{(4,1)}$ such that $\l\neq\l_0$, we have $\G_1(G/B,\O_{\l})=0$. \end{lemma}
\begin{proof} For generic weights, this follows from \rlem{lemma 3 generic weights}. For weights close to the walls of the Weyl chamber, we compute from \rlem{3} in a similar way as for $\F^{(1,1)}$ in \rlem{i1122} or for generic weights. 
\end{proof}

\begin{lemma}\label{410} For non-generic weight $\l=\l_0\in F^{(4, 1)}$, we have $\G_0(G/B, \mathcal{O}_{\l_0})= L_{\l_0}$ and $\G_1(G/B, \mathcal{O}_{\l_0})= L_{\l_0}$. 
\end{lemma}

\begin{proof} From \rlem{3} and \rlem{2}, we have $[\G_0(G/B, \mathcal{O}_{\l_0}): L_{\l_0}]=1$ and $[\G_0(G/B, \mathcal{O}_{\l_0}): L_{\l_\s}]=0$ for $\s\neq \l_0$. Also, that $\G_i(G/B, \mathcal{O}_{\l_0})=0$ for $i>1$. 

Also, we have $0=sdim{\G_0(G/B, \mathcal{O}_{\l_0})}-sdim{\G_1(G/B, \mathcal{O}_{\l_0})}$. This implies $\G_1(G/B, \mathcal{O}_{\l_0})\neq 0$. Since \rlem{3} implies that any simple subquotient of $\G_1(G/B, \mathcal{O}_{\l_0})$ has highest weight less than $\l_0$, we must have $\G_1(G/B, \mathcal{O}_{\l_0})=L_{\l_0}$. 

\end{proof}

\begin{lemma}\label{exact41} For all simple modules $L_\l\in \F^{(4,1)}$ such that $\l\neq \l_0$, there is a unique dominant weight $\m\in F^{(4,1)}$ with $\m=\l-\sum_{i=1}^n\a_i$ with $\a_i\in\D^+_\1$ and $n\in\{ 1,2,3,4\}$ such that we have an exact sequence: 

\begin{displaymath}
    \xymatrix{
        0\ar[r] & L_{\l} \ar[r] & \G_0(G/B, \mathcal{O}_{\l})\ar[r] & L_{\m} \ar[r] & 0 }
\end{displaymath}

We also have $\G_i(G/B, \mathcal{O}_{\l})=0$ for $i>0$. 

\end{lemma}

\begin{proof} We have the lemma for generic weights, thus it remains to prove for the special weghts above. For  non-generic weight $\l=\l_4\in F^{(4, 1)}$, we have $\m=\l_3$. From \rlem{3}, we have $[\G_0(G/B, \mathcal{O}_{\l_4}): L_{\l_3}]\leq 1$ and $[\G_0(G/B, \mathcal{O}_{\l_4}): L_{\l_\s}]=0$ for $\s\neq \l_3, \l_4$. 

Also, we have: 
\begin{equation}\label{34}
0=sdim{\G_0(G/B, \mathcal{O}_{\l_4})}=sdim{L_{\l_4}}+[\G_0(G/B, \mathcal{O}_{\l_4}): L_{\l_3}]sdim{L_{\l_3}}.
\end{equation} 

Since, starting with generic weight, we have $sdim{L_{\l_4}}\neq 0$, this implies $[\G_0(G/B, \mathcal{O}_{\l_4}): L_{\l_3}]\neq 0$, proving the lemma for $\l=\l_4$. The proof is similar for $\l=\l_7$, we have $\m=\l_6$ and use that $sdimL_{\l_7}\neq 0$ starting with a generic weight.

Similarly, now using these cases we obtain first that $sdim{L_{\l_3}}\neq 0$, since we have $sdim{L_{\l_4}}\neq 0$, then we obtain exact sequences for $\l=\l_3$ with $\m=\l_2$. Similarly, for $\l=\l_6$ with $\m=\l_0$.

It remains to understand the cohomology groups for the dominant non-generic weights $\l_1$, $\l_2$. These cases are more complicated and we first prove the following lemma:

\begin{lemma}\label{41sdim} For all $\l\in F^{(4, 1)}$, we have $sdim{L_{\l}}=\pm d$, where $d=dim\,{L_\m(\g_x)}$, where $\m$ is from theorem \rthm{CBF}. 
\end{lemma}

\begin{proof} Starting with generic weights $\l$ and using the \rthm{character formula generic weights} for generic weight, we have $sdim{L_{\l}}=\pm d$ for generic weight. For the weights close to the walls of the Weyl chamber, we use the above lemmas and exact sequences to show this.

From exact sequences for $\l=\l_4$, $\l_7$, $\l_3$, and $\l_6$, we know that $sdim{L_i}=\pm d$ for $i=6,0,2$. Since, in each case we know that $\G_0(G/B, \mathcal{O}_{\l_i})=0$ and $sdim{L_j}=\pm d$ for the other $L_j$ in the exact sequence.

To prove that  $sdim{L_1}=\pm d$ is more challenging. We first apply translation functor $T$ to the dominant weights $\l_0$, $\l_1$, $\l_2$, $\l_6$ twice to get dominant weights $\l_0'$, $\l_1'$, $\l_2'$, $\l_6'$ in the equivalent block $\F^{(6, 3)}$.

The categories $\F^{(4, 1)}$ and $\F^{(6, 3)}$ are equivalent from \rthm{eqab}. Thus, by \rlem{eqcoho}, we have $[\G_0(G/B, \mathcal{O}_{\l'}):L_{\l'}]=[\G_0(G/B, \mathcal{O}_{\l}):L_{\l}]$.

We apply odd reflections with respect to odd roots $\b, \b', \b'',\b'''$ to obtain dominant weights $\l_0''$, $\l_1''$, $\l_2''$, $\l_6''$ with respect to another Borel subalgebra $B''$.

We get the following:

$\l_6'+\r=(\frac{11}{2},\frac{7}{2},\frac{1}{2}|\frac{1}{2})$;

$\l_0'+\r=(5,4,1|0)$;

$\l_1'+\r=(5,4,2|1)$;

$\l_2'+\r=(\frac{11}{2},\frac{7}{2},\frac{5}{2}|\frac{3}{2})$;

$\l_3'+\r=(\frac{13}{2},\frac{7}{2},\frac{5}{2}|\frac{5}{2})$;

$\l_4'+\r=(7,4,2|3)$.

After applying the odd reflections we get the following dominant weights with respect to the new Borel $B''$:

$\l_6''+\r''=(\frac{11}{2},\frac{7}{2},\frac{1}{2}|\frac{1}{2})$;

$\l_0''+\r''=(\frac{9}{2},\frac{9}{2},\frac{3}{2}|\frac{1}{2})$;

$\l_1''+\r''=(5,4,2|1)$;

$\l_2''+\r''=(\frac{11}{2},\frac{7}{2},\frac{5}{2}|\frac{3}{2})$;

$\l_3''+\r''=(\frac{13}{2},\frac{7}{2},\frac{5}{2}|\frac{5}{2})$;

$\l_4''+\r''=(7,4,2|3)$.

From \rlem{oddrootbase}, the positive odd roots with respect to the new Borel $B''$ are all the odd roots with first coordinate $\frac{1}{2}$.

From \rlem{3} with respect to $B''$, we have $[\G_0(G/B'', \mathcal{O}_{\l_2''}):L_{\l_1''}]\leq 1$ and $[\G_0(G/B'', \mathcal{O}_{\l_2''}):L_{\l_\s''}]=0$ for all $\s\neq \l_1'',\l_2''$.

We also have $$0=sdim{\G_0(G/B'', \mathcal{O}_{\l_2''})}=$$ $$= sdim{L_{\l_2''}}+[\G_0(G/B'', \mathcal{O}_{\l_2''}):L_{\l_1''}]sdim{L_{\l_1''}}$$ implying that 
$[\G_0(G/B'', \mathcal{O}_{\l_2''}):L_{\l_1''}]=1$ and $sdim{L_{\l_1''}}=\pm d$. Now we have $sdim{L_{\l_1'}}=sdim{L_{\l_1''}}=\pm d$. 
\end{proof}

From \rlem{41sdim}, we know that $sdim{L_{\l_1}}=-sdim{L_{\l_0}}=\pm d$. Using this and similar argument as for $\l=\l_4$, we have the exact sequence for $\l=\l_1$ with $\m=\l_0$. 

It remains to prove the lemma for $\l=\l_2$. In this case, we have $\m=\l_2$. Follows from computation using \rlem{3}, that $[\G_0(G/B, \mathcal{O}_{\l_2}): L_{\l_1}]\leq 1$, $[\G_0(G/B, \mathcal{O}_{\l_2}): L_{\l_6}]\leq 1$, and $[\G_0(G/B, \mathcal{O}_{\l_2}): L_{\l_\s}]=0$ for $\s\neq \l_1, \l_2, \l_6$.

We also know that $$0=sdim{\G_0(G/B, \mathcal{O}_{\l_2})}=sdim{L_{\l_2}}+$$ $$+[\G_0(G/B, \mathcal{O}_{\l_2}): L_{\l_1}]sdim{L_{\l_1}}+[\G_0(G/B, \mathcal{O}_{\l_2}): L_{\l_6}]sdim{L_{\l_6}}.$$ From \rlem{41sdim}, we know that $sdim{L_{\l_2}}=\pm sdim{L_{\l_1}}=\pm sdim{L_{\l_6}}=\pm d\neq 0$. This implies that one of the numbers $[\G_0(G/B, \mathcal{O}_{\l_2}): L_{\l_1}]$ or $[\G_0(G/B, \mathcal{O}_{\l_2}): L_{\l_6}]$ is one and another is zero.

We prove $[\G_0(G/B, \mathcal{O}_{\l_2}): L_{\l_6}]=0$.

The odd reflections with respect to the weight $\l_2'$ are typical, which means that the weight doesn't change. Using the fact from \cite{Penkov} that for an odd typical reflection $r$ with respect to the weight $\l$, we have $\G_0(G/r(B), \mathcal{O}_{r(\l)})=\G_0(G/B, \mathcal{O}_{\l})$, this implies that $\G_0(G/B, \mathcal{O}_{\l_2'})=\G_0(G/B'', \mathcal{O}_{\l_2''})$. An odd reflection $r$ {\it{typical}} with respect to the weight $\l$ if $r(\l)=\l$. The later module has subquotients $L_{\l_1''}=L_{\l_1'}$ and $L_{\l_0''}=L_{\l_0'}$. Thus, $[\G_0(G/B, \mathcal{O}_{\l_2'}): L_{\l_6'}]=0$.

Since $T$ is an equivalence of categories from \rthm{eq4152}, from \rlem{eqcoho} we have $[\G_0(G/B, \mathcal{O}_{\l'}):L_{\l'}]=[\G_0(G/B, \mathcal{O}_{\l}):L_{\l}]$, which proves the exact sequence. 

\end{proof}

\subsection{Cohomology groups in the block $F^{(a, b)}$ with $a=b+3n$, $n>1$.}

For $n>1$, we assume $b=1$. The dominant weights close to the walls of the Weyl chamber have different arrangements in this case and they are correspondingly denoted:

$\l_{t_2+1}+\r=(a+2,2,1|t_2+1)$;

$\l_{t_2+\frac{1}{2}}+\r=(a+\frac{3}{2},\frac{3}{2},\frac{1}{2}|t_2+\frac{1}{2})$;

$\l_{t_3-\frac{1}{2}}+\r=(a-\frac{1}{2},\frac{3}{2},\frac{1}{2}|t_3-\frac{1}{2})$;

$\l_{t_3-1}+\r=(a-1,2,1|t_3-1)$;

...

$\l_{\frac{1}{2}}+\r=(t_1+\frac{1}{2},t_2-\frac{1}{2},t_3-\frac{1}{2}|\frac{1}{2})$;

$\l_{0}+\r=(t_1,t_2,t_3|0)$;

$\l_{-\frac{1}{2}}+\r=(t_1-\frac{1}{2},t_2+\frac{1}{2},t_3+\frac{1}{2}|\frac{1}{2})$;

...

$\l_{-\frac{t_3}{2}+1}+\r=(\frac{a}{2}+\frac{3}{2},\frac{a}{2}-\frac{1}{2},\frac{a}{2}-\frac{3}{2}|\frac{t_3}{2}-1)$;

$\l_{-\frac{t_3}{2}+\frac{1}{2}}+\r=(\frac{a}{2}+1,\frac{a}{2},\frac{a}{2}-1|\frac{t_3}{2}-\frac{1}{2})$;

$\l_{-\frac{t_2}{2}-\frac{1}{2}}+\r=(\frac{a}{2}+\frac{3}{2},\frac{a}{2}+\frac{1}{2},\frac{a}{2}-\frac{1}{2}|\frac{t_2}{2}+\frac{1}{2})$;

$\l_{-\frac{t_2}{2}-1}+\r=(\frac{a}{2}+2,\frac{a}{2}+1,\frac{a}{2}-1|\frac{t_2}{2}+1)$;

...

$\l_{-t_1+1}+\r=(a,a-1,1|t_1-1)$;

$\l_{-t_1+\frac{1}{2}}+\r=(a+\frac{1}{2},a-\frac{1}{2},\frac{1}{2}|t_1-\frac{1}{2})$;

$\l_{-t_1-\frac{1}{2}}+\r=(a+\frac{3}{2},a+\frac{1}{2},\frac{1}{2}|t_1+\frac{1}{2})$;

$\l_{-t_1-1}+\r=(a+2,a+1,1|t_1+1)$.

\begin{lemma} For all $\l\in F^{(a,b)}$ such that $\l\neq\l_0$, we have $\G_1(G/B,\O_{\l})=0$. \end{lemma}
\begin{proof} For generic weights, this follows from \rlem{lemma 3 generic weights}. For weights close to the walls of the Weyl chamber, we compute from \rlem{3} in a similar way as for $\F^{(1,1)}$ in \rlem{i1122} or for generic weights. 
\end{proof}

\begin{lemma}\label{exact41} For all simple modules $L_\l\in \F^{(a,1)}$ such that $\l\neq \l_0$, there is a unique dominant weight $\m\in F^{(a,1)}$ with $\m=\l-\sum_{i=1}^n\a_i$ with $\a_i\in\D^+_\1$ and $n\in\{ 1,2,3,4\}$ such that we have an exact sequence: 

\begin{displaymath}
    \xymatrix{
        0\ar[r] & L_{\l} \ar[r] & \G_0(G/B, \mathcal{O}_{\l})\ar[r] & L_{\m} \ar[r] & 0 }
\end{displaymath}

We also have $\G_i(G/B, \mathcal{O}_{\l})=0$ for $i>0$. 

\end{lemma}

\begin{proof} We proved the lemma for generic weights, thus it remains to prove for the special weghts above. For  non-generic weight $\l=\l_{t_2+1}\in F^{(a, 1)}$, we have $\m=\l_{t_2+\frac{1}{2}}$. From \rlem{3}, we have $[\G_0(G/B, \mathcal{O}_{\l_{t_2+1}}): L_{\l_{t_2+\frac{1}{2}}}]\leq 1$ and $[\G_0(G/B, \mathcal{O}_{\l_{t_2+1}}): L_{\l_\s}]=0$ for $\s\neq \l_{t_2+\frac{1}{2}}, \l_{t_2+1}$.

We also know that $$0=sdim{\G_0(G/B, \mathcal{O}_{\l_{t_2+1}})}=sdim{L_{\l_{t_2+1}}}+[\G_0(G/B, \mathcal{O}_{\l_{t_2+1}}): L_{\l_{t_2+\frac{1}{2}}}]sdim{L_{\l_{t_2+\frac{1}{2}}}}.$$ Since, starting with generic weight, we know that $sdim{L_{\l_{t_2+1}}}\neq 0$, we must have that $[\G_0(G/B, \mathcal{O}_{\l_{t_2+1}}): L_{\l_{t_2+\frac{1}{2}}}]\neq 0$, proving the lemma.

Similarly, we obtain the exact sequences for weights $\l=\l_{-t_1-1}\in F^{(a, 1)}$ with $\m=\l_{-t_1-\frac{1}{2}}$, then $\l=\l_{-t_1-\frac{1}{2}}\in F^{(a, 1)}$ with $\m=\l_{-t_1+\frac{1}{2}}$, and     $\l=\l_{t_2+\frac{1}{2}}$ with $\m=\l_{t_3-\frac{1}{2}}$. 
   
Using induction, for non-generic weight $\l=\l_c\in F^{(a, 1)}$ with $c\in I_4$, we have an exact sequence for $\m=\l_c-\b_4$ with $\b_4=(\frac{1}{2},-\frac{1}{2},-\frac{1}{2}|\frac{1}{2})$. Similarly, For non-generic weight $\l=\l_c\in F^{(a, 1)}$ with $c\in I_7$ such that $c\neq -\frac{t_2}{2}-\frac{1}{2}, -\frac{t_2}{2}-1$, we have an exact sequence with $\m=\l_c-\b_7$ with $\b_7=(\frac{1}{2},\frac{1}{2},-\frac{1}{2}|\frac{1}{2})$. For non-generic weight $\l=\l_c\in F^{(a, 1)}$ with $c\in I_5$, we have an exact sequence with $\m=\l_c-\b_5$ with $\b_5=(-\frac{1}{2},\frac{1}{2},\frac{1}{2}|\frac{1}{2})$.

For the remaining cases, we will first prove the following lemmas: 

\begin{lemma} For $\l\in F^{(a,b)}$ such that $\l=\l_0$, we have $\G_0(G/B,\O_{\l_0})=\G_1(G/B,\O_{\l_0})=L_{\l_0}$. 
\end{lemma}
\begin{proof} From \rlem{3}, we have that any simple subquotient in  $\G_0(G/B,\O_{\l_0})$ has weight less or equal to $\l_0$. Thus, $\G_0(G/B,\O_{\l_0})=L_{\l_0}$. 

Also, we have $$sdim\,\G_0(G/B,\O_{\l_0})=sdim\,\G_1(G/B,\O_{\l_0})$$ and $\G_0(G/B,\O_{\l_0})=L_{\l_0}$. Since from above, $sdim \l_0\neq 0$, we have $\G_1(G/B,\O_{\l_0})\neq 0$. From \rlem{3}, we can see that any simple subquotient in  $\G_1(G/B,\O_{\l_0})$ has weight less or equal to $\l_0$. This proves the lemma. 
\end{proof}

It remains to understand the cohomology groups for weights with $c\in I_5$, and weights $\l_c$ with $c\neq -\frac{t_2}{2}-\frac{1}{2}, -\frac{t_2}{2}-1\in I_7$. We also need the following lemma: 

\begin{lemma}\label{newsuper} We have $sdimL_{\l_{-\frac{t_3}{2}-2}}= -sdimL_{\l_{-\frac{t_3}{2}}}$. 
\end{lemma}

\begin{proof} Follows from \rlem{generic}, the previous exact sequences for $I_5$, and the fact that the parity of the weight in $I_6$ will coincide with the sign of the superdimension.
\end{proof}





Since \rlem{3}, doesn't give good description of cohomology groups for $\l_{-\frac{t_2}{2}-1}, \l_{-\frac{t_2}{2}-\frac{1}{2}}\in F^{(a,1)}$, we first apply translation functor to the dominant weights $\l_{-\frac{t_3}{2}+1}$, $\l_{-\frac{t_3}{2}+\frac{1}{2}}$, $\l_{-\frac{t_2}{2}-\frac{1}{2}}$, $\l_{-\frac{t_2}{2}-1}$ twice to get dominant weights $\l_{-\frac{t_3}{2}+1}'$, $\l_{-\frac{t_3}{2}+\frac{1}{2}}'$, $\l_{-\frac{t_2}{2}-\frac{1}{2}}'$, $\l_{-\frac{t_2}{2}-1}'$ in the equivalent block $\F^{(a+2, 3)}$.

Then we apply odd reflections with respect to odd roots $\b, \b', \b'',\b'''$ to obtain dominant weights $\l''_{\frac{-t_3}{2}+1}$, $\l''_{-\frac{t_3}{2}+\frac{1}{2}}$, $\l''_{-\frac{t_2}{2}-\frac{1}{2}}$, $\l''_{-\frac{t_2}{2}-1}$ with respect to another Borel subalgebra $B''$.

We have: 

$\l_{-\frac{t_3}{2}+1}+\r=(\frac{a}{2}+\frac{3}{2},\frac{a}{2}-\frac{1}{2},\frac{a}{2}-\frac{3}{2}|\frac{t_3}{2}-1)$;

$\l_{-\frac{t_3}{2}+\frac{1}{2}}+\r=(\frac{a}{2}+1,\frac{a}{2},\frac{a}{2}-1|\frac{t_3}{2}-\frac{1}{2})$;

$\l_{-\frac{t_2}{2}-\frac{1}{2}}+\r=(\frac{a}{2}+\frac{3}{2},\frac{a}{2}+\frac{1}{2},\frac{a}{2}-\frac{1}{2}|\frac{t_2}{2}+\frac{1}{2})$;

$\l_{-\frac{t_2}{2}-1}+\r=(\frac{a}{2}+2,\frac{a}{2}+1,\frac{a}{2}-1|\frac{t_2}{2}+1)$.

After applying the translation functor twice, we have:

$\l_{-\frac{t_3}{2}+1}'+\r=(\frac{a}{2}+2,\frac{a}{2}+1,\frac{a}{2}-2|\frac{t_3}{2}-\frac{1}{2})$;

$\l_{-\frac{t_3}{2}+\frac{1}{2}}'+\r=(\frac{a}{2}+2,\frac{a}{2}+1,\frac{a}{2}-1|\frac{t_3}{2}+\frac{1}{2})$;

$\l_{-\frac{t_2}{2}-\frac{1}{2}}'+\r=(\frac{a}{2}+\frac{5}{2},\frac{a}{2}+\frac{1}{2},\frac{a}{2}-\frac{1}{2}|\frac{t_3}{2}+1)$;

$\l_{-\frac{t_2}{2}-1}'+\r=(\frac{a}{2}+\frac{7}{2},\frac{a}{2}+\frac{1}{2},\frac{a}{2}-\frac{1}{2}|\frac{t_3}{2}+2)$. 

After applying odd reflections we have: 

$\l_{-\frac{t_3}{2}+1}''+\r''=(\frac{a}{2}+\frac{3}{2},\frac{a}{2}+\frac{3}{2},\frac{a}{2}-\frac{3}{2}|\frac{t_3}{2})$;

$\l_{-\frac{t_3}{2}+\frac{1}{2}}''+\r''=(\frac{a}{2}+2,\frac{a}{2}+1,\frac{a}{2}-1|\frac{t_3}{2}+\frac{1}{2})$;

$\l_{-\frac{t_2}{2}-\frac{1}{2}}''+\r''=(\frac{a}{2}+\frac{5}{2},\frac{a}{2}+\frac{1}{2},\frac{a}{2}-\frac{1}{2}|\frac{t_3}{2}+1)$;

$\l_{-\frac{t_2}{2}-1}''+\r''=(\frac{a}{2}+\frac{7}{2},\frac{a}{2}+\frac{1}{2},\frac{a}{2}-\frac{1}{2}|\frac{t_3}{2}+2)$.

From \rlem{oddrootbase}, the positive odd roots with respect to the new Borel $B''$ are all the odd roots with first coordinate $\frac{1}{2}$. Now computation using \rlem{3} with respect to $B''$, implies:

$[\G_0(G/B'', \mathcal{O}_{\l''_{-\frac{t_3}{2}-2}}):L_{\l''_{-\frac{t_3}{2}-1}}]\leq 2$ and

$[\G_0(G/B'', \mathcal{O}_{\l''_{-\frac{t_3}{2}-2}}):L_{\l''_{-\frac{t_3}{2}-\frac{1}{2}}}]\leq 1$ and

$[\G_0(G/B'', \mathcal{O}_{\l''_{-\frac{t_3}{2}-2}}):L_{\l''_{-\frac{t_3}{2}}}]=0$.

Also, 

$[\G_0(G/B'', \mathcal{O}_{\l''_{-\frac{t_3}{2}-1}}):L_{\l''_{-\frac{t_3}{2}-\frac{1}{2}}}]=1$ and

$[\G_0(G/B'', \mathcal{O}_{\l''_{-\frac{t_3}{2}-1}}):L_{\l''_{-\frac{t_3}{2}}}]=0$.

Also, 

$[\G_0(G/B'', \mathcal{O}_{\l''_{-\frac{t_3}{2}-\frac{1}{2}}}):L_{\l''_{-\frac{t_3}{2}}}]=1$.

All other multiplicities are zero.

Since the odd reflections with respect to the weight $\l_{-\frac{t_3}{2}-2}$ are typical, we have the first equality below. And, since $T$ is an equivalence, by \rlem{eqcoho} we have the second equality. From \rlem{3} with respect to Borel $B$, we have the equality to $0$.

$$[\G_0(G/B'', \mathcal{O}_{\l''_{-\frac{t_3}{2}-2}}):L_{\l''_{-\frac{t_3}{2}-1}}]=[\G_0(G/B, \mathcal{O}_{\l'_{-\frac{t_3}{2}-2}}):L_{\l'_{-\frac{t_3}{2}-1}}]=$$ $$=[\G_0(G/B, \mathcal{O}_{\l_{-\frac{t_3}{2}-2}}):L_{\l_{-\frac{t_3}{2}-1}}]\leq 1.$$  

We also have $0=sdim{\G_0(G/B'', \mathcal{O}_{\l''_{-\frac{t_3}{2}-\frac{1}{2}}})}=sdimL_{\l''_{-\frac{t_3}{2}-\frac{1}{2}}}+[\G_0(G/B'', \mathcal{O}_{\l''_{-\frac{t_3}{2}-\frac{1}{2}}}):L_{\l''_{-\frac{t_3}{2}}}] sdimL_{\l''_{-\frac{t_3}{2}}}=sdimL_{\l''_{-\frac{t_3}{2}-\frac{1}{2}}}+sdimL_{\l''_{-\frac{t_3}{2}}}$, implying $$sdimL_{\l''_{-\frac{t_3}{2}-\frac{1}{2}}}=-sdimL_{\l''_{-\frac{t_3}{2}}}.$$

Similarly, we get:  $$sdimL_{\l''_{-\frac{t_3}{2}-\frac{1}{2}}}=-sdimL_{\l''_{-\frac{t_3}{2}-1}}.$$

We have: 

$$0=sdim{\G_0(G/B'', \mathcal{O}_{\l''_{-\frac{t_3}{2}-2}})}=$$ $$=sdimL_{\l''_{-\frac{t_3}{2}-2}}+[\G_0(G/B'', \mathcal{O}_{\l''_{-\frac{t_3}{2}-2}}):L_{\l''_{-\frac{t_3}{2}-1}}] sdimL_{\l''_{-\frac{t_3}{2}-1}}+$$ $$+[\G_0(G/B'', \mathcal{O}_{\l''_{-\frac{t_3}{2}-2}}):L_{\l''_{-\frac{t_3}{2}-\frac{1}{2}}}] sdimL_{\l''_{-\frac{t_3}{2}-\frac{1}{2}}}.$$

From above $[\G_0(G/B'', \mathcal{O}_{\l''_{-\frac{t_3}{2}-2}}):L_{\l''_{-\frac{t_3}{2}-1}}]\leq 1$ and $[\G_0(G/B'', \mathcal{O}_{\l''_{-\frac{t_3}{2}-2}}):L_{\l''_{-\frac{t_3}{2}-\frac{1}{2}}}] \leq 1$.

Since $sdimL_{\l_{-\frac{t_3}{2}-2}}= -sdimL_{\l_{-\frac{t_3}{2}}}$ from \rlem{newsuper}, we must have $[\G_0(G/B'', \mathcal{O}_{\l''_{-\frac{t_3}{2}-2}}):L_{\l''_{-\frac{t_3}{2}-1}}]=1$ and 
 $[\G_0(G/B'', \mathcal{O}_{\l''_{-\frac{t_3}{2}-2}}):L_{\l''_{-\frac{t_3}{2}-\frac{1}{2}}}]=0$.

Again using the fact that the odd reflections were typical with respect to $\l_{-\frac{t_3}{2}-2}$ and \rlem{eqcoho}, we have: $[\G_0(G/B, \mathcal{O}_{\l_{-\frac{t_3}{2}-2}}):L_{\l_{-\frac{t_3}{2}-1}}]=1$ and 
 $[\G_0(G/B, \mathcal{O}_{\l_{-\frac{t_3}{2}-2}}):L_{\l_{-\frac{t_3}{2}-\frac{1}{2}}}]=0$.

Similarly, we obtain the second exact sequence. 

\end{proof}

\section{Equivalence of blocks in $G(3)$}

\subsection{Equivalence of blocks $\F^{1}$ and $\F^{3}$} 

Let $\g=G(3)$. We prove the equivalence of the blocks $\F^{1}$ and $\F^{3}$ as the first step of mathematical induction of proving the equivalence of the blocks $\F^{a}$ and $\F^{a+2}$. As before we have a picture of the translator functor from block $\F^{1}$ to $\F^{3}$, which is defined by $T(L_{\l})=(L_{\l}\otimes\g)^{3}$. 

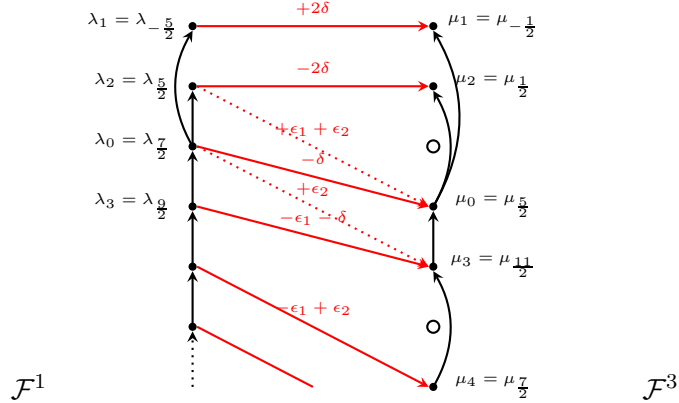
\begin{center}
  \begin{tikzpicture}[scale=0.4]
  
    \draw (-2.5,0) node[anchor=east]  {$\F^{1}$}; 
    \draw (+18.5,0) node[anchor=east]  {$\F^{3}$}; 
    
  \draw[xshift=1 cm,thick,fill=black] (1, 12 cm) circle (.1cm);
    \draw[xshift=1 cm,thick,fill=black] (1, 10 cm) circle (.1cm);
      \draw[xshift=1 cm,thick,fill=black] (1, 8 cm) circle (.1cm);
        \draw[xshift=1 cm,thick,fill=black] (1, 6 cm) circle (.1cm);
        \draw[xshift=1 cm,thick,fill=black] (1, 4 cm) circle (.1cm);
        \draw[xshift=1 cm,thick,fill=black] (1, 2 cm) circle (.1cm);
        
   \draw[xshift=0 cm] (0, 12 cm) node[anchor=center]  {{\tiny $\l_1=\l_{-\frac{5}{2}}$}};
    \draw[xshift=0 cm] (0, 10 cm) node[anchor=center]  {{\tiny $\l_2=\l_{\frac{5}{2}}$}};
     \draw[xshift=0 cm] (0, 8 cm) node[anchor=center]  {{\tiny $\l_0=\l_{\frac{7}{2}}$}};
      \draw[xshift=0 cm] (0, 6 cm) node[anchor=center]  {{\tiny $\l_3=\l_{\frac{9}{2}}$}};
      
  \draw[xshift=5 cm,thick,fill=black] (5, 12 cm) circle (.1cm);
    \draw[xshift=5 cm,thick,fill=black] (5, 10 cm) circle (.1cm);
     \draw[xshift=5 cm,thick, black] (5, 8 cm) circle (.2cm);
      \draw[xshift=5 cm,thick,fill=black] (5, 6 cm) circle (.1cm);
        \draw[xshift=5 cm,thick, fill=black] (5, 4 cm) circle (.1cm);
         \draw[xshift=5 cm,thick,black] (5, 2 cm) circle (.2cm);
                \draw[xshift=5 cm,thick,fill=black] (5, 0 cm) circle (.1cm);
        
   \draw[xshift=6 cm] (6, 12 cm) node[anchor=center]  {{\tiny $\m_1=\m_{-\frac{1}{2}}$}};
    \draw[xshift=6 cm] (6, 10 cm) node[anchor=center]  {{\tiny $\m_2=\m_{\frac{1}{2}}$}};
        \draw[xshift=6 cm] (6, 8 cm) node[anchor=center]  {$$};
     \draw[xshift=6 cm] (6, 6 cm) node[anchor=center]  {{\tiny $\m_0=\m_{\frac{5}{2}}$}};
         \draw[xshift=6 cm] (6, 0 cm) node[anchor=center]  {{\tiny $\m_4=\m_{\frac{7}{2}}$}};
      \draw[xshift=6 cm] (6, 4 cm) node[anchor=center]  {{\tiny $\m_3=\m_{\frac{11}{2}}$}};
      
       \foreach \y in {1,2,3,...,4}
    \draw[-stealth, xshift=1 cm,thick] (1, 2*\y cm) -- +(0, 1.85 cm);
     \foreach \y in {0}
    \draw[-stealth, xshift=1 cm,dotted,thick] (1, 2*\y cm) -- +(0, 1.85 cm);
    
    \path [-stealth, solid,black,thick,draw=none] (2, 8 cm) edge[bend left] (2, 11.85 cm);
    \path [-stealth, solid,black,thick,draw=none]  (10, 6 cm) edge[bend right] (10, 9.85 cm);
    \path [-stealth, solid,black,thick,draw=none]  (10, 0 cm) edge[bend right] (10, 3.85 cm);
    \path [-stealth, solid,black,thick,draw=none]  (10, 6 cm) edge[bend right] (10, 11.85 cm);
    \path [-stealth, solid,black,thick,draw=none]  (10, 4 cm) edge (10, 5.85 cm);
    
     \path [-stealth, solid,red,thick,draw=none]  (2.15, 12 cm) edge node[above=0mm]{\tiny $+2\d$} (9.85, 12 cm) ;
    
   \path [-stealth, solid,red,thick,draw=none]  (2.15, 10 cm) edge node[above=0]{\tiny $-2\d$} (9.85, 10 cm) ;
    
    \path [-stealth, solid,red,dotted,thick,draw=none]  (2.15, 10 cm) edge node[above=0mm,draw=none,rectangle]{\tiny $+\e_1+\e_2$} (9.85, 6 cm) ;
    
        \path [-stealth, solid,red,dotted,thick,draw=none]  (2.15, 8 cm) edge node[above=0mm,draw=none,rectangle]{\tiny $+\e_2$} (9.85, 4 cm) ;
    
    \path [-stealth, solid,red,thick,draw=none]  (2.15, 8 cm) edge node[above=0]{\tiny $-\d$} (9.85, 6 cm) ;
    
    \path [-stealth,solid,red,thick,draw=none]  (2.15, 6 cm) edge node[draw=none,rectangle,above=0.1mm]{\tiny $-\e_1-\d$} (9.85, 4 cm) ;
    
     \path [-stealth,solid,red,thick,draw=none]  (2.15, 4 cm) edge node[draw=none,rectangle,above=0.1mm]{\tiny $-\e_1+\e_2$} (9.85, 0 cm) ;
    
    \path [solid,red,thick,draw=none]  (2.15, 2 cm) edge (6, 0 cm) ;

  \end{tikzpicture}
\end{center}

In the above picture we have:
$\l_1+\r=(2,3|-\frac{5}{2})$;
$\l_2+\r=(2,3|\frac{5}{2})$;
$\l_0+\r=(3,4|\frac{7}{2})$;
$\l_3+\r=(4,5|\frac{9}{2})$;
$\m_1+\r=(2,3|-\frac{1}{2})$;
$\m_2+\r=(2,3|\frac{1}{2})$;
$\m_0+\r=(3,4|\frac{5}{2})$;
$\m_3+\r=(3,5|\frac{7}{2})$.
Note that the indices are distinct from the index $c$ describing $\l$, they are described in the picture. 

\begin{theorem}\label{eq02} The blocks $\F^{1}$ and $\F^{3}$ are equivalent as categories. 
\end{theorem}
\begin{proof} The proof is similar as the proof of theorem 6.13 for $F(4)$. However, as we see in the picture above, in this case we have more complications at weight $\l_0$. Unlike before, there are two dominant weights of the form $\l_0+\gamma$ with $\gamma\in\D$. This we prove additionally that $T(L_{\l_0})=L_{\l_0}$. 

We first prove $T(L_{\l_i})=L_{\m_i}$, for all $i\neq 2,0$ as for $F(4)$ case. By definition, $T(L_{\l_0})=(L_{\l_0}\otimes\g)^{3}$. The only dominant weights in the block $\F^{3}$ of the form $\l_0+\gg$ with $\gg\in \D$ are $\m_3$ and $\m_0$.

The module $L_{\l_3}$ is a quotient of $\G_0(G/B,\O_{\l_3})$ from \rlem{z}. We obtain the following exact sequence:

$$0\rightarrow L_{\l_0} \rightarrow \G_0(G/B,\O_{\l_3})\rightarrow L_{\l_3}\rightarrow 0.$$

Since $T$ is an exact functor, we get the following exact sequence: 

$$0\rightarrow T(L_{\l_0})\rightarrow T(\G_0(G/B,\O_{\l_3}))\rightarrow T(L_{\l_3})\rightarrow 0.$$

We have $T(L_{\l_3})=L_{\m_3}$. By \rlem{6} and \rlem{w}, we have $$T(\G_0(G/B,\O_{\l_3}))=\G_0(G/B,T(\O_{\l_3}))=\G_0(G/B,\O_{\m_3}).$$ The later module has a unique quotient $L_{\m_3}$. Therefore, $T(L_{\l_0})$ has no simple subquotient $L_{\m_3}$, which proves the statement. 

Now using that $T(L_{\l_i})=L_{\m_i}$, for all $i\neq 2$, we prove that $T(L_{\l_2})=L_{\m_2}$ as in the \rlem{special1122}. The theorem will follow as in the proof of \rthm{eq1122}.
\end{proof}

\subsection{Equivalence of blocks $\F^{a}$ and $\F^{a+2}$} 
Let $\g=G(3)$. This section is the inductive step of the proof of equivalence of the blocks of $G(3)$. We prove that all blocks are equivalent and find all cohomology groups. The picture of translator functor from block $\F^{a}$ to $\F^{a+2}$ is below. It is defined by $T(L_{\l})=(L_{\l}\otimes\g)^{a+2}$. 

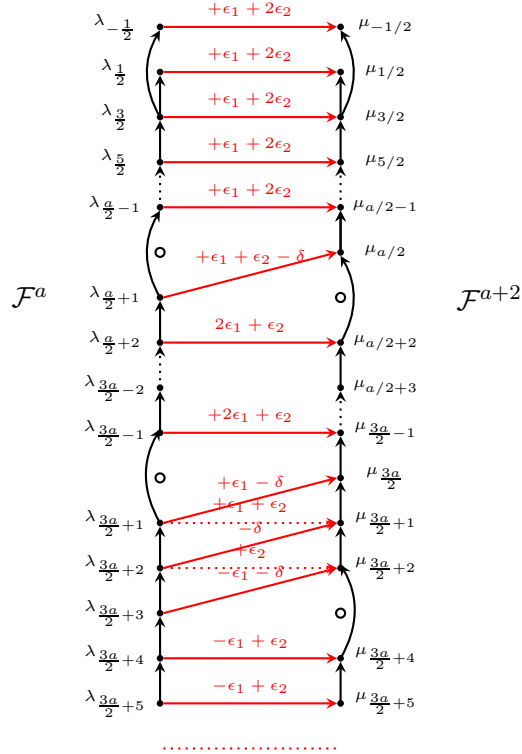
\begin{center}
  \begin{tikzpicture}[scale=.3]
  
    \draw (-2.5,0) node[anchor=east]  {$\F^{a}$}; 
    \draw (+18.5,0) node[anchor=east]  {$\F^{a+2}$}; 
    
  \draw[xshift=1 cm,thick,fill=black] (1, 12 cm) circle (.1cm);
    \draw[xshift=1 cm,thick,fill=black] (1, 10 cm) circle (.1cm);
      \draw[xshift=1 cm,thick,fill=black] (1, 8 cm) circle (.1cm);
        \draw[xshift=1 cm,thick,fill=black] (1, 6 cm) circle (.1cm);
        \draw[xshift=1 cm,thick,fill=black] (1, 4 cm) circle (.1cm);
        \draw[xshift=1 cm,thick,black] (1, 2 cm) circle (.2cm);
                \draw[xshift=1 cm,thick,fill=black] (1, 0 cm) circle (.1cm);
        \draw[xshift=1 cm,thick,fill=black] (1, -2 cm) circle (.1cm);
        \draw[xshift=1 cm,thick,fill=black] (1, -4 cm) circle (.1cm);
        \draw[xshift=1 cm,thick,fill=black] (1, -6 cm) circle (.1cm);
        \draw[xshift=1 cm,thick,black] (1, -8 cm) circle (.2cm);
        \draw[xshift=1 cm,thick,fill=black] (1, -10 cm) circle (.1cm);
        \draw[xshift=1 cm,thick,fill=black] (1, -12 cm) circle (.1cm);
        \draw[xshift=1 cm,thick,fill=black] (1, -14 cm) circle (.1cm);
        \draw[xshift=1 cm,thick,fill=black] (1, -16 cm) circle (.1cm);
        \draw[xshift=1 cm,thick,fill=black] (1, -18 cm) circle (.1cm);

   \draw[xshift=0 cm] (0, 12 cm) node[anchor=center]  {{\tiny$\l_{-\frac{1}{2}}$}};
    \draw[xshift=0 cm] (0, 10 cm) node[anchor=center]  {{\tiny$\l_{\frac{1}{2}}$}};
     \draw[xshift=0 cm] (0, 8 cm) node[anchor=center]  {{\tiny$\l_{\frac{3}{2}}$}};
      \draw[xshift=0 cm] (0, 6 cm) node[anchor=center]  {{\tiny$\l_{\frac{5}{2}}$}};
   \draw[xshift=0 cm] (0, 4 cm) node[anchor=center]  {{\tiny$\l_{\frac{a}{2}-1}$}};
    \draw[xshift=0 cm] (0, 0 cm) node[anchor=center]  {{\tiny$\l_{\frac{a}{2}+1}$}};
     \draw[xshift=0 cm] (0, -2 cm) node[anchor=center]  {{\tiny$\l_{\frac{a}{2}+2}$}};
         \draw[xshift=0 cm] (0, -4 cm) node[anchor=center]  {{\tiny$\l_{\frac{3a}{2}-2}$}};
    \draw[xshift=0 cm] (0, -6 cm) node[anchor=center]  {{\tiny$\l_{\frac{3a}{2}-1}$}};
      \draw[xshift=0 cm] (0, -10 cm) node[anchor=center]  {{\tiny$\l_{\frac{3a}{2}+1}$}};
         \draw[xshift=0 cm] (0, -12 cm) node[anchor=center]  {{\tiny$\l_{\frac{3a}{2}+2}$}};
    \draw[xshift=0 cm] (0, -14 cm) node[anchor=center]  {{\tiny$\l_{\frac{3a}{2}+3}$}};
     \draw[xshift=0 cm] (0, -16 cm) node[anchor=center]  {{\tiny$\l_{\frac{3a}{2}+4}$}};
      \draw[xshift=0 cm] (0, -18 cm) node[anchor=center]  {{\tiny$\l_{\frac{3a}{2}+5}$}};
            
  \draw[xshift=5 cm,thick,fill=black] (5, 12 cm) circle (.1cm);
    \draw[xshift=5 cm,thick,fill=black] (5, 10 cm) circle (.1cm);
     \draw[xshift=5 cm,thick, fill=black] (5, 8 cm) circle (.1cm);
      \draw[xshift=5 cm,thick,fill=black] (5, 6 cm) circle (.1cm);
        \draw[xshift=5 cm,thick, fill=black] (5, 4 cm) circle (.1cm);
         \draw[xshift=5 cm,thick,fill=black] (5, 2 cm) circle (.1cm);
                \draw[xshift=5 cm,thick,black] (5, 0 cm) circle (.2cm);
                                \draw[xshift=5 cm,thick,fill=black] (5, -2 cm) circle (.1cm);
                \draw[xshift=5 cm,thick,fill=black] (5, -4 cm) circle (.1cm);
                \draw[xshift=5 cm,thick,fill=black] (5, -6 cm) circle (.1cm);
                \draw[xshift=5 cm,thick,fill=black] (5, -8 cm) circle (.1cm);

                \draw[xshift=5 cm,thick,fill=black] (5, -10 cm) circle (.1cm);
                \draw[xshift=5 cm,thick,fill=black] (5, -12 cm) circle (.1cm);
                \draw[xshift=5 cm,thick,black] (5, -14 cm) circle (.2cm);
                \draw[xshift=5 cm,thick,fill=black] (5, -16 cm) circle (.1cm);
                \draw[xshift=5 cm,thick,fill=black] (5, -18 cm) circle (.1cm);

   \draw[xshift=6 cm] (6, 12 cm) node[anchor=center]  {{\tiny $\m_{-1/2}$}};
    \draw[xshift=6 cm] (6, 10 cm) node[anchor=center]  {{\tiny $\m_{1/2}$}};
        \draw[xshift=6 cm] (6, 8 cm) node[anchor=center]  {{\tiny $\m_{3/2}$}};
     \draw[xshift=6 cm] (6, 6 cm) node[anchor=center]  {{\tiny $\m_{5/2}$}};
         \draw[xshift=6 cm] (6, 4 cm) node[anchor=center]  {{\tiny $\m_{a/2-1}$}};
      \draw[xshift=6 cm] (6, 2 cm) node[anchor=center]  {{\tiny $\m_{a/2}$}};
            \draw[xshift=6 cm] (6, 0 cm) node[anchor=center]  {$$};
                  \draw[xshift=6 cm] (6, -2 cm) node[anchor=center]  {{\tiny $\m_{a/2+2}$}};
      \draw[xshift=6 cm] (6, -4 cm) node[anchor=center]  {{\tiny $\m_{a/2+3}$}};
      \draw[xshift=6 cm] (6, -6 cm) node[anchor=center]  {{\tiny $\m_{\frac{3a}{2}-1}$}};
      \draw[xshift=6 cm] (6, -8 cm) node[anchor=center]  {{\tiny $\m_{\frac{3a}{2}}$}};
      \draw[xshift=6 cm] (6, -10 cm) node[anchor=center]  {{\tiny $\m_{\frac{3a}{2}+1}$}};
      \draw[xshift=6 cm] (6, -12 cm) node[anchor=center]  {{\tiny $\m_{\frac{3a}{2}+2}$}};
      \draw[xshift=6 cm] (6, -16 cm) node[anchor=center]  {{\tiny $\m_{\frac{3a}{2}+4}$}};
      \draw[xshift=6 cm] (6, -18 cm) node[anchor=center]  {{\tiny $\m_{\frac{3a}{2}+5}$}};

       \foreach \y in {-3,-1}
    \draw[-stealth, xshift=1 cm,thick] (1, 2*\y cm) -- +(0, 1.85 cm);  
       \foreach \y in {-9,...,-6}
    \draw[-stealth, xshift=1 cm,thick] (1, 2*\y cm) -- +(0, 1.85 cm);            
       \foreach \y in {3,...,4}
    \draw[-stealth, xshift=1 cm,thick] (1, 2*\y cm) -- +(0, 1.85 cm);
     \foreach \y in {-2,2}
    \draw[-stealth, xshift=1 cm,dotted,thick] (1, 2*\y cm) -- +(0, 1.85 cm);
    
          \foreach \y in {-2,-5,-6,-4,-9}
    \draw[-stealth, xshift=9 cm,thick] (1, 2*\y cm) -- +(0, 1.85 cm);
     \foreach \y in {-3}
    \draw[-stealth, xshift=9 cm,dotted,thick] (1, 2*\y cm) -- +(0, 1.85 cm);

    \path [-stealth, solid,black,thick,draw=none] (2, 8 cm) edge[bend left] (2, 11.85 cm);
    \path [-stealth, solid,black,thick,draw=none]  (10, 8 cm) edge (10, 9.85 cm);
    \path [-stealth, solid,black,thick,draw=none]  (10, -2 cm) edge[bend right] (10, 1.85 cm);
    \path [-stealth, solid,black,thick,draw=none]  (10, 8 cm) edge[bend right] (10, 11.85 cm);
        \path [-stealth, solid,black,thick,draw=none]  (10, -16 cm) edge[bend right] (10, -11.85 cm);

        \path [-stealth, solid,black,thick,draw=none]  (2, 0 cm) edge[bend left] (2, 3.85 cm);
        \path [-stealth, solid,black,thick,draw=none]  (2, -10 cm) edge[bend left] (2, -5.85 cm);

    \path [-stealth, solid,black,thick,draw=none]  (10, 2 cm) edge (10, 3.85 cm);
        \path [-stealth, solid,black,dotted,thick,draw=none]  (10, 4 cm) edge (10, 5.85 cm);
    \path [-stealth, solid,black,thick,draw=none]  (10, 6 cm) edge (10, 7.85 cm);
    \path [-stealth, solid,black,thick,draw=none]  (10, 2 cm) edge (10, 3.85 cm);

     \path [-stealth, solid,red,thick,draw=none]  (2.15, 12 cm) edge node[above=0mm]{\tiny $+\e_1+2\e_2$} (9.85, 12 cm) ;
    
   \path [-stealth, solid,red,thick,draw=none]  (2.15, 10 cm) edge node[above=0]{\tiny $+\e_1+2\e_2$} (9.85, 10 cm) ;

    \path [-stealth, solid,red,thick,draw=none]  (2.15, 8 cm) edge node[above=0mm]{\tiny $+\e_1+2\e_2$} (9.85, 8 cm) ;
    
    \path [-stealth,solid,red,thick,draw=none]  (2.15, 6 cm) edge node[draw=none,rectangle,above=0mm]{\tiny $+\e_1+2\e_2$} (9.85, 6 cm) ;
        \path [-stealth,solid,red,thick,draw=none]  (2.15, 4 cm) edge node[draw=none,rectangle,above=0mm]{\tiny $+\e_1+2\e_2$} (9.85, 4 cm) ;
           \path [-stealth,solid,red,thick,draw=none]  (2.15, 0 cm) edge node[draw=none,rectangle,above=0mm]{\tiny $+\e_1+\e_2-\d$} (9.85, 2 cm) ;
            \path [-stealth,solid,red,thick,draw=none]  (2.15, -2 cm) edge node[draw=none,rectangle,above=0mm]{\tiny $2\e_1+\e_2$} (9.85, -2 cm) ;
            \path [-stealth,solid,red,thick,draw=none]  (2.15, -10 cm) edge node[draw=none,rectangle,above=0mm]{\tiny $+\e_1-\d$} (9.85, -8 cm) ;
            
 \path [-stealth,solid,red,thick,draw=none]  (2.15, -12 cm) edge node[draw=none,rectangle,above=0mm]{\tiny $-\d$} (9.85, -10 cm) ;
     \path [-stealth,solid,red,thick,draw=none]  (2.15, -14 cm) edge node[draw=none,rectangle,above=0mm]{\tiny $-\e_1-\d$} (9.85, -12 cm) ;
      \path [-stealth,solid,red,thick,draw=none]  (2.15, -16 cm) edge node[draw=none,rectangle,above=0mm]{\tiny $-\e_1+\e_2$} (9.85, -16 cm) ;
        \path [-stealth,solid,red,thick,draw=none]  (2.15, -18 cm) edge node[draw=none,rectangle,above=0mm]{\tiny $-\e_1+\e_2$} (9.85, -18 cm) ;
      \path [-stealth,solid,red,thick,draw=none]  (2.15, -6 cm) edge node[draw=none,rectangle,above=0mm]{\tiny $+2\e_1+\e_2$} (9.85, -6 cm) ;
    
      \path [-stealth, solid,red,dotted,thick,draw=none]  (2.15, -10 cm) edge node[above=0.1mm,draw=none,rectangle]{\tiny $+\e_1+\e_2$} (9.85, -10 cm) ;
    
     \path [-stealth,solid,red,thick,dotted,draw=none]  (2.15, -12 cm) edge node[draw=none,rectangle,above=0mm]{\tiny $+\e_2$} (9.85, -12 cm) ;
    
    \path [solid,red,thick,dotted,draw=none,above=0mm]  (2.15, -20 cm) edge (10, -20 cm) ;

  \end{tikzpicture}
\end{center}

\begin {theorem} The categories  $\F^{a}$ and $\F^{a+2}$ are equivalent for all $a\geq 1$. 
\end {theorem}
\begin{proof} As in the case for $F(4)$ in \rlem{aa1x}, we have $T(L_{\l_i})=L_{\l_i+\g}$ for $i\neq \frac{3}{2}a+1$ or $\frac{3}{2}+2$. Also as in \rlem{eqaa1}, if for each $\l_\in F^a$, $T(L_\l)$ is a simple module in $\F^{a+2}$, then we can see that the categories $\F^a$ and $\F^{a+1}$ are equivalent. 

Thus, it remains to prove that $T(L_{\l_{\frac{3}{2}a+1}})$ and $T(L_{\l_{\frac{3}{2}a+2}})$ are simple modules. In fact we have $T(L_{\l_{\frac{3}{2}a+1}})=L_{\l_{\frac{3}{2}a}}$ and $T(L_{\l_{\frac{3}{2}a+2}})=L_{\l_{\frac{3}{2}a+1}}$ as it is shown in the picture above. To prove this we use a complicated in induction in $a$ similar to the \rlem{aa1}. We assume the blocks $\F^c$ are all equivalent for $c\leq a$. Using this assumption we prove the equivalence of $\F^{a}$ and $\F^{a+2}$ by first showing that $T(L_{\l_{\frac{3}{2}a+2}})=L_{\l_{\frac{3}{2}a+1}}$, and using it showing that $T(L_{\l_{\frac{3}{2}a+1}})=L_{\l_{\frac{3}{2}a}}$ as in \rlem{aa1}. The theorem follows.
\end{proof}

\section{Characters and superdimension}
The following lemma summarizes some results from sections 5-8 on the multiplicities of simple modules $L_{\m}$ in the cohomology groups $\G_i(G/B, \O_{\l})=H^i(G/B, \O_{\l}^*)^*$. It is used to prove some of the main results in this thesis. Recall that $\l_0, \l_1,\l_2$ are the special weights defined above. 

\begin{lemma}\label{exact} For all simple modules $L_\l\in \F^{(a,b)}$ (or $\F^a$) such that $\l\neq \l_0, \l_1,\l_2$, there is a unique dominant weight $\m\in F^{(a,b)}$ (or $F^a$) with $\m=\l-\sum_{i=1}^n\a_i$ with $\a_i\in\D^+_\1$ and $n\in\{ 1,2,3,4\}$ such that we have an exact sequence: 

\begin{displaymath}
    \xymatrix{
        0\ar[r] & L_{\l} \ar[r] & \G_0(G/B, \mathcal{O}_{\l})\ar[r] & L_{\m} \ar[r] & 0 }
\end{displaymath}

We also have $\G_i(G/B, \mathcal{O}_{\l})=0$ for $i>0$. 




\end{lemma}

\subsection{Superdimension formulae}

We denote $s(\l):=p(\l)$ if $\l=\l_c$ with $c\in I_i$ or $J_i$ with $i=1,3,6,8$. And $s(\l):=p(\l)+1$ if $\l=\l_c$ with $c\in I_i$ or $J_i$ with $i=2,4,5,7$.

\begin{theorem}\label{SFF} Let $\g=F(4)$. 
Let $\l\in F^{(a,b)}$ and $\m+\r_l=a\w_1+b\w_2$. If $\l \neq \l_1,\,\l_2$, the following superdimension formula holds:

\begin{equation}
sdim\,{L_{\l}}=(-1)^{s(\l)} 2dim\,{L_\m(\g_x)}.
\end{equation}

For the special weights, we have: 
\begin{equation}
sdim\,{L_{\l_1}}=sdim\,{L_{\l_2}}=dim\,{L_\m(\g_x)}. 
\end{equation}

\end{theorem}

\begin{proof} For generic weight, the theorem follows from \rthm{character formula generic weights}. For other cases, if $\l\neq\l_0,\l_1,\l_2$ we have

$$0=sdim {H^0(G/B, \mathcal{O}_{\l}^*)^*}=sdim{L_{\l}}+[H^0(G/B, \mathcal{O}_{\l}^*)^*:L_{\m}]sdim{L_{\m}},$$ where $\m$ is the unique dominant weight in \rlem{exact}. This gives $$sdim{L_{\l}}=-sdim{L_{\m}}.$$

From \rlem{exact}, we have $\m=\l-\sum_{i=1}^n\a_i$ with $\a_i\in\D^+_\1$ and $n\in\{ 1,2,3,4\}$. Thus, if $n$ is even, we have $p(\m)=p(\l)$, thus in those cases the sign changes. This occurs each time the last coordinates of $\m$ and $\l$ belong to adjacent intervals. Thus, the theorem follows.  
\end{proof}

\begin{theorem}\label{SFG} Let $\g=G(3)$. Let $\l\in F^{a}$ and $\m+\r_l=a\w_1$. If $\l\neq \l_1,\,\l_2$, the following superdimension formula holds:

\begin{equation}
sdim\,{L_\l}=(-1)^{s(\l)} 2dim\,{L_\m(\g_x)}.
\end{equation}

For the special weights, we have: 
\begin{equation}
sdim\,{L_{\l_1}}=sdim\,{L_{\l_2}}=dim\,{L_\m(\g_x)}.
\end{equation}

\end{theorem}

\begin{proof} Similar to the proof for $F(4)$. 
\end{proof}

\subsection{Kac-Wakimoto conjecture}

A root $\a$ is called {\it{isotropic}} if $(\a, \a)=0$. The {\it{degree of atypicality}} of the weight $\l$ the maximal number of mutually orthogonal linearly independent isotropic roots $\a$ such that $(\l+\r,\a)=0$. The {\it{defect}} of $\g$ is the maximal number of linearly independent mutually orthogonal isotropic roots. The above theorem proves the following conjecture by Kac-Wakimoto Conjecture for $\g=F(4)$ and $G(3)$, see \cite{KW}.

\begin{theorem}\label{KW} The superdimension of a simple module of highest weight $\l$ is nonzero if and only if the degree of atypicality of the weight is equal to the defect of the Lie superalgebra. 
\end{theorem}

\begin{proof} Follows from \rthm{SFF} and \rthm{SFG}.
\end{proof}

\subsection{Character formulae}

In this section, we prove a Weyl character type formula for the dominant atypical weights. 

\begin{lemma}\label{sigma} For a dominant weight $\l$ and corresponding $\m$ and $n$ from \rlem{exact}, there is a unique $\s\in W$ such that $\l+\r-\s(\m+\r)=n\a$ for $\a\in\D_\1$ satisfying $(\l+\r,\a)=0$. Also, $sign{\s}=(-1)^{n-1}$.

If $\b\in\D_\1$ is such that $(\m+\r,\b)=0$, then $\s(\b)=\a$. 
\end{lemma}

\begin{proof} This follows from \rlem{3}. 
\end{proof}

\begin{theorem}\label{BL} For a dominant weight $\l\neq \l_1,\l_2$, let $\a\in \D_\1$ be such that $(\l+\r,\a)=0$. Then
\begin{equation}\label{BL}
ch{L_{\l}}=\frac{D_1\cdot e^{\r}}{D_0}\cdot\sum_{w\in W} sign(w)\cdot w(\frac{e^{\l+\r}}{(1+e^{-\a})}).
\end{equation}

For $\l=\l_i$ with $i=1,2$, we have the following similar formula: 
\begin{equation}\label{BLS}
ch{L_{\l}}=\frac{D_1\cdot e^{\r}}{2D_0}\cdot\sum_{w\in W} sign{(w)}\cdot w(\frac{e^{\l+\r}(2+e^{-\a})}{(1+e^{-\a})}).
\end{equation}

\end{theorem}

\begin{proof} Let $\m$ be dominant weight, then it corresponds to some $\l$ and $n$ in \rlem{exact} such that we have: 

\begin{displaymath}
    \xymatrix{
        0\ar[r] & L_{\l} \ar[r] & \G_0(G/B, \mathcal{O}_{\l})\ar[r] & L_{\m} \ar[r] & 0 }.
\end{displaymath}

It follows that $ch(\G_0(G/B,\mathcal{O}_{\l}))=ch( L_{\l})+ch( L_{\m})$.

Assume the formula is true for $\l$. We show that this together with \rlem{chcoh}, proves the formula for $\m$. Since we can obtain each dominant weight from generic one by similar correspondence, from \rlem{chgeneric} the formula follows for all dominant weights.

We have from \rlem{sigma}:

$$\frac{e^{\l+\r}}{1+e^{-\a}}+(-1)^{n-1}\s(\frac{e^{\m+\r}}{1+e^{-\b}})=\frac{e^{\l+\r}}{1+e^{-\a}}+(-1)^{n-1}(\frac{e^{\l+\r-n\a}}{1+e^{-\a}})=$$

$$=\frac{e^{\l+\r}(1+(-1)^{n-1}e^{-n\a})}{1+e^{-\a}}=e^{\l+\r}(1+\sum_{i=1}^{n-1} (-1)^ie^{-i\a}).$$





Using the above equation, we have: 

$$ch( L_{\m})=ch(\G_0(G/B,\mathcal{O}_{\l}))-ch( L_{\l})=\frac{D_1\cdot e^{\r}}{D_0}\cdot\sum_{w\in W} sign{w}\cdot w(e^{\l+\r}-\frac{e^{\l+\r}}{(1+e^{-\a})})=$$

$$=\frac{D_1\cdot e^{\r}}{D_0}\cdot\sum_{w\in W} sign{(w\s)}\cdot (w\s)(\frac{e^{\m+\r}}{(1+e^{-\b})}+\sum_{i=1}^n (-1)^i e^{\l+\r-i\a})=$$

$$=\frac{D_1\cdot e^{\r}}{D_0}\cdot\sum_{w\in W} sign{(w)}\cdot w(\frac{e^{\m+\r}}{(1+e^{-\b})})+\frac{D_1\cdot e^{\r}}{D_0}\cdot\sum_{w\in W} sign{(w)}\cdot w(\sum_{i=1}^n (-1)^i e^{\l+\r-i\a}).$$

The second summand is zero as the weights $\l-i\a$ are acyclic. Thus we get the required formula.

Similarly, for $\m=\l_1,\l_2$, we have 

$$ch( L_{\l_1})+ch( L_{\l_2})=ch(\G_0(G/B,\mathcal{O}_{\l_0}))-ch( L_{\l_0})=$$ $$=\frac{D_1\cdot e^{\r}}{D_0}\cdot\sum_{w\in W} sign{w}\cdot w(e^{\l_0+\r}-\frac{e^{\l_0+\r}}{(1+e^{-\a_0})})=$$

$$=\frac{D_1\cdot e^{\r}}{D_0}\cdot\sum_{w\in W} sign{(w\s)}\cdot (w\s)(\frac{e^{\l_1+\r}}{(1+e^{-\a_1})}+\sum_{i=1}^n (-1)^i e^{\l_0+\r-i\a_0})=$$

$$=\frac{D_1\cdot e^{\r}}{D_0}\cdot\sum_{w\in W} sign{(w)}\cdot w(\frac{e^{\m+\r}}{(1+e^{-\b})})+\frac{D_1\cdot e^{\r}}{D_0}\cdot\sum_{w\in W} sign{(w)}\cdot w(\sum_{i=1}^n (-1)^i e^{\l+\r-i\a}).$$

The second summand is zero as the weights $\l+\r-i\a$ are acyclic. Thus we get the required formula for the sum $ch( L_{\l_1})+ch( L_{\l_2})$.

On the other hand, we have

$$ch( L_{\l_1})-ch( L_{\l_2})=ch(\G_0(G/B,\mathcal{O}_{\l_1}))-ch(\G_1(G/B,\mathcal{O}_{\l_1}))=$$ $$=\frac{D_1\cdot e^{\r}}{D_0}\cdot\sum_{w\in W} sign{(w)}\cdot w(e^{\l_1+\r}).$$

Adding both equations above, we get:

$$ch( L_{\l_1})=\frac{D_1\cdot e^{\r}}{2D_0}\cdot\sum_{w\in W} sign{(w)}\cdot w(\frac{e^{\l_1+\r}(2+e^{-\a_1})}{(1+e^{-\a_1})}).$$

The same proof works for the weight $\l_2$. 
\end{proof}

\section{Indecomposable modules}

\subsection{Quivers}

It is easy to see that $\CC$ is a nice category and it is contravariant. For $\m,\m' \neq \l_1,\l_2$, assume $\m$ and $\m'$ are the adjacent vertices of the quiver with $\m>\m'$. From \rlem{exact} and \rlem{3}, we have $$dim\,Ext^1(L_\m,L_{\m'})=[\G_0(G/B, \mathcal{O}_{\m}):L_{\m'}].$$ 

Since the category $\mathcal{C}$ is a contravariant, we also have $$Ext^1(L_{\m'},L_{\m})=Ext^1(L_{\m},L_{\m'}).$$

This describes the quivers and proves the parts 3 and 4 of theorems \rthm{CBF} and \rthm{CBG}.










\subsection{Projective modules}

\begin{lemma}\label{relations} Let $\g=F(4)$ (or $G(3)$). Then the projective indecomposable modules in the block $\F^{(a,b)}$ (or $\F^{a}$) have the following radical layer structure:

If $\l_i\in F^{(a,b)}$ or $\l_i\in F^{(a,a)}$ (or $F^{a}$) with $i=0,1,2$. Then $P_{\l_i}$ has a radical layer structure:

$$L_{\l_i}$$
$$L_{\l_{i-1}}\oplus L_{\l_{i+1}}$$
$$L_{\l_i}$$ where $\l_{i-1}$ and $\l_{i+1}$ are the adjacent vertices of $\l_i$ in the quiver.

If $\l_i\in F^{(a,a)}$ with $i=1,2$. Then $P_{\l_i}$ has a radical layer structure:

$$L_{\l_i}$$
$$L_{\l_0}$$
$$L_{\l_i}$$ 

For $\l_0\in F^{(a,a)}$ (or $F^{a}$), $P_{\l_0}$ has a radical layer structure:

$$L_{\l_0}$$
$$L_{\l_1}\oplus L_{\l_2} \oplus L_{\l_3}$$
$$L_{\l_0}$$ 

\end{lemma}

\begin{proof} For the top radical layer structure, we have: 

$$P_\l/rad\, P_\l=soc\, P_\l\cong L_\l,$$ since projective morphisms in $\CC$ are injective and have a simple socle (see \cite{SerQ}).

Since $rad\, P_{\l_i}/rad^2\, P_{\l_i}$ is the direct sum of simple modules which have a non-split extension by $L_{\l_i}$, for the middle radical layer structure, we have: 

$$rad\, P_{\l_i}/rad^2\, P_{\l_i}\cong L_{\l_{i-1}}\oplus L_{\l_{i+1}}.$$  

We also have:

 $Ext^1(L_{\l_i},L_{\l_{i-1}})\neq 0$, $Ext^1(L_{\l_i},L_{\l_{i+1}})\neq 0$, and $Ext^1(L_{\l_i},L_{\s})\neq 0$ for $\s\neq\l_{i-1},\l_{i+1}$.

By BGG reciprocity from \cite{BGGSeGr}, we have $$[P_\l:L_\m]=\sum_{\nu}[P_\l:\varepsilon_{\nu}]\cdot[\varepsilon_{\nu}:L_\m]=\sum_{\nu}[\varepsilon_{\nu}:L_\l]\cdot[\varepsilon_{\nu}:L_\m].$$

Thus, \[ [P_\l:L_\m] = \left\{ \begin{array}{ll}
         2 & \mbox{if $\m=\l$};\\
       1 & \mbox{if $\m$ is adjacent to $\l$}.\end{array} \right. \] 

This implies that there are only three radical layers. Therefore, for the bottom radical layer structure, we have: 
$$ rad^2\, P_\l\cong L_\l.$$ 
\end{proof}

\subsection{Germoni's conjecture and the indecomposable modules}

From \rthm{CBF}, \rthm{CBG}, and \rlem{relations}, we will have \rthm{tame}. From \rthm{tame}, we obtain the proof of \rthm{germoni1}, since all the blocks for $F(4)$ and $G(3)$ are of atypicality less or equal 1. 

For $\F^{(a,b)}$ and for $\F^{(a,a)}$, $\F^{a}$ if $l\geq 3$, we let $d^+_l$ denote the arrow from vertex with weight $\l_l$ to the adjacent vertex $\l'_l$ on the left in the quiver. And let $d^-_l$ denote the arrow in the opposite direction. These arrows correspond to the irreducible morphisms $D^{\pm}_{\l_l}$ from $P_{\l'_l}$ to $P_{\l_l}$. 

For $\F^{(a,a)}$, $\F^{a}$, also let $d^+_0$ denote the arrow from vertex $\l_0$ to $\l_3$ and $d^-_0$ the arrow in the opposite direction. Similarly, for $i=1,2$, denote by $d^+_i$ the arrow from vertex $\l_1$ to $\l_0$ and $d^-_i$ the arrow in the opposite direction.

The relations in \rthm{relations} follow by computations in \cite{G1} or \cite{Gruson}, since the radical filtrations of projectives are the same. Using quivers in \rthm{CBF} and \rthm{CBG}, and \rlem{relations}, we obtain \rthm{relations}. The \rthm{relations} with quiver theorem in \cite{GR} gives a description of the indecomposable modules. 

\bibliographystyle{plain}
\bibliography{fg}

\begin{thebibliography}{10}

\bibitem{DuSe}
Michel Duflo and Vera Serganova.
\newblock {On Associated Variety for Lie superalgebras}.
\newblock {\em ARXIV}, page 21 pages, 2005.

\bibitem{GR}
P.~Gabriel and A.~V. Ro{\u\i}ter.
\newblock Representations of finite-dimensional algebras.
\newblock In {\em Algebra, {VIII}}, volume~73 of {\em Encyclopaedia Math.
  Sci.}, pages 1--177. Springer, Berlin, 1992.
\newblock With a chapter by B. Keller.

\bibitem{G1}
J{\'e}r{\^o}me Germoni.
\newblock Indecomposable representations of special linear lie superalgebras.
\newblock {\em J. Algebra}, 209(2):367--401, 1998.

\bibitem{Gruson}
Caroline Gruson.
\newblock Cohomologie des modules de dimension finie sur la super alg\`ebre de
  {L}ie osp(3,2).
\newblock {\em J. Algebra}, 259(2):581--598, 2003.

\bibitem{BGGSeGr}
Caroline Gruson and Vera Serganova.
\newblock {B}ernstein-{G}el'fand-{G}el'fand reciprocity and indecomposable
  projective modules for classical algebraic supergroups.
\newblock {\em Moscow Mathematical Journal, to appear}.

\bibitem{GrSe}
Caroline Gruson and Vera Serganova.
\newblock Cohomology of generalized supergrassmannians and character formulae
  for basic classical {L}ie superalgebras.
\newblock {\em Proc. Lond. Math. Soc. (3)}, 101(3):852--892, 2010.

\bibitem{K2}
V.~Kac.
\newblock Representations of classical {L}ie superalgebras.
\newblock In {\em Differential geometrical methods in mathematical physics,
  {II} ({P}roc. {C}onf., {U}niv. {B}onn, {B}onn, 1977)}, volume 676 of {\em
  Lecture Notes in Math.}, pages 597--626. Springer, Berlin, 1978.

\bibitem{K1}
V.~G. Kac.
\newblock Classification of simple {L}ie superalgebras.
\newblock {\em Funkcional. Anal. i Prilo\v zen.}, 9(3):91--92, 1975.

\bibitem{K}
V.~G. Kac.
\newblock Lie superalgebras.
\newblock {\em Advances in Math.}, 26(1):8--96, 1977.

\bibitem{KW}
Victor~G. Kac and Minoru Wakimoto.
\newblock Integrable highest weight modules over affine superalgebras and
  number theory.
\newblock In {\em Lie theory and geometry}, volume 123 of {\em Progr. Math.},
  pages 415--456. Birkh\"auser Boston, Boston, MA, 1994.

\bibitem{Musson}
Ian~M. Musson.
\newblock {\em Lie superalgebras and enveloping algebras}, volume 131 of {\em
  Graduate Studies in Mathematics}.
\newblock American Mathematical Society, Providence, RI, 2012.

\bibitem{P}
I.~B. Penkov.
\newblock Borel-{W}eil-{B}ott theory for classical {L}ie supergroups.
\newblock In {\em Current problems in mathematics. {N}ewest results, {V}ol.\
  32}, Itogi Nauki i Tekhniki, pages 71--124. Akad. Nauk SSSR Vsesoyuz. Inst.
  Nauchn. i Tekhn. Inform., Moscow, 1988.
\newblock Translated in J. Soviet Math. {{$\bf{5}$}1} (1990), no. 1,
  2108--2140.

\bibitem{Penkov}
Ivan Penkov.
\newblock Characters of strongly generic irreducible {L}ie superalgebra
  representations.
\newblock {\em Internat. J. Math.}, 9(3):331--366, 1998.

\bibitem{S2}
Vera Serganova.
\newblock Kazhdan-{L}usztig polynomials for {L}ie superalgebra
  {${\mathfrak{gl}(m\vert n)}$}.
\newblock In {\em I. {M}. {G}el'fand {S}eminar}, volume~16 of {\em Adv. Soviet
  Math.}, pages 151--165. Amer. Math. Soc., Providence, RI, 1993.

\bibitem{Ser}
Vera Serganova.
\newblock Characters of irreducible representations of simple {L}ie
  superalgebras.
\newblock In {\em Proceedings of the {I}nternational {C}ongress of
  {M}athematicians, {V}ol. {II} ({B}erlin, 1998)}, number Extra Vol. II, pages
  583--593 (electronic), 1998.

\bibitem{S1}
Vera Serganova.
\newblock Kac-{M}oody superalgebras and integrability.
\newblock In {\em Developments and trends in infinite-dimensional {L}ie
  theory}, volume 288 of {\em Progr. Math.}, pages 169--218. Birkh\"auser
  Boston Inc., Boston, MA, 2011.

\bibitem{S3}
Vera Serganova.
\newblock On the superdimension of an irreducible representation of a basic
  classical {L}ie superalgebra.
\newblock In {\em Supersymmetry in mathematics and physics}, volume 2027 of
  {\em Lecture Notes in Math.}, pages 253--273. Springer, Heidelberg, 2011.

\bibitem{SerQ}
Vera Serganova.
\newblock Quasireductive supergroups.
\newblock In {\em New developments in {L}ie theory and its applications},
  volume 544 of {\em Contemp. Math.}, pages 141--159. Amer. Math. Soc.,
  Providence, RI, 2011.

\end{thebibliography}

\end{document}